\documentclass[11pt,final,
]{article}
\usepackage{epsfig}
\usepackage{amsfonts}
\usepackage{natbib}
 \bibpunct[, ]{(}{)}{,}{a}{}{,}%

\usepackage[margin=1in]{geometry} 

\usepackage{setspace}
\setdisplayskipstretch{0.85}

\newcommand{\ls}[1]
  {\dimen0=\fontdimen6\the\font \lineskip=#1\dimen0
  \advance\lineskip.5\fontdimen5\the\font \advance\lineskip-\dimen0
  \lineskiplimit=.9\lineskip \baselineskip=\lineskip
  \advance\baselineskip\dimen0 \normallineskip\lineskip
  \normallineskiplimit\lineskiplimit \normalbaselineskip\baselineskip
  \ignorespaces }

\usepackage{lineno}
\usepackage{amssymb}
\usepackage{amsmath}
\usepackage{amsthm}
\usepackage{latexsym}
\usepackage{graphicx}
\usepackage{bm}
\usepackage{tabularx}
\usepackage{booktabs}
\usepackage{enumitem}

\setlist[itemize]{leftmargin=1.6em,  labelsep=0.4em, topsep=1pt, itemsep=0pt, parsep=0pt, partopsep=0pt}
\setlist[enumerate]{leftmargin=1.6em,  labelsep=0.4em, topsep=1pt, itemsep=0pt, parsep=0pt, partopsep=0pt}

\usepackage{caption}
\usepackage{subcaption}
\usepackage[titletoc]{appendix}

\usepackage{mathtools}
\usepackage{multirow}
\usepackage{authblk}
\newcommand{\field}[1]{\mathbb{#1}}
\def\N{\field{N}}
\def\A{\field{A}}                                
\def\X{\field{X}}                                

\def\Z{\field{Z}}

\catcode`\@=11

\numberwithin{equation}{section}
\numberwithin{table}{section}
\numberwithin{figure}{section}

\theoremstyle{plain}
\newtheorem{theorem}{Theorem}[section]
\newtheorem{lemma}[theorem]{Lemma}
\newtheorem{corollary}[theorem]{Corollary}
\newtheorem{proposition}[theorem]{Proposition}

\theoremstyle{definition}
\newtheorem{definition}{Definition}[section]
\newtheorem{assumption}[definition]{Assumption}

\theoremstyle{remark}

\graphicspath{{./}{Figs/}}

\usepackage[colorlinks=true,linkcolor=magenta,citecolor=blue,pagebackref]{hyperref}

\title{Telehealth Control Policies: Bridging the Gap Between Patients and Doctors}
\author[1]{Shuwen Lu}
\author[2]{Mark E. Lewis}
\author[3] {Jamol Pender}

\affil[1]{\small Department of Systems Engineering, Cornell University, Ithaca, NY 14850, USA {\tt\small sl3243@cornell.edu}}%
\affil[2,3]{\small School of Operations Research and Information Engineering,
        Cornell University, Ithaca, NY 14850, USA
        {\tt\small mark.lewis@cornell.edu, jjp274@cornell.edu}}

\date{}

\begin{document}
\maketitle
\setlength{\abovedisplayskip}{0pt plus 2pt minus 2pt}
\setlength{\belowdisplayskip}{0pt plus 2pt minus 2pt}
\setlength{\abovedisplayshortskip}{0pt plus 2pt minus 2pt}
\setlength{\belowdisplayshortskip}{0pt plus 2pt minus 2pt}
\ls{1.56}
\begin{abstract}
This paper studies a sequential decision-making problem in a two-stage queueing system modeled after operations in CVS MinuteClinics, where nurse practitioners (NPs) oversee patient care throughout the entire visit. All services are non-preemptive, and NPs cannot begin treating a new patient until the current patient has completed both stages of care. Following an initial diagnosis in the upstream phase, NPs must decide for low-acuity patients whether to proceed with treatment independently through immediate service, or to collaborate with a dedicated general physician (GP) via telemedicine. While collaboration typically improves service quality and is preferred by individual patients, it may introduce delays as the NP-patient pair waits for a GP to become available. 

This work explores the structural properties of optimal policies under different system parameters, with a focus on large initial upstream queues, revealing unconventional and complex policy behaviors. Leveraging these structural insights and supporting theoretical results, we design simple and effective heuristics that are computable in linear time and suitable for practical implementation. These heuristics are robust across the entire parameter space of interest, and offer clear, actionable guidance for NPs as system parameters vary. They also achieve near-optimal performance, averaging within 0.1\% of the optimal, while commonly used benchmark policies are highly sensitive to parameter shifts and can incur costs more than 100\% higher than optimal. 
The work provides applicable insights for decision-makers on improving policy robustness and effectiveness, as well as recommendations for stakeholders on the value of investing in telemedicine infrastructure. For instance, we identify scenarios where such investments may be either unnecessary or essential based on specific system parameters.

\end{abstract}

\section{Introduction}
Telemedicine has gradually expanded with advances in digital technology, offering greater convenience, time efficiency, and cost-effectiveness for both patients and medical service providers (MSPs) \citep{calton2019top}. Through our corporate partnership with CVS Telemedicine, we observe that CVS operates a nationwide network of readily accessible brick-and-mortar locations, many of which are equipped with the IT infrastructure to support on-site telemedicine visits. 
Historically, broad adoption of telemedicine was constrained by regulatory inconsistency and service-quality concerns. The COVID-19 pandemic, however, accelerated the shift to virtual appointments, prompting regulatory relaxation and expanded training of MSPs to deliver quality care across state lines \citep{loeb2020departmental, contreras2020telemedicine, calton2020telemedicine}. Under these more flexible conditions, telemedicine has demonstrated potential to reduce patient waiting and improve efficiency by lowering physician idle time, helping to mitigate ongoing labor shortages \citep{lu2024physician}.

This work develops a model that captures the operational benefits of telemedicine while addressing service-quality considerations in decision-making, with particular emphasis on settings with a large initial queue. 
In CVS MinuteClinics, patient care proceeds in two stages as seen by the illustration given in Figure \ref{fig:CVS_queue}. Patients first enter the (upstream) \textit{triage} phase, in which a nurse practitioner (NP) conducts an initial diagnosis and remains with the patient throughout the visit. Based on the assessment, the NP determines the appropriate second-stage service (downstream), specifically whether collaboration with a general physician (GP) is necessary. Higher-acuity patients require collaborative care, whereas for low-acuity cases, physician involvement is optional, and the NP must decide whether to proceed independently or seek collaboration. Although collaboration enhances service quality, it introduces the risk of additional delay, as the NP–patient pair may need to join a second queue to wait for an available GP, who are typically fewer in number than NPs.

Based on practical operational needs, we incorporate two key modeling requirements. First, the decisions must remain impartial to the preferences of individual NPs and patients, which in practice may vary with NP experience or patient attitudes and may favor collaboration even when it downgrades system performance. 
Independent care begins immediately but delivers lower quality without GP involvement, whereas collaboration can delay the downstream service while awaiting GP availability, creating compound blocking effects that can increase waiting for subsequent patients.
Second, the question of when to call a GP is of interest particularly in the aftermath of demand spikes when the clinic attempts to return to its normative workload. Such demand surges occur frequently, including lunch and weekend peaks, post-holiday and start-of-week backlogs, and seasonal outbreaks of flu, allergies, or COVID-19. Prior research shows that high nursing workloads increase stress and burnout among NPs and registered nurses (RNs) \citep{greenglass2003reactions} and can compromise patient safety \citep{lang2004nurse}.

To meet these practical requirements, we model the sequential decision-making process as a Markov decision process (MDP) clearing system. NPs observe the system states through an information technology implementation, including the number of patients upstream and those in the continued downstream phase (both with and without a collaborating GP). To reflect the relative importance of patients having to wait at different parts of the system, we assign different costs per patient per unit time in each station of service, referred to as \textit{holding costs} in alignment with standard queueing theory terminology.
To capture recovery from demand surges, we adopt a clearing system model where the “zero state” represents the normative workload level and seek a control policy that minimizes the total cost incurred until the system empties.

In developing control policies, it is important to recognize that MSPs must prioritize patient care over operational concerns. This motivates the need for practical, non-prescriptive policies that naturally integrate with clinical workflows. At the same time, the problem size increases rapidly with the number of NPs, referred to as \textit{servers} in queueing literature, making exact optimal solutions computationally expensive or even infeasible, particularly when the system starts with a large initial upstream queue. Moreover, healthcare settings often demand real-time decisions following initial consultations, especially in the presence of time-sensitive risks. These challenges underscore the importance of simple, implementable heuristics. 

By analyzing the structural properties of optimal policies across the entire parameter space, we are able to develop effective and robust heuristics to support real-time decision-making. Numerical experiments demonstrate that our heuristic achieves near-optimal performance, averaging within 0.1\% of the optimal cost, while substantially outperforming existing parameter-sensitive policies, which can incur costs exceeding the optimum by over 100\%. In addition to its strong quantitative performance, the heuristic provides actionable guidance for NPs in adapting to changes in the system parameters. For instance, as NPs receive more training and their independent service rates improve relative to collaborative care, the recommended actions should naturally adjust to reflect these efficiency gains.

Importantly, our insights are intended not only to support frontline decision-makers, but also to inform broader system design and strategic investment. 
In particular, our framework identifies conditions under which investment in telemedicine-enabling technologies can lead to meaningful improvements in system performance. For example, in some regions of the parameter space, the optimal strategy involves seeking physician collaboration only when necessary, while in others, collaboration may always be called. Understanding where these thresholds lie enables stakeholders to assess whether expanding telemedicine capabilities is essential or potentially inefficient, thereby supporting evidence-based decisions on infrastructure and resource allocation.

The remainder of the paper is organized as follows. Section \ref{sec:manage_insights} presents key managerial insights, followed by a review of relevant literature in Section \ref{sec:lit_review}. Section \ref{sec:model} formally describes the model, and Sections \ref{sec:main_results} and \ref{sec:support_results} present the main theoretical results and supporting analyses, respectively. Section \ref{sec:heuristics} outlines the construction of our proposed heuristics, and Section \ref{sec:numerical} reports the results of our numerical study. Finally, Section \ref{sec:conclusions} concludes the paper. Proofs of the main theorems, supporting results, and details related to the heuristic design are provided in Appendix \ref{sec:appendix}.

\subsection{Managerial insights} \label{sec:manage_insights}
Our work investigates a two-stage controlled queueing system in which, following an initial screening, a decision is made between two downstream service options: one that offers immediate processing, and another that delivers a more comprehensive outcome with higher service quality, but relies on limited, reusable resources that potentially introduce delays. This framework yields several managerial insights relevant to practice:

\begin{itemize}
    \item \textbf{Model Generality.} 
    The framework applies to a broad class of systems where the second option in the downstream relies on constrained resources. Beyond healthcare settings involving limited access to specialists, similar structures appear in customer service or technical support, where cases may be escalated to higher-tier agents or domain experts. In manufacturing and logistics, this includes routing jobs to specialized machinery or capacity-limited facilities that restrict concurrent processing.

    The concept of service quality in this context extends beyond traditional measures such as technical accuracy or performance. It also encompasses dimensions such as customer satisfaction, gains in operational efficiency through automation, and the convenience enabled by digital or computer-assisted processes.
    \item \textbf{System Recovery and Job Completion.}
    We study control policies that guide the system back to a normative state following periods of elevated demand. This analysis is especially valuable for businesses facing demand volatility, offering insights into how to anticipate congestion and allocate limited resources more effectively.
    \item \textbf{System-Level Optimality Over Local Preferences.} 
    Our model emphasizes optimizing system-wide performance rather than catering to the immediate preferences of individual jobs or servers. This perspective equips managers to evaluate when local, short-term preferences may conflict with broader operational goals.
    
    \item \textbf{Policy Insights Under Heavy Initial Load.}
    We examine the structural properties of optimal policies with a focus on large initial upstream queues. 
    Our analysis identifies parameter regions where collaborative care is preferable when faced with heavy initial workloads, offering guidance for effective response strategies in high-demand scenarios.
    
    \item \textbf{Scalable, Robust, and Accurate Heuristics.}
    We design simple yet effective heuristics that apply across the entire parameter space and can be implemented in linear time. In particular, we develop a threshold-type policy with the associated threshold computable in constant time, enabling fast and practical decision-making. This efficiency is especially valuable when exact solutions or learning-based methods are computationally infeasible. In numerical experiments, the heuristics demonstrate strong robustness and near-optimal performance, consistently outperforming existing benchmark policies.
    
    \item \textbf{Practical Guidance for Decision-Makers and Stakeholders.}
    Our results offer decision-makers actionable strategies for adjusting policies in response to parameter changes and provide managers with guidance on dynamic resource allocation across multiple locations. We also offer recommendations for stakeholders on when to invest in expanding constrained resources, such as acquiring specialized equipment or hiring expert personnel. For example, in certain parameter regions, collaborative service is consistently preferred, highlighting the need for telemedicine implementation and targeted capacity expansion.

\end{itemize}

\subsection{Literature review} \label{sec:lit_review}
A substantial body of research has investigated optimal policies for controlled queueing systems. These studies span various settings, including tandem queues, parallel queueing systems, and hybrid systems with both tandem and parallel branches \citep{hordijk1992shortest}, such as the well-studied N-network \citep{harrison1998heavy, bell2001dynamic, ahn2004optimal, down2010n}. 
For tandem queueing systems, much of the control literature focuses on server configuration policies, where flexible servers are dynamically allocated to different stages of service \citep{pandelis1994optimal, ahn1999optimal, zayas2016dynamic}. Some models also consider combinations of flexible and dedicated servers \citep{wu2006dynamic, wu2008heuristics, pandelis2008optimal}. In contrast, in parallel queueing structures, control typically refers to routing policies that assign arriving jobs to different servers or queues \citep{hordijk1992assignment, down2006dynamic}. Alternatively, in single-server systems with parallel arrival streams, the control problem involves dynamic scheduling decisions, i.e., selecting which class of job to serve next \citep{nain1989interchange, huang2022dynamically}.

While our model incorporates structural similarities of both tandem and parallel elements, its operational constraints differ significantly from classical formulations. Specifically, NPs accompany patients through both service stages and cannot be reassigned once engaged, eliminating the server-configuration flexibility typically found in tandem models. 
If NPs were instead fixed to a particular stage, the system dynamics could be modeled using a standard Jackson network framework \citep{massey1984operator, massey1986family, massey1987calculating, pender2017approximating}. 
On the parallel side, the decision of whether to seek collaborative care constitutes a routing control. However, the asymmetric roles of MSPs create a unique dynamic: NPs remain tied to individual patients while GPs serve as limited and dedicated downstream resources. Service rates are jointly determined by these personnel constraints, setting our model apart from others. Extending the work in \citep{lu2026balancingindependentcollaborativeservice}, where decisions are made upon the initial job interaction, this two-stage model adds an upstream assessment phase, allowing additional time to evaluate whether to pursue collaboration.

Despite the practical importance of control policies in healthcare, there are relatively few analytical studies in this area. Among the existing work, \citet{dobson2012queueing} model the interaction between attending physicians and residents as a three-stage tandem queue and derive work prioritization policies to maximize throughput in a system with one attending physician and up to two residents. \citet{andradottir2021optimizing} examine a two-stage service system for the same interaction, focusing on how attending physicians allocate time between working with residents and attending to their own tasks. Their model allows an arbitrary number of residents and aims to maximize the long-term reward, with throughput optimization as a special case. More generalized work is presented in \citep{yu2024optimal, niyirora2016optimal}, which analyzes systems with multiple attending physicians (referred to as supervisors), includes customer abandonment, and considers a distinct cost structure.

One of the most prevalent techniques for analyzing optimal policy structures is \textit{sample path arguments} \citep{liu1992optimal, nain1994optimal}. Rather than handling the difficulty of probabilistic analysis, these arguments compare outcomes along individual sample paths of the stochastic process, defined on a common probability space. The formal foundations laid by \citet{liu1995sample} subsequently lead to a surge of applications across various domains. For example, in communication networks and online service systems, \citet{paschalidis2000congestion} use this method to establish monotonicity and concavity properties of the relative reward function, ultimately leading to monotonic pricing policies. In the healthcare domain, particularly emergency departments, \citet{zayas2016dynamic} derive sufficient conditions under which customers with higher rewards are prioritized in a tandem queue with customer impatience. From the perspective of energy efficiency, \citet{badian2021optimal} prove the existence of a work-conserving, monotone optimal policy with respect to the number of jobs in the system. 
Despite their elegance, sample path arguments require structural conclusions to hold along individual sample paths rather than in expectation, whereas value functions are defined in expectations. Moreover, these methods suggest decision structures by revealing only qualitative relationships between value functions, without quantifying the magnitude of differences. These distinctions highlight both the strengths and limitations of the approach.

Another powerful analytical approach is the use of inductive methods. Induction on the decision horizon is commonly employed to establish convexity in single-dimensional problems and submodularity in multi-dimensional ones \citep{koole1998structural}. This approach is particularly natural for finite-horizon problems \citep{chong2018two, yoon2004optimal}, though it can also be extended to infinite-horizon models through limiting arguments \citep{kumar2013dynamic}, where the finite-step value function is shown to converge to its infinite-horizon counterpart \citep{feinberg2007optimality}. However, induction on decision epochs is not well-suited to our setting, as optimal decisions over a short horizon may be myopic when viewed from the perspective of an infinite horizon, thereby obscuring structural regularities.
Instead, a more natural inductive approach in clearing systems is to proceed by the number of jobs or services in the system (or in the queue), as demonstrated in prior work \citep{ahn2004optimal, argon2008scheduling}. This job-based induction enables the characterization of policies that respond to system load rather than time horizon, aligning with the core objective of clearing system models: to efficiently recover from demand surges by emptying the system.

\section{Model Description and Main Results} \label{sec:model}
In our model, we consider a fixed patient population within the MinuteClinic, with no external arrivals. The queueing dynamic is shown in Figure \ref{fig:CVS_queue}. All services are non-preemptive and delivered on a first-come-first-served basis. Let $p$, independent across patients, denote the probability that a patient is classified as high acuity and therefore requires physician involvement. 
We assume that GPs are dedicated to the telemedicine operation. Practically, this corresponds to a physician being assigned to MinuteClinic telehealth consultations during a scheduled shift and remaining on call while performing other duties. From a modeling perspective, this implies that a GP becomes immediately available for the next patient upon completing the current.


\begin{figure*} 
    \centering
    \begin{subfigure}[htbp]{0.45\textwidth}
        \centering
        \includegraphics [scale=.4] {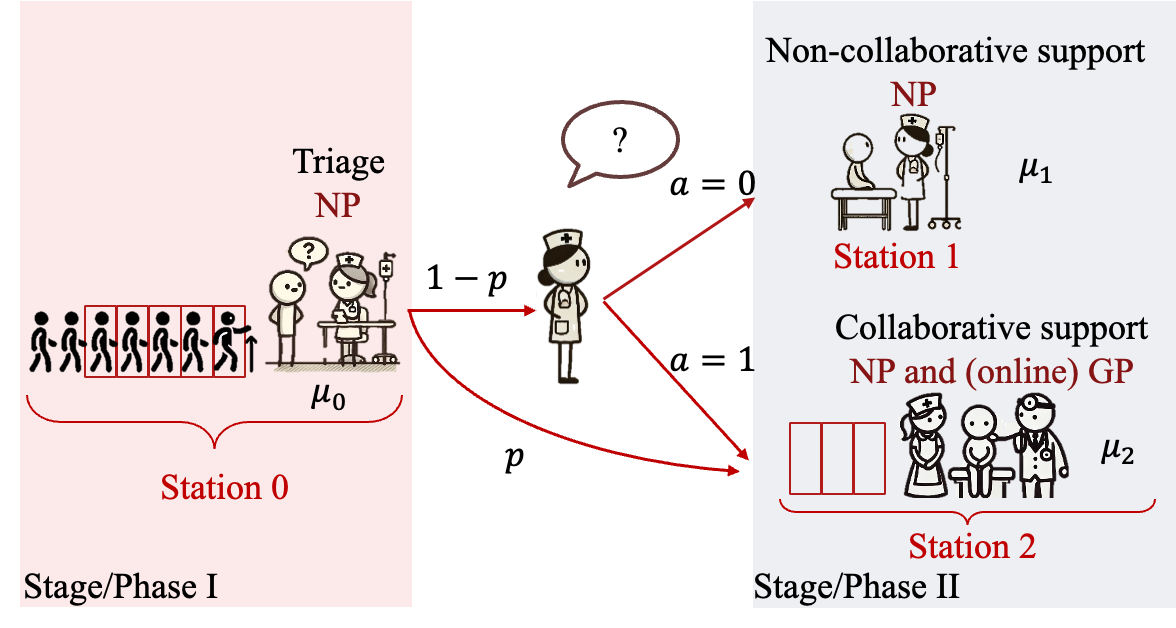}
        \caption{CVS two-stage queueing model.}
        \label{fig:CVS_queue}
    \end{subfigure} 
    \hfill
    \begin{subfigure}[htbp]{0.45\textwidth}
        \centering
        \includegraphics [scale=.4] {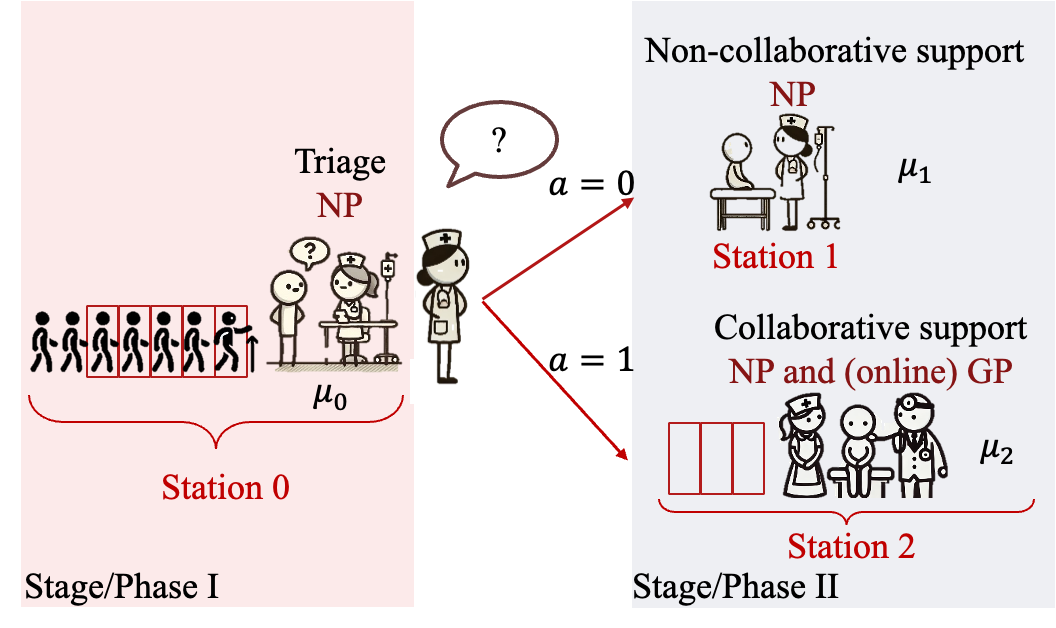}
        \caption{Assume $p=0$ in the model.}
        \label{fig:MDP_queue}
    \end{subfigure} 
    \caption{Diagram of Two-Stage Collaborative Support Model.} \label{fig:CVS}
\end{figure*}


Let $C^p$ and $C^G$ denote the total number of available NPs and GPs at the clinic, respectively, both assumed to be positive finite integers. To reflect typical staffing practices, e.g., in urban areas, we assume $C^G < C^p$ unless otherwise noted; the case $C^G \geq C^p$ is treated separately as a special case.
For analytical convenience, we label the triage, non-collaborative, and collaborative service stations as Stations 0, 1, and 2. Assume service times at Station $n$ are exponentially distributed with rate $\mu_n$, and each patient incurs a holding cost of $h_n$ per unit time at Station $n$, for $n = 0,1,2$.

We formulate the problem as a Markov Decision Process (MDP) clearing system and seek a control policy that minimizes the total expected cost required to empty the system. Considering the operational constraints of the clinic, the state space is defined as
\begin{equation*}
    \X = \Big\{(i,j,k,\ell)\in \Z_+^4\Big{|} i=0, j+k+\ell < C^p \text{ or } i \geq 0, j+k+\ell = C^p\Big\},
\end{equation*}
where
\begin{itemize}
    \item $i$ is the number initially waiting (excluding those in service) at the triage station,
    \item $j$ is the number in service at the triage station,
    \item $k$ is the number of patients currently receiving non-collaborative service, and
    \item $\ell$ is the number of patients at the collaborative station (including those
        in service).
\end{itemize}

A decision is made only when triage is completed, and the patient does not require elevated service from a GP. The set of feasible decision states is therefore defined as:
\begin{align*}
    \X_{D} := \Big\{x \in \X | j \geq 1\}. 
\end{align*} 
The associated action set is given by:
\begin{align*}
    \A(i,j,k,\ell) & = \begin{cases}
        \{0,1\} & \text{if $j \geq 1$,}\\
        \{0\} & \text{otherwise.}
    \end{cases}
\end{align*}
Each allowable action $a \in \A(i,j,k,\ell)$ is a binary indicator:
\begin{itemize}
    \item $a = 0$ indicates that the NP chooses to serve independently (non-collaborative care),
    \item $a = 1$ denotes that a GP is (or will be) engaged for collaborative care.
\end{itemize}  
Each patient faces only one decision point during their visit through the system, namely following the completion of triage if collaboration is not mandated. 

Controlled by a given control policy $\pi$, define $\{Q^{\pi}_{0,0}(t), t\geq 0\}$ and $\{Q^{\pi}_{0,1}(t), t\geq 0\}$ as the stochastic processes denoting the number of patients queueing and in service at the triage station at time $t$, respectively. Let $\{Q^{\pi}_1(t), t\geq 0\}$ and $\{Q^{\pi}_2(t), t\geq 0\}$ represent the number of patients at the non-collaborative and collaborative stations at time $t$, respectively. 
Under the policy $\pi$, the total cost $T^{\pi}$ given the initial state $s=(i,j,k,\ell)$ is
\begin{align}
    T^{\pi}(s) &  = \int_{0}^{\infty}\Big(h_0\big(Q^{\pi}_{0,0}(t) + Q^{\pi}_{0,1}(t)\big) + h_1 Q^{\pi}_1(t) + h_2 Q^{\pi}_2(t)\Big) dt, \label{eq:total_cost_sto}
\end{align}
where $\big(Q^{\pi}_{0,0}(0), Q^{\pi}_{0,1}(0), Q^{\pi}_1(0), Q^{\pi}_2(0)\big) =(i,j,k,\ell) = s$.
Starting with any finite number of patients, $T^{\pi} < \infty$ almost surely under any non-idling policy. The objective is to find a policy that minimizes $\mathbb{E}\big[T^{\pi}(s)\big]$ given any (fixed) initial state $s=(i,j,k,\ell)$. 

Since $p$ is independent of the states, without loss of generality, we assume $p = 0$ in the derivation of structural results. The resulting queueing structure is depicted in Figure \ref{fig:MDP_queue}.
Define the \textit{value function} $v(i,j,k,\ell)$ to be the optimal expected total cost starting in state $(i,j,k,\ell) \in \X$ until the system finishes all existing services so that $v(0,0,0,0) = 0
$. According to the \textit{Principle of Optimality} (c.f.
Section 4.3 in \citep{puterman2014markov}), the value functions satisfy the following optimality equations (c.f. Theorem 7.3.3 in \citep{puterman2014markov} discrete-time MDP) for non-zero states.
\begin{enumerate}
    \item If $i = 0$,
    {\small
    \begin{align}\label{eq:opt-zero}
        \begin{split}
            v(0,j,k,\ell) & = \frac{j h_0+ k h_1 + \ell h_2 }{d(j,k,\ell)} + \Big[ \frac{j \mu_0}{d(j,k,\ell)} \min\{v(0, j-1, k+1, \ell), v(0, j-1, k, \ell+1)\} \\
                & \quad  + \frac{k \mu_1}{d(j,k,\ell)} v(0, j, k-1, \ell)  + \frac{\min\{\ell, C^G\}\mu_2}{d(j,k,\ell)} v(0, j, k, \ell-1)  \Big],
        \end{split}
    \end{align}
    }
    \item If $i \geq 1$,
    {\small
    \begin{align}\label{eq:opt1}
        \begin{split}
            v(i,j,k,\ell) & = \frac{(i+j) h_0+ k h_1 + \ell h_2 }{d(j,k,\ell)} + \Big[ \frac{j \mu_0}{d(j,k,\ell)} \min\{v(i, j-1, k+1, \ell), v(i, j-1, k, \ell+1)\} \\
            & \quad + \frac{k \mu_1}{d(j,k,\ell)} v(i-1, j+1, k-1, \ell) + \frac{\min\{\ell, C^G\}\mu_2}{d(j,k,\ell)} v(i-1, j+1, k, \ell-1) \Big],
        \end{split}
    \end{align}
    }
\end{enumerate}where $d(j,k,\ell) := j\mu_0 + k\mu_1 + \min\{\ell,C^G\}\mu_2$.
Note that $d(j,k,\ell)$ is the overall service rate for state $(i,j,k,\ell)\in \X$.
One observation from the equations above is that the decision of whether to seek collaborative care in state $(i,j,k,\ell) \in \X_D$ depends on the sign of the following value difference:
\begin{align}
    D(i,j,k,\ell) := v(i, j-1, k+1, \ell) - v(i, j-1, k, \ell+1). \label{def:D}
\end{align}
This reflects the difference in future cost between choosing non-collaborative and collaborative care.

In our model, the holding cost parameters $h_n$ are specified by the manufacturer or designer to represent the relative importance per unit of time in each station. A lower value of $h_n$ indicates a greater tolerance for patients spending time at Station $n$. For example, a ratio of $\frac{h_1}{h_2} = 2$ implies that staying in the non-collaborative station is considered twice as costly as in the collaborative one. In contrast, service rate parameters $\mu_n$ capture the speed at which care is delivered and are determined by factors such as provider expertise and digital infrastructure. These rates can typically be measured directly. The relative magnitudes of $\mu_1$ and $\mu_2$ vary across settings. For instance, $\mu_1 > \mu_2$ corresponds to faster service by NPs acting alone, while $\mu_2 > \mu_1$ reflects improved efficiency through collaboration with GPs. 

The ratio $\frac{h_n}{\mu_n}$, $n = 0, 1,2$, provides a measure of service quality (cost-efficiency) at each station, representing the expected cost of completing the service itself. 
Unless stated otherwise, our analysis proceeds under the following assumption, which asserts that collaborative care is more cost-efficient than non-collaborative care, and is therefore preferred from an individual patient’s perspective:
\begin{assumption} \label{assm:preferred_collab}
    $\frac{h_1}{\mu_1} > \frac{h_2}{\mu_2}$.
\end{assumption}
Evaluating the value difference in \eqref{def:D} at a base state yields
\begin{align*}
    D(0,1,0,0) = v(0,0,1,0) - v(0,0,0,1) = \frac{h_1}{\mu_1} - \frac{h_2}{\mu_2}.
\end{align*}
This shows that Assumption \ref{assm:preferred_collab} corresponds exactly to the condition under which collaborative care minimizes the expected cost of a single service.

\subsection{Main results} \label{sec:main_results}
This section presents and proves (over several steps) our main results on the structural properties of optimal policies.
Theorem \ref{thm:lowcost_collab} focuses on the primary setting in which Assumption \ref{assm:preferred_collab} holds, namely that collaborative care provides higher service quality. The corresponding results are summarized in Table \ref{table:policy}.
In contrast, Theorem \ref{prop:noncollab} characterizes conditions under which an optimal policy avoids collaborative service when the inequality in Assumption \ref{assm:preferred_collab} is reversed.
Finally, Theorem \ref{thm:enough_GP} addresses the special case where $C^G \geq C^p$, capturing staffing scenarios typical in rural areas.

\begin{theorem} \label{thm:lowcost_collab}
    Suppose Assumption \ref{assm:preferred_collab} holds and consider $(i,j,k,\ell) \in \X_D$.
    \begin{enumerate}
        \item \label{state:fast_noncollab}
        If $\mu_1 > \mu_2$, there exists a finite threshold $N(j,k,\ell)$ such that an optimal policy does not seek collaborative care for all $i \geq N(j,k,\ell)$.
        \item Suppose $\mu_1 = \mu_2$. \label{state:equal_rate}
        \begin{enumerate} [ref=\theenumi(\alph*)]
            \item \label{state:equal_rate_no_queue_in_collab} 
            If $\ell < C^G$ (i.e., at least one GP is available), an optimal policy seeks collaborative care;
            \item \label{state:equal_rate_queue_in_collab} If $\ell \geq C^G$ (i.e., all GPs are occupied), a similar result holds as when $\mu_1 > \mu_2$.
        \end{enumerate}
        \item Suppose $\mu_1 < \mu_2$. 
        \begin{enumerate} [ref=\theenumi(\alph*)]
            \item \label{state:fast_collab_no_queue_in_collab} 
            If $\ell < C^G$ (i.e., at least one GP is available), an optimal policy seeks collaborative care;
            \item \label{state:fast_collab_queue_in_collab} If $\ell \geq C^G$ (i.e., all GPs are occupied), $\frac{1}{\mu_1} \leq \frac{\ell +1}{C^G\mu_2}$ and $\frac{h_1}{\mu_1} \leq \frac{\ell + 1}{C^G} \frac{h_2}{\mu_2}$, an optimal policy does not seek collaborative care. 
        \end{enumerate}
    \end{enumerate}
\end{theorem}

\begin{proof}
    The proof is provided step by step using the supporting results outlined in Section \ref{sec:support_results}. In particular, Statements \ref{state:fast_noncollab} and \ref{state:equal_rate_queue_in_collab} follow from Proposition \ref{prop:larger_mu1_threshold}. Statements \ref{state:equal_rate_no_queue_in_collab} and \ref{state:fast_collab_no_queue_in_collab} are implied by Proposition \ref{prop:always_collab}. Furthermore, Statement \ref{state:noncollab_queue_in_collab} in Proposition \ref{prop:noncollab} confirms Statement \ref{state:fast_collab_queue_in_collab}.
\end{proof}

\renewcommand{\arraystretch}{1.3}
\begin{table} [htbp]
\caption{An optimal action at state $(i,j,k,\ell) \in X_D$. All statements refer to Theorem \ref{thm:lowcost_collab}.}
\label{table:policy}
\centering
\begin{tabular}{|c|c|c|}
\hline

\multicolumn{1}{|c|}{} & \multicolumn{1}{c|}{\textbf{$\ell< C^G$}} & \multicolumn{1}{c|}{\textbf{$\ell\geq C^G$}} \\ \hline
\multirow{1}{*}{\textbf{$\mu_1 < \mu_2$}}  & \multicolumn{1}{c|}{\multirow{2}{*}{$a=1$ (Statement \ref{state:fast_collab_no_queue_in_collab})}}                           & \multicolumn{1}{c|}{If $\frac{1}{\mu_1} \leq \frac{\ell+1}{C^G\mu_2}$ and $\frac{h_1}{\mu_1} \leq \frac{\ell+1}{C^G}\frac{h_2}{\mu_2}$, $a=0$ (Statement \ref{state:fast_collab_queue_in_collab})}  \\ 
\cline{1-1} \cline{3-3} 
\textbf{$\mu_1 = \mu_2$}    &               & \multicolumn{1}{c|}{$a=0$  for large $i$ (Statement \ref{state:fast_noncollab})}    \\
\hline 
\textbf{$\mu_1 > \mu_2$}    & \multicolumn{2}{c|}{$a=0$ for large $i$ (Statement \ref{state:fast_noncollab})}      \\
\hline
\end{tabular}
\end{table}

\begin{theorem} \label{thm:lowcost_noncollab}
    Suppose Assumption \ref{assm:preferred_collab} does not hold ($\frac{h_1}{\mu_1} \leq \frac{h_2}{\mu_2}$) and that $\frac{1}{\mu_1} \leq \frac{1}{\mu_2}$.
    There exists an optimal policy that does not call for collaborative care in any $(i,j,k,\ell) \in \X_D$.
\end{theorem}

\begin{proof}
    This result follows directly from Statement \ref{state:noncollab_no_queue_in_collab} in Proposition \ref{prop:noncollab}.
\end{proof}

Theorems \ref{thm:lowcost_collab}-\ref{thm:lowcost_noncollab} extend the corresponding results from \citep{lu2026controlpoliciestwostagequeueing} where there was only one GP and two NPs. These generalizations reveal intuitive structural insights and form the foundation for our heuristic design in the next section.

\begin{itemize}
    \item When $\mu_1 >\mu_2$, having the NP conduct the downstream service independently enables a quicker return to triage, thereby reducing the upstream cost incurred by patients waiting in the queue. As the queue length increases, this upstream cost saving dominates the (potentially) higher service cost of non-collaboration (if Assumption \ref{assm:preferred_collab} holds). Moreover, non-collaborative care avoids downstream blocking effects by providing immediate service, making it the preferred option when the initial (upstream) queue is large. This observation aligns with Statement \ref{state:fast_noncollab} in Theorem \ref{thm:lowcost_collab}, indicating that a longer initial queue discourages seeking collaborative care.
    \begin{itemize}
        \item If Assumption \ref{assm:preferred_collab} does not hold, non-collaborative care is not only faster but also incurs a lower or equal service cost. Combined with the absence of downstream blocking, this makes non-collaborative care optimal in all decision-requiring states, as stated in Theorem \ref{thm:lowcost_noncollab}.
    \end{itemize}
    \item When $\mu_1 < \mu_2$, and Assumption \ref{assm:preferred_collab} holds, collaborative care is both faster and less costly. The presence or absence of downstream queueing then becomes the key determinant.
    \begin{itemize}
        \item If $\ell < C^G$, collaborative care can begin immediately. In this case, it is always optimal to seek collaboration, thereby avoiding idle GPs, as shown in Statement \ref{state:fast_collab_no_queue_in_collab} in Theorem \ref{thm:lowcost_collab}.
        \item If $\ell \geq C^G$, all GPs are currently busy, and seeking collaborative care requires the NP to wait with the patient until a GP becomes available. The expected time (resp. cost) to complete collaborative care now includes this waiting period and equals $\frac{\ell+1}{C^G \mu_2}$ (resp. $\frac{\ell+1}{C^G}\frac{h_2}{\mu_2}$). Detailed derivations are provided in the proof of Proposition~\ref{prop:noncollab} in Appendix~\ref{sec:proof_support}. When both 
        \begin{align}
            &\frac{1}{\mu_1} \leq \frac{\ell+1}{C^G\mu_2}, &\frac{h_1}{\mu_1} \leq \frac{\ell+1}{C^G}\frac{h_2}{\mu_2} \label{eq:time_and_cost}
        \end{align}
        hold, the NP’s independent service yields lower or equal total time ($\frac{1}{\mu_1}$) and cost ($\frac{h_1}{\mu_1}$). Since non-collaborative service is also free of blocking concerns, it becomes an optimal decision under these conditions. This behavior is captured in Statement \ref{state:fast_collab_queue_in_collab} in Theorem \ref{thm:lowcost_collab}.
    \end{itemize}
    \item When $\mu_1 = \mu_2$, the two options are equal in service speed, and the analysis reduces to the previous cases. Specifically, the subcases $\ell < C^G$ and $\ell \geq C^G$ correspond to those discussed under $\mu_1 < \mu_2$ with $\ell < C^G$ and $\mu_1 > \mu_2$, respectively; See Statement \ref{state:equal_rate} in Theorem \ref{thm:lowcost_collab}.
\end{itemize}

Notice Statement \ref{state:fast_noncollab} in Theorem \ref{thm:lowcost_collab} suggests a natural conjecture regarding policy monotonicity in the upstream queue length. Specifically, one might expect the optimal policy switches from collaborative to non-collaborative care as the upstream queue length increases. While this intuition aligns with empirical observations under various parameter settings, establishing the existence of such a threshold-type policy remains elusive. 
In particular, a conventional approach of proving the difference function $D(i,j,k,\ell)$ in \eqref{def:D} is non-increasing in $i$ fails in general. 
For example, when $C^p = 2, C^G = 1, \mu_0 = 5, \mu_1 = 3.1, \mu_2 = 3, h_0 = 0.1, h_1 = 22, h_2 = 10$, we have $D(2,1,1,0) = 2.5714$ and $D(3,1,1,0) = 2.5716$, which violates the expected non-increasing behavior in $i$.

When $\mu_1 < \mu_2$ and $\ell \geq C^G$, reversing either inequality in \eqref{eq:time_and_cost} leads to markedly more intricate dynamics and induces subtle but consequential structural changes in the optimal policy. This contrasts with the simpler, typically monotone behavior observed when $\mu_1 \geq \mu_2$. To illustrate these complexities, we examine how different combinations of inequality directions in \eqref{eq:time_and_cost} influence the optimal action at $(i, 2, 0, 2)$, where we fix $C^p = 4$, $C^G = 2$, $\mu_1 = 10$, $h_1 = 1$, and vary others.

\begin{itemize}
    \item Consider $\frac{1}{\mu_1} \leq \frac{\ell+1}{C^G \mu_2}$ and $\frac{h_1}{\mu_1} > \frac{\ell+1}{C^G} \frac{h_2}{\mu_2}$. When $\mu_0 = 1$, $\mu_2 = 12$, $h_0 = 0.5$, $h_2 = 0.6667$, the optimal policy is indeed monotone: $a = \mathbb{I}\{i < 4\}$. However, increasing the triage service rate to $\mu_0 = 100$ leads to a policy that always opts for non-collaborative care, i.e., $a = 0$.
    \item Suppose $\frac{1}{\mu_1} > \frac{\ell+1}{C^G\mu_2}$ and $\frac{h_1}{\mu_1} \leq \frac{\ell+1}{C^G}\frac{h_2}{\mu_2}$. For parameters $\mu_0 = 1$, $\mu_2 = 18$, $h_0 = 0.5$, $h_2 = 1.5$, the policy switches to collaborative care when the initial upstream queue is high: $a = \mathbb{I}\{i > 13\}$. Again, when $\mu_0$ increases to $100$, the policy collapses to $a = 0$.
    \item If $\frac{1}{\mu_1} > \frac{\ell+1}{C^G\mu_2}$ and $\frac{h_1}{\mu_1} > \frac{\ell+1}{C^G}\frac{h_2}{\mu_2}$, for $\mu_0 = 8$, $\mu_2 = 18$, $h_0 = 5$, $h_2 = 0.75$, the optimal policy is non-monotonic: $a = \mathbb{I}\{i < 1 \text{ or } i > 7\}$. Varying $\mu_0$ again shows the policy is sensitive to the triage rate. When $\mu_0 = 0.8$, the system always chooses collaborative care ($a = 1$); when $\mu_0 = 9$, the action becomes $a = \mathbb{I}\{a<1\}$; and when $\mu_0 = 100$, the policy reverts to $a = 0$.
\end{itemize}
\noindent
These examples reveal several key insights. First, unlike the case $\mu_1 \geq \mu_2$, the optimal control cannot be fully described by a monotone threshold policy. In some parameter settings, the decision switches more than once as $i$ increases. Nonetheless, numerical evidence suggests that the policy eventually stabilizes for sufficiently large $i$, becoming either always collaborative or always non-collaborative beyond a threshold. 
Second, the stabilized decision under depends not only on the relative service rates $\mu_1$ and $\mu_2$, but also crucially on the triage rate $\mu_0$. The interplay between $\mu_0$, $\mu_1$, and $\mu_2$ determines whether blocking effects accumulate and whether time savings at the downstream phase translate into meaningful upstream cost reductions.
To elaborate:
\begin{enumerate}
    \item \label{state:lowtime_concollab}
    If $\frac{1}{\mu_1} \leq \frac{\ell+1}{C^G\mu_2}$, non-collaborative service completes the current patient in less or equal time, allowing the NP to return to triage sooner and thus reducing the cost associated with upstream queueing. Since non-collaborative service is not subject to downstream blocking, the resulting upstream cost savings can outweigh the higher per-patient service cost ($\frac{h_1}{\mu_1} > \frac{\ell+1}{C^G} \frac{h_2}{\mu_2}$) when the initial queue is large, thus shifting the optimal decision toward non-collaborative care.
    \item \label{state:lowtime_collab}
    If $\frac{1}{\mu_1} > \frac{\ell+1}{C^G\mu_2}$, collaborative care completes the current patient more quickly but introduces blocking complications, which may compound over time.
    However, when $\mu_0$ is small relative to both $\mu_1$ and $\mu_2$, downstream services are often completed before the upstream service finishes, and thus before any subsequent decision is required.
    In this case, the blocking effect does not propagate beyond the NP-patient pairs that are initially queueing for collaboration. As a result, the upstream cost savings from faster patient completion become dominant, especially under heavy initial triage queues, making collaborative care the preferred option.
\end{enumerate}
Statements \ref{state:lowtime_concollab}-\ref{state:lowtime_collab} are proved in \citep{lu2026controlpoliciestwostagequeueing} for the special case with $C^p = 2$ and $C^G = 1$. 
We conclude this section with Theorem \ref{thm:enough_GP}, which complements Theorems \ref{thm:lowcost_collab}-\ref{thm:lowcost_noncollab} by assuming that $C^G \geq C^p$. Under this condition, GP capacity is sufficient to eliminate downstream blocking.

\begin{theorem} \label{thm:enough_GP}
    Suppose $C^G \geq C^p$ and consider any state $(i,j,k,\ell) \in \X_D$.
    \begin{enumerate}
        \item If $\frac{h_1}{\mu_1} > \frac{h_2}{\mu_2}$,
        \begin{enumerate}
            \item if $\mu_1 \leq \mu_2$, the optimal policy always seeks collaborative care;
            \item if $\mu_1 > \mu_2$, there exists some finite $N(j,k,\ell)$ such that the optimal policy does not seek collaborative care for all $i \geq N(j,k,\ell)$;
        \end{enumerate}
        \item If $\frac{h_1}{\mu_1} \leq \frac{h_2}{\mu_2}$,
        \begin{enumerate}
            \item if $\mu_1 \geq \mu_2$, the optimal policy never seeks collaborative care;
            \item if $\mu_1 < \mu_2$, there exists some finite $N(j,k,\ell)$ such that the optimal policy does seeks collaborative care all $i \geq N(j,k,\ell)$.
        \end{enumerate}
    \end{enumerate}
\end{theorem}

\begin{proof}
    This conclusion is an immediate consequence of Proposition \ref{cor:enough_GP}.
\end{proof}

\begin{table} [htbp]
\caption{Action at state $(i,j,k,\ell) \in X_D$ under an optimal policy if $C^G \geq C^p$.}\label{table:enough_GP_full}
\centering
\begin{tabular}{|c|c|c|}
\hline

\multicolumn{1}{|l|}{} & \multicolumn{1}{c|}{\textbf{$\frac{h_1}{\mu_1} \geq \frac{h_2}{\mu_2}$}} & \multicolumn{1}{c|}{\textbf{$\frac{h_1}{\mu_1} \leq \frac{h_2}{\mu_2}$}} \\ \hline
\textbf{$\mu_1 < \mu_2$}    & $a=1$               & $a=1$ for large $i$ \\
\textbf{$\mu_1 = \mu_2$}    & $a=1$               & $a=0$       \\
\textbf{$\mu_1 > \mu_2$}    & $a=0$ for large $i$ & $a=0$       \\
\hline
\end{tabular}
\end{table}

\subsection{Supporting results} \label{sec:support_results}
This subsection analyzes the sign behavior of the difference $D(i,j,k,\ell)$ defined in \eqref{def:D}, which reflects decision preferences.
The proofs of these supporting results are presented in Appendix \ref{sec:proof_support}, with preliminary lemmas and their proofs detailed in Appendix \ref{sec:proof_prelim}.

\begin{proposition} \label{prop:larger_mu1_threshold}
    Suppose $\mu_1 \geq \mu_2$. Consider any decision state $(i,j,k,\ell) \in \X_D$. 
    For $\ell \geq C^G$, there exists a finite threshold $N(j,k,\ell)$ such that for all $i \geq N(j,k,\ell)$,
    \begin{align}
        D(i,j,k,\ell) = v(i,j-1,k+1,\ell) - v(i,j-1,k,\ell+1) \leq 0. \label{eq:larger_mu1_threshold}
    \end{align}
    If we further assume $\mu_1 >\mu_2$ strictly, then a similar result holds for $\ell <C^G$ as well.
\end{proposition}

\begin{proposition} \label{prop:always_collab}
    Suppose $\frac{h_1}{\mu_1} \geq \frac{h_2}{\mu_2}$ and $\mu_2 \geq \mu_1$.
    Consider any decision state $(i,j,k,\ell) \in \X_D$ such that $\ell < C^G$. For all $i$, we have 
    \begin{equation} 
        D(i,j,k,\ell) = v(i,j-1,k+1,\ell) - v(i,j-1,k,\ell+1) \geq 0. \label{eq:always_collab}
    \end{equation}
    Moreover, if Assumption \ref{assm:preferred_collab} holds, i.e.,  $\frac{h_1}{\mu_1} > \frac{h_2}{\mu_2}$, then inequality \eqref{eq:always_collab} holds strictly. 
\end{proposition}

\begin{proposition} \label{prop:noncollab}
    Consider any state $(i,j,k,\ell) \in \X_D$.
    \begin{enumerate}
        \item \label{state:noncollab_queue_in_collab}
        For $\ell \geq C^G$, suppose $\frac{h_1}{\mu_1} \leq \frac{\ell + 1}{C^G} \frac{h_2}{\mu_2}$ and $\frac{1}{\mu_1} \leq \frac{\ell +1}{C^G\mu_2}$, then $D(i,j,k,\ell) \leq 0$.
         \item \label{state:noncollab_no_queue_in_collab}
         For $\ell < C^G$, suppose $\frac{h_1}{\mu_1} \leq \frac{h_2}{\mu_2}$ and $\frac{1}{\mu_1} \leq \frac{1}{\mu_2}$, then $D(i,j,k,\ell) \leq 0$.
    \end{enumerate}
\end{proposition}

\begin{corollary} \label{cor:enough_GP}
    Suppose $C^G \geq C^p$ and consider any state $(i,j,k,\ell) \in \X_D$.
    \begin{enumerate}
        \item If $\frac{h_1}{\mu_1} \geq \frac{h_2}{\mu_2}$,
        \begin{enumerate}
            \item if $\mu_1 \leq \mu_2$, we have $D(i,j,k,\ell) \geq 0$ for all $i$.
            \item if $\mu_1 > \mu_2$, there exists a finite threshold $N(j,k,\ell)$ such that $D(i,j,k,\ell) \leq 0$ for all $i \geq N(j,k,\ell)$.
        \end{enumerate}
        \item If $\frac{h_1}{\mu_1} \leq \frac{h_2}{\mu_2}$,
        \begin{enumerate}
            \item if $\mu_1 \geq \mu_2$, we have $D(i,j,k,\ell) \leq 0$ for all $i$.
            \item if $\mu_1 < \mu_2$, there exists a finite threshold $N(j,k,\ell)$ such that $D(i,j,k,\ell) \geq 0$ for all $i \geq N(j,k,\ell)$.
        \end{enumerate}
    \end{enumerate}
    Notice that if $\frac{h_1}{\mu_1} = \frac{h_2}{\mu_2}$ and $\mu_1 = \mu_2$, a tie arises between the two options, and any decision rule can be applied to break the tie.
\end{corollary}

\section{Heuristics Design} \label{sec:heuristics}
Building on structural patterns observed in optimal policies, this section develops a simple and effective heuristic to support practical decision-making. Recall from the Bellman equations, we observe that the sign of the difference $D(i,j,k,\ell)$ defined in \eqref{def:D} determines the preferred action in decision state $(i,j,k,\ell) \in \X_D$ (with $j\geq 1$): collaborative or non-collaborative care.

Although the optimal policy cannot, in general, be fully characterized as a threshold policy, in the sense that the sign of $D$ may change more than once as $i$ increases, we observe, both analytically and empirically, that for sufficiently large $i$, the sign of $D(i,j,k,\ell)$ stabilizes and remain unchanged. Given our focus on scenarios with large initial upstream queues (i.e., large $i$), the goal is to estimate the threshold at or beyond which the decision remains fixed  or equivalently, where $D$ ceases to change sign and to identify the corresponding sign and associated action.

\subsection{Motivation and definition of the piecewise linear approximation}
Motivated by the observation that the sign of $D$ stabilizes once $i$ reaches or exceeds a certain threshold, we approximate $D$ with a piecewise linear function $H$ in $i$ that preserves this structural property, and estimate the sign of $D$ using that of $H$.
Depending on the parameters, the sign of $H$ may remain either positive or non-positive for all sufficiently large values of $i$. Definition \ref{def:H_ceiling} introduces a specific piecewise linear form,  $H$, while Definition \ref{def:H_linear} further simplifies it to a linear form for $i \geq 1$, and denoted by $H_{\text{Lin}}$ for distinction.

\begin{definition} \textbf{(Piecewise linear form)} \label{def:H_ceiling}
    For any $(i,j,k,\ell) \in \X_D$, define $H(i,j,k,\ell)$ as follows:
    \begin{enumerate}
        \item When $i = 0$, 
        {\small
        \begin{align}
            H(0,j,k,\ell) := 
            \begin{cases}
                b,      & \text{ if }  \ell < C^G,\\
                b_\ell  & \text{ if }  \ell \geq C^G, 
            \end{cases} \label{eq:def_H_base}
        \end{align} 
        }
        where
        {\small
        \begin{align}
            b &:= \frac{h_1}{\mu_1}-\frac{h_2}{\mu_2}, \label{eq:def_b} \\
            b_\ell &:= \frac{h_1}{\mu_1}-\frac{\ell+1}{C^G}\frac{h_2}{\mu_2}, \quad \text{for } \ell \geq C^G \label{eq:def_b_ell}.
        \end{align}
        }
        \item When $i \geq 1$, so that $j+k+\ell=C^p$, 
        {\scriptsize
        \begin{align}
            H(i,j,k,\ell)
            &:= 
            \begin{cases}
                -1 & \text{ if } \ell \geq C^G \text{ and } c_\ell \leq 0 \text{ and } b_\ell \leq 0,\\
                w_{k,\ell} \cdot H^\infty(i,j,k,\ell) + (1-w_{k,\ell}) \cdot H^0(i,j,k,\ell) & \text{ if } \ell \geq C^G \text{ and } (c_\ell > 0 \text{ or } b_\ell > 0),\\
                H^{0}(i,j,k,\ell), & \text{ if }  \ell < C^G,
            \end{cases} \label{eq:def_H}
        \end{align}
        }
        where 
        {\scriptsize
            \begin{align}
                H^{0}(i,j,k,\ell) 
                &:= \begin{cases}
                \left( \left\lceil\frac{i-k}{C^p}\right\rceil \left(\frac{1}{\mu_1} - \frac{1}{\mu_2}\right)  - \sum_{r=C^G}^\ell \left\lceil\frac{i-k-r}{C^p}\right\rceil \frac{1}{C^G \mu_2}\right) h_0 + b_\ell, & \text{ if } \ell \geq C^G \text{ and } c \leq 0,\\
                \left( \left\lceil\frac{i-k-\ell'}{C^p}\right\rceil \left(\frac{1}{\mu_1} - \frac{\ell'+1}{C^G \mu_2}\right)  - \sum_{r=\ell'+1}^{\ell} \left\lceil\frac{i-k-r}{C^p}\right\rceil \frac{1}{C^G \mu_2}\right)h_0 + b_\ell, & \text{ if } \ell \geq C^G \text{ and } c_\ell \leq 0 < c,\\
                \left\lceil\frac{i-\ell}{C^p}\right\rceil \left(\frac{1}{\mu_1} - \frac{\ell+1}{C^G \mu_2}\right) h_0 + b_\ell, & \text{ if } \ell \geq C^G \text{ and } c_\ell > 0,\\
                \left\lceil\frac{i-k}{C^p}\right\rceil \left(\frac{1}{\mu_1}-\frac{1}{\mu_2}\right) h_0 + b, & \text{ if }  \ell < C^G \text{ and } c \leq 0,\\
                b, & \text{ if }  \ell < C^G \text{ and } c > 0,
            \end{cases} \label{eq:small_mu0_H}
            \end{align}
        }
        {\small
        \begin{align}
            c &: = \frac{h_0}{C^p}\left(\frac{1}{\mu_1} - \frac{1}{\mu_2}\right), \label{eq:def_c}\\
            c_\ell &:= \frac{h_0}{C^p}\left(\frac{1}{\mu_1} - \frac{\ell+1}{C^G\mu_2}\right), \quad \text{for } \ell \geq C^G, \label{eq:def_c_ell} \\
             \ell' &:= \max\left\{\ell \in \N \Big{|}\frac{\ell}{C^G \mu_2} < \frac{1}{\mu_1}\right\}, \label{eq:def_ell_prime} \quad \text{for } c > 0 \text{ and } \ell \geq C^G, 
        \end{align}
        \begin{align}
            H^\infty(i,j,k,\ell) &:= (i-y_\ell)c'+b', \label{eq:large_mu0_H_general}\\
            c' &:= -\frac{h_0}{C^G\mu_2}, & b'& := \frac{h_1-h_2}{\mu_1} - \frac{C^ph_2}{C^G\mu_2}, \label{def:b'_and_c'} \\
            y_\ell &:= C^p - \ell -1 + \frac{C^G \mu_2}{\mu_1}, \quad \text{ for } \ell \geq C^G, \label{def:y}\\
            w_{k,\ell} &:= \frac{j\mu_0}{j\mu_0 + k\mu_1 + C^G\mu_2}, \quad \text{ for } \ell \geq C^G.   \label{eq:def_w}
        \end{align}
        }
    \end{enumerate}
\end{definition}
Notice that $w_{k,\ell}$ in \eqref{eq:def_w} is independent of $j$, as $j$ is determined by the identity $j+k+\ell=C^p$.
Definition \ref{def:H_linear} below follows the same structure as Definition \ref{def:H_ceiling}, except that it replaces the piecewise function $H^0(i,j,k,\ell)$ with a linear form $H^0_{\text{Lin}}(i,j,k,\ell)$, as given in \eqref{eq:small_mu0_H_linear}.
\begin{definition} \textbf{(Linear form)} \label{def:H_linear}
    Consider any $(i,j,k,\ell) \in \X_D$. Define $H_{\text{Lin}}(i,j,k,\ell)$ as follows:
    \begin{enumerate}
        \item When $i = 0$, $H_{\text{Lin}}(0,j,k,\ell)$ is the same as in \eqref{eq:def_H_base}.
        \item When $i \geq 1$, so that $j+k+\ell=C^p$, 
        {\scriptsize
        \begin{align}
            H_{\text{Lin}}(i,j,k,\ell) 
            &:=
            \begin{cases}
                -1 & \text{ if } \ell \geq C^G \text{ and } c_\ell \leq 0 \text{ and } b_\ell \leq 0,\\
                w_{k,\ell} \cdot H^\infty(i,j,k,\ell) + (1-w_{k,\ell}) \cdot H^0_{\text{Lin}}(i,j,k,\ell) & \text{ if } \ell \geq C^G \text{ and } (c_\ell > 0 \text{ or } b_\ell > 0),\\
                H^{0}_{\text{Lin}}(i,j,k,\ell), & \text{ if }  \ell < C^G,
            \end{cases} \label{eq:def_H_linear}
          \end{align}
          }
        where 
        {\small
        \begin{align}
            H^0_{\text{Lin}}(i,j,k,\ell) := \begin{cases}
                ic_\ell + b_\ell, & \text{ if }  \ell \geq C^G, \\
                ic+b & \text{ if } \ell < C^G,
            \end{cases} \label{eq:small_mu0_H_linear}
        \end{align}
        }
        $b, b_\ell, c, c_\ell, b_\ell$, $H$, and $w_{k,\ell}$ are defined in \eqref{eq:def_b}, \eqref{eq:def_b_ell}, \eqref{eq:def_c}, \eqref{eq:def_c_ell}, \eqref{eq:large_mu0_H_general} and \eqref{eq:def_w}, respectively.
    \end{enumerate}
\end{definition}

The construction of $H$ and $H_{\text{Lin}}$ is well-motivated. Focusing on large $i$, we provide an intuitive case-by-case preview when $i\geq1$, before introducing the detailed design process in Section \ref{sec:construct_H}.
\begin{enumerate}[label= \textbf{Case} \arabic*:, leftmargin=2.5\parindent]
    \item Suppose $\ell \geq C^G$.
    \begin{enumerate}[label= \textbf{Subcase} \alph*:, ref=\theenumi(\alph*), leftmargin=2\parindent]
        \item If $ c_\ell \leq 0$ and $b_\ell \leq 0$ (the first case in \eqref{eq:def_H} and \eqref{eq:def_H_linear}), we let both $H$ and $H_{\text{Lin}}$ to $-1 \leq 0$, consistent with $D \leq 0$ (Statement \ref{state:noncollab_queue_in_collab} in Proposition \ref{prop:noncollab}).
        \item If $ c_\ell > 0$ or $b_\ell > 0$ (the second case in \eqref{eq:def_H} and \eqref{eq:def_H_linear}), $H$ and $H_{\text{Lin}}$ are defined as a convex combination of two piecewise linear functions, with the weight capturing the relative magnitude of upstream and downstream service rates. 
    \end{enumerate}
    \item For $\ell < C^G$ (the final case of \eqref{eq:def_H} and \eqref{eq:def_H_linear}), recall that $b > 0$ by Assumption \ref{assm:preferred_collab}.
    \begin{enumerate}[label= \textbf{Subcase} \alph*:, leftmargin=2\parindent]
        \item If $\mu_1 > \mu_2$, which implies that $c < 0$ as given in \eqref{eq:def_c}, then both $H$ and $H_{\text{Lin}}$ become non-positive once $i$ reaches a finite threshold, aligning with the result that $D(i,j,k,\ell) \leq 0$ for all sufficiently large $i$ (Proposition \ref{prop:larger_mu1_threshold}).
        \item If $\mu_1 \leq \mu_2$, both $H$ and $H_{\text{Lin}}$ remain strictly positive for all $i$, in agreement with $D(i,j,k,\ell) > 0$ for all $i$ (Proposition \ref{prop:always_collab}).
    \end{enumerate}
\end{enumerate}

\subsection{Using \texorpdfstring{$H$}{H} and \texorpdfstring{$H_{\text{Lin}}$}{HLin} to approximate decisions with \texorpdfstring{$D$}{D}}
The function $H$ (resp. $H_{\text{Lin}}$) induces a heuristic policy based on its sign and is formalized in Definition \ref{def:action_heuristic}. Specifically, it suggests choosing collaborative care (i.e., $a=1$) if and only if $H(i,j,k,\ell) > 0$ (resp. $H_{\text{Lin}}(i,j,k,\ell) > 0$) in decision state $(i,j,k,\ell) \in \X_D$. 
\begin{definition} \label{def:action_heuristic}
    Consider $(i,j,k,\ell) \in \X_D$. The actions suggested by $H$ and $H_{\text{Lin}}$ are given by
    \begin{align}
        a_H(i,j,k,\ell) := \mathbb{I}\{H(i,j,k,\ell)>0\}, \label{eq:def_a_H}\\
        a_{H_{\text{Lin}}}(i,j,k,\ell) := \mathbb{I}\{H_{\text{Lin}}(i,j,k,\ell)>0\}, \label{eq:def_a_Hl}
    \end{align}respectively, where $H$ and $H_{\text{Lin}}$ are provided in Definitions \ref{def:H_ceiling} and \ref{def:H_linear}, respectively.
\end{definition}
In particular, as an affine function of $i$, $H_{\text{Lin}}(i,j,k,\ell)$ changes sign at most once as $i$ increases, suggesting a threshold-type heuristic policy. Building on the expression of $H_{\text{Lin}}$ in \eqref{eq:def_H_linear}, the corresponding action in \eqref{eq:def_a_H} is explicitly determined by system parameters (and $k,\ell$):
{\small
\begin{align}
    a_{H_{\text{Lin}}}(i,j,k,\ell) = 
    \begin{cases}
        0 & \text{ if } \ell \geq C^G \text{ and } c_\ell \leq 0 \text{ and } b_\ell \leq 0,\\
        \mathbb{I}\{i < R_2(k,\ell)\} & \text{ if } \ell \geq C^G \text{ and } \tilde{c}_{k,\ell} \leq 0,\\
        \mathbb{I}\{i > R_2(k,\ell)\} & \text{ if } \ell \geq C^G \text{ and } \tilde{c}_{k,\ell} > 0,\\
        \mathbb{I}\{i < R_1\} & \text{ if }  \ell < C^G \text{ and } \mu_1 > \mu_2, \\
        1 &\text{ if } \ell < C^G \text{ and } \mu_2 \geq \mu_1,
    \end{cases} \label{def:iH}
\end{align}
}
where
{\small
\begin{align}
    R_1 &:= -\frac{\left(\frac{h_1}{\mu_1}-\frac{h_2}{\mu_2}\right)C^p}{\left(\frac{1}{\mu_1} - \frac{1}{\mu_2}\right)h_0}, \label{eq:R1} \\
    R_2(k,\ell) &:= \begin{cases}
        -\frac{\tilde{b}_{k,\ell}}{\tilde{c}_{k,\ell}}, &\text{ if } \tilde{c}_{k,\ell} \neq 0, \\
        0, &\text{ if } \tilde{c}_{k,\ell} = 0 \text{ and }  \tilde{b}_{k,\ell} \leq 0, \\
        \infty, &\text{ if } \tilde{c}_{k,\ell} = 0 \text{ and } \tilde{b}_{k,\ell}>0. \\
    \end{cases} \\
    \tilde{c}_{k,\ell} &:= w_{k,\ell} \cdot c' + (1-w_{k,\ell}) \cdot c_\ell,\\
    \tilde{b}_{k,\ell} &:= w_{k,\ell} \cdot ( -y_\ell c'+b') + (1-w_{k,\ell}) \cdot b_\ell.
\end{align}
}
Note that $R_1 > 0$ by Assumption \ref{assm:preferred_collab} and $\mu_1 > \mu_2$. 
The heuristic is also summarized in Table \ref{table:heuristic}.

\begin{table}[htbp]
\caption{Action at state $(i,j,k,\ell) \in X_D$ suggested by $H_{\text{Lin}}$, as defined in \eqref{eq:def_a_Hl}.} \label{table:heuristic}
\centering
\small
\begin{tabular}{|c|c|c|}
\hline
\multicolumn{1}{|c|}{}                         & $\ell< C^G$ & $\ell \geq C^G$ \\ \hline
\multicolumn{1}{|l|}{$\mu_1 \geq \mu_2$}                        &     $a=1$               & \multirow{3}{*}{ $a = \left\{
\begin{array}{ll}
0 & \text{if } c_\ell, b_\ell \leq 0, \\
\mathbb{I}\{i < R_2(k,\ell)\} & \text{if } \tilde{c}_{k,\ell} \leq 0, \\
\mathbb{I}\{i > R_2(k,\ell)\} & \text{if } \tilde{c}_{k,\ell} > 0.
\end{array}
\right. $}      \\ \cline{1-2}
\multicolumn{1}{|c|}{\multirow{2}{*}{$\mu_1 > \mu_2$}}  &  \multirow{2}{*}{$a = \mathbb{I}\{i < R_1\}$}                 &                        \\
\multicolumn{1}{|c|}{}                     &                  &                        \\ \hline
\end{tabular}
\end{table}

Note that, unlike the theoretical results, which analyze optimal policies based on the relative magnitudes of the downstream service rates $\mu_1$ and $\mu_2$ (see Table \ref{table:policy}), the heuristic is instead characterized by the relative ordering of $\ell$ and $C^G$, reflecting the impact of collaborative care queueing on upstream patients.

\subsection{Construction of \texorpdfstring{$H$}{H} and \texorpdfstring{$H_{\text{Lin}}$}{HLin} across the parameter space} \label{sec:construct_H}
In this subsection, we present the detailed process of constructing $H$ and its linear simplification $H_{\text{Lin}}$ under various parameter configurations. 
Looking closely at the expression of $D(i,j,k,\ell)$ in \eqref{def:D}, which compares the values of two costs, we note that
{\small
\begin{align*}
    v(i,j-1,k+1,\ell)
    & = c_0^1h_0 + c_1^1h_1 + c_2^1h_2,
\end{align*}
}
where the coefficients $c_0^1 := \mathbb{E}_{(i,j-1,k+1,\ell)} \big[\int_{0}^{\infty}\big(Q^{\pi^*}_{0,0}(t) + Q^{\pi^*}_{0,1}(t)\big) dt \big]$, $c_1^1 := \mathbb{E}_{(i,j-1,k+1,\ell)} \big[\int_{0}^{\infty}Q^{\pi^*}_1(t) dt \big]$, and $c_2^1 := \mathbb{E}_{(i,j-1,k+1,\ell)} \big[\int_{0}^{\infty}Q^{\pi^*}_2(t) dt \big]$. Here $\pi^*$ denotes an optimal policy.
Similarly, 
\begin{align*}
    v(i,j-1,k,\ell+1) = c_0^2h_0 + c_1^2h_1 + c_2^2h_2,
\end{align*}
for some coefficients $c_0^2, c_1^2$ and $c_2^2$. Therefore, $D(i,j,k,\ell)$ can be expressed as
\begin{align}
     D(i,j,k,\ell) = D_{0}(i,j,k,\ell) + D_{1,2}(i,j,k,\ell), \label{eq:D_decompose}
\end{align}where 
\begin{align}
    D_0(i,j,k,\ell) = (c_0^1-c_0^2)h_0, \qquad
    D_{1,2}(i,j,k,\ell) = (c_1^1-c_1^2)h_1 + (c_2^1-c_2^2)h_2, \label{eq:D0_and_D12}
\end{align}
represent the upstream and downstream cost differences under the two actions, respectively.
This decomposition suggests estimating $D_0$ and $D_{1,2}$ separately. Let $H_0$ and $H_{1,2}$ be their respective approximations, each piece-wise linear in $i$, leading to the combined estimate:
\begin{align}
    H(i,j,k,\ell) = H_{0}(i,j,k,\ell) + H_{1,2}(i,j,k,\ell). \label{eq:H_decompose}
\end{align}
Notably, $D_0$ in \eqref{eq:D0_and_D12} is proportional to $h_0$ and independent of $h_1$ and $h_2$, while $D_{1,2}$ in \eqref{eq:D0_and_D12} depends only on $h_1$ and $h_2$. The design of $H_0$ and $H_{1,2}$ should reflect these dependencies accordingly.
The breakdown of the upstream and downstream components in \eqref{eq:D_decompose} and \eqref{eq:H_decompose} is also behaviorally grounded, as it mirrors the trade-off a decision-maker faces when evaluating a single action against two (potentially competing) objectives:
\begin{enumerate}[leftmargin=3cm, label=\textbf{Objective} \arabic*.,ref={ \arabic*}]
    \item \label{obj:upstream} The cost difference at Station 0 due to different downstream service times.
    \item \label{obj:downstream} The cost difference of service between Stations 1 and 2.
\end{enumerate}
The precise forms of $H_0$ and $H_{1,2}$ depend critically on the relative magnitudes of the service rates at Stations 1 and 2. We examine these through distinct cases based on the value of $\ell$ in the remainder of this subsection. Specifically, whether the downstream service allocation leads to immediate collaborative service or induces a queue, requiring the NP (and the associated patient) to wait.

\subsubsection{No immediate queueing effects if collaborating \texorpdfstring{($\ell < C^G$)}{(ell < CG)}} \label{sec:H_no_queue_in_collab}
Consider first the decision state $(i,j,k,\ell) \in \X_D$ where $\ell < C^G$, so that both actions ensure immediate service for the current patient. 
The analysis can be separated into two cases:
\begin{enumerate} [label= \textbf{Case} \arabic*:, ref = \arabic*, leftmargin=3\parindent] 
    \item \label{case:large_mu2_positive}
    If $\mu_2 \geq \mu_1$, Proposition \ref{prop:always_collab} implies that $D(i,j,k,\ell) > 0$ for all $i$. In this case, any choice of $H(i,j,k,\ell) >0$ will suffice for the consistency of the sign of $D$.
    \item \label{case:large_mu1} If $\mu_1 > \mu_2$, recall from Proposition \ref{prop:larger_mu1_threshold} that there exists a finite threshold $N(j,k,\ell)$ such that $D(i,j,k,\ell) \leq 0$ for all $i\geq N(j,k,\ell)$. 
\end{enumerate}
Therefore, we focus on Case \ref{case:large_mu1} in developing $H$. 
We begin by introducing the deterministic process.

\begin{assumption}[\textbf{Deterministic process}] \label{assum:deterministic} We say that the system follows the \textit{deterministic model assumptions} if
    \begin{enumerate}
        \item \label{det:state_space}
        The state space remains unchanged and discrete.
        \item \label{det:amount_of_work}
        The system evolves deterministically, with the time spent at each station equal to the expected sojourn time at the corresponding station in the stochastic model.
        \item \label{det:downstream_blocking}
        $C^p$ remains unchanged; Replace $C^G$ with $C^G_{det}$ such that $C^G_{det} \geq C^p$. This eliminates downstream blocking effects (used to approximate the case with $\ell < C^G$ in the stochastic model).  
        \item \label{det:decision_points}
        Decision points continue to be discrete and coincides with a triage completion. 
    \end{enumerate}
    By Statements \ref{det:amount_of_work} and \ref{det:downstream_blocking}, the total time required at Station $n$ equals $\frac{1}{\mu_n}$ for $n = 0, 1, 2$.
\end{assumption}
While the theoretical results do not guarantee a monotone optimal control, they motivate a threshold heuristic (and thus a monotone $H$), such that the decision is $a=1$ when $i < N(j,k,\ell)$ and $a=0$ otherwise. 
It remains to construct $H(i,j,k,\ell)$ to estimate $N(j,k,\ell)$.
Note that if at the point $i$, there is a question of what to do, then we are at or close to the threshold. Since from that point, the system load is reducing ($i$ is getting smaller) and is likely below the threshold after the initial, we consider a policy that collaborates for all time.
\begin{definition}[\textbf{Always-collaborative control} $\mathbf{\pi^{co\ell}}$] \label{defn:always_collab}
    We refer to the control policy that chooses collaborative service ($a=1$) for all decisions beyond the current time as \emph{always-collaborative control}.
\end{definition}

To formalize Objectives \ref{obj:upstream} and \ref{obj:downstream} and define $H$ accordingly under the \textit{deterministic model assumptions} and \textit{always-collaborative control $\mathbf{\pi^{co\ell}}$}, consider two systems that evolve from an identical decision state $(i,j,k,\ell) \in \X_D$ but differ in the action taken at the initial decision point:
\begin{itemize}
    \item System 1 proceeds immediately after the first decision-maker (denoted as $\text{NP}^*_{(1)}$) selects non-collaborative care ($a=0$). That is, System 1 starts at state $(i,j-1,k+1,\ell)$, where all services are poised to begin.
    \item System 2 starts from state $(i, j - 1, k, \ell + 1)$ with all services about to commence, resulting from a second decision-maker ($\text{NP}^*_{(2)}$) who instead chooses collaborative care ($a = 1$) at $(i,j,k,\ell)$. 
\end{itemize}

\begin{assumption}\label{assm:always-collab}
    Suppose all decisions in Systems 1 and 2 follow always-collaborative control.
\end{assumption}
We now provide formal definitions of $H_0$ and $H_{1,2}$, which represent approximations of the cost difference incurred upstream and downstream, respectively, by differing the initial action only.
\begin{definition}\label{def:H0_and_H12}
    Consider $(i,j,k,\ell)\in\X_D$. Suppose the modeling and policy assumptions described in Assumptions \ref{assum:deterministic} and \ref{assm:always-collab}, respectively, hold for Systems 1 and 2. Define the following quantities:
    \begin{itemize}
        \item $H_{1,2}(i,j,k,\ell)$ represents the \textbf{total} downstream cost difference (at Station 1 or 2) between Systems 1 and 2 due to the initial decision not to collaborate versus to collaborate;
        \item $H_0(i,j,k,\ell)$ be the \textbf{total} upstream cost difference (at Station 0) between the two systems due to the time difference required to leave Station 1 or 2;
    \end{itemize}
\end{definition}

Proposition \ref{prop:heur_no_queue_in_collab} below provides the expressions of $H_0$ and $H_{1,2}$ in the case that $\mu_1 \geq \mu_2$. The proof is given in Appendix \ref{sec:proof_heuristic}.
\begin{proposition} \label{prop:heur_no_queue_in_collab}
    Recall $\ell < C^G$ and suppose $\mu_1 \geq \mu_2$. Consider any state $(i,j,k,\ell) \in \X_D$. The following hold for $H_0(i,j,k,\ell)$ and $H_{1,2}(i,j,k,\ell)$, as defined in Definition \ref{def:H0_and_H12}.
    \begin{align}
        H_0(i,j,k,\ell)& = \left\lceil\frac{i-k}{C^p}\right\rceil \left(\frac{1}{\mu_1}-  \frac{1}{\mu_2}\right) h_0,
        \label{eq:H0} \\
        H_{1,2}(i,j,k,\ell)& = b = \frac{h_1}{\mu_1}-  \frac{h_2}{\mu_2}, \label{eq:H12}
    \end{align}
    as provided in \eqref{eq:def_b}.
\end{proposition}

Substituting \eqref{eq:H0} and \eqref{eq:H12} into \eqref{eq:H_decompose} yields the expression for $H(i,j,k,\ell)$ presented in \eqref{eq:def_H} of Definition \ref{def:H_ceiling} for the case $\ell<C^G$ with $\mu_1 \geq \mu_2$. When $\mu_1 \leq \mu_2$, we set $H(i,j,k,\ell) = b > 0$ by Assumption \ref{assm:preferred_collab} to match the sign of $D$ (see Case \ref{case:large_mu2_positive}). 
Definition \ref{def:H_linear} further simplifies this formulation by replacing the ceiling term $\left\lceil\frac{i-k}{C^p}\right\rceil$ 
with a linear term $\frac{i}{C^p}$, thereby yielding a unified representation applicable to both $\mu_1 > \mu_2$ and $\mu_2 \geq \mu_1$. This modification is motivated by the following considerations regarding the structure of the resulting linear approximation $H_{\text{Lin}}$ in \eqref{eq:def_H_linear}:
\begin{itemize}
    \item A linear function is monotone and therefore induces a threshold-type policy via \eqref{eq:def_a_Hl}. Moreover, the formula in \eqref{eq:small_mu0_H_linear} has simple expressions for both slope $c = \frac{h_0}{C^p}\left(\frac{1}{\mu_1}-  \frac{1}{\mu_2}\right)$ and intercept $b = \frac{h_1}{\mu_1}-  \frac{h_2}{\mu_2}$, yielding a simple formula for the estimated threshold $R_1$ (see \eqref{eq:R1}).
    \item The linear form preserves the original analytical and practical interpretation of $H$ as a composite approximation of cost differences arising from both upstream and downstream dynamics. These two components remain clearly identifiable in the resulting expression:
    \begin{itemize}
        \item The intercept $H_{1,2}$ depends solely on downstream parameters $\mu_1$, $h_1$, $\mu_2$, and $h_2$, thus approximating the downstream contribution $D_{1,2}$ in \eqref{eq:D0_and_D12} and capturing Objective \ref{obj:downstream}.
        \item The remaining term in \eqref{eq:small_mu0_H_linear}, which is proportional to $h_0$ and the time difference $\frac{1}{\mu_1}-  \frac{1}{\mu_2}$ required to leave downstream stations, approximates $D_0$ in \eqref{eq:D0_and_D12} and reflects Objective \ref{obj:upstream}.
    \end{itemize}
    \item The linear form in \eqref{eq:small_mu0_H_linear} extends the preceding decomposition of upstream and downstream cost differences to the case $\mu_1 \leq \mu_2$ for $\ell < C^G$ through a unified expression. Furthermore, when $\mu_1 \leq\mu_2$, we have $H_{\text{Lin}}(i,j,k,\ell) = ic+b >0$, which is consistent with the result $D > 0$ (see Case \ref{case:large_mu2_positive}). This follows from $b > 0$ (by \eqref{eq:def_b} and Assumption \ref{assm:preferred_collab}) and $c \geq 0$ (see \eqref{eq:def_c}). 
    
    \item The two expressions in \eqref{eq:small_mu0_H} and \eqref{eq:small_mu0_H_linear} often yield similar values when $\mu_1 > \mu_2$ for $\ell<C^G$: 
    \begin{itemize}
        \item Applying $i = 0$ in \eqref{eq:small_mu0_H} and \eqref{eq:small_mu0_H_linear} returns $H(0,j,k,\ell) = H_{\text{Lin}}(0,j,k,\ell) = b > 0$. This coincides with the base case in \eqref{eq:def_H_base}, and suggest a finite, positive threshold estimate. 
        \item When $i \geq 1$, the identity $j + k + \ell = C^p$ implies $\frac{i-k}{C^p} \leq \left\lceil\frac{i-k}{C^p}\right\rceil < \frac{i-k}{C^p} + 1 = \frac{i + j + \ell}{C^p}$.
        The term $\frac{i}{C^p}$ used in \eqref{eq:small_mu0_H_linear} also lies within this range. Furthermore, as $i$ (and hence the threshold $R_1$) increases, the gap between the ceiling-based expression \eqref{eq:small_mu0_H} and its linear counterpart \eqref{eq:small_mu0_H_linear} narrows, supporting the validity of the linear approximation.
    \end{itemize}
\end{itemize}

\subsubsection{Immediate queueing effects if collaborating \texorpdfstring{($\ell \geq C^G$)}{(ell >= CG)}} \label{sec:H_queue_in_collab}
We now examine the decision state $(i, j, k, \ell) \in \X_D$ with $\ell \geq C^G$, where choosing collaborative care does not result in immediate service. 
In this setting, the optimal policy exhibits distinct structural behaviors depending on the relative magnitude of $\mu_0$ compared to $\mu_1$ and $\mu_2$: 
\begin{enumerate} [label= \textbf{Case} \arabic*:,ref={ \arabic*}, leftmargin=3\parindent] 
    \item \label{case:small_mu0}
    If $\mu_0$ is very small relative to both $\mu_1$ and $\mu_2$, any downstream service (if applicable) is likely to complete before the next triage finishes (hence before the next decision point). Consequently, blocking is limited to the NP–patient pairs initially in the collaborative queue, while all subsequent pairs begin service without waiting.
    \item \label{case:large_mu0}
    If $\mu_0$ is relatively large, the blocking in the collaborative station tends to compound, since NPs complete triage more quickly than GPs clear their queue. This leads to successive decisions being made while blocking persists or intensifies.
\end{enumerate}
For fixed downstream service rates $\mu_1$ and $\mu_2$, we add superscripts to $D^{\mu_0}$ and its approximation $H^{\mu_0}$ to highlight their dependence on $\mu_0$. 
Accordingly, the function $H^{0}$ (resp. $H^{\infty}$) approximates $D^0$ (resp. $D^\infty$), the limiting form of $D^{\mu_0}$ as $\mu_0 \to 0$ (resp. $\mu_0 \to \infty$) in these extreme cases.

Recall first from Proposition \ref{prop:noncollab} that if $c_\ell \leq 0$ and $b_\ell \leq 0$, where $b_\ell$ and $c_\ell$ are defined in \eqref{eq:def_b_ell} and \eqref{eq:def_c_ell}, respectively, then $D(i,j,k,\ell) \leq 0$. Accordingly, it suffices for any approximation $H\leq 0$ to preserve the sign consistency with $D$. 
Otherwise (either $c_\ell > 0$ or $b_\ell>0$), we construct $H$ as a convex combination of two approximations: $H^0$ derived from the small-$\mu_0$ model under Case \ref{case:small_mu0}, and $H^\infty$ based on the large-$\mu_0$ limiting regime under Case \ref{case:large_mu0}: 
\begin{align}
    H(i,j,k,\ell) & := (1-w_{\ell}) \cdot H^{0}(i,j,k,\ell) + w_{\ell} \cdot H^{\infty}(i,j,k,\ell). \label{eq:convex_comb}
\end{align}
Here, the weighting parameter $w_{k,\ell} \in (0,1)$, given in \eqref{eq:def_w}, represents the probability that triage completion occurs before all other events at state $x=(i,j,k,\ell) \in X_D$. 
Moreover, for fixed $x$, notice that as $\mu_0 \to 0$, we have $w_{k,\ell} \to 0$ and $H \to H^0$; conversely, as $\mu_0 \to \infty$, $w_{k,\ell} \to 1$ and $H \to H^\infty$.

The remainder of this subsection develops the approximations $H^0$ and $H^\infty$, which underpin the convex combination in \eqref{eq:convex_comb}, under the two extreme cases of $\mu_0$. We begin with Case \ref{case:small_mu0}.
The assumption that $\mu_0$ is sufficiently small, so that downstream blocking complications are negligible beyond the initial patients already present in the queue, suggests an extension of the deterministic process framework introduced for the case $\ell < C^G$ in Assumption \ref{assum:deterministic}. 
The proposed extension preserves the original structure with a key modification to Statement \ref{det:downstream_blocking}, as formalized below.
\begin{assumption} [\textbf{Deterministic process with initial downstream blocking}] \label{assum:deterministic_initial_block}    
    We say that the system satisfies the \textit{assumptions of a deterministic model with initial downstream blocking} (hereafter referred to as the \textit{modified deterministic model assumptions}) if Statements \ref{det:state_space}, \ref{det:amount_of_work}, and \ref{det:decision_points} in Assumption \ref{assum:deterministic} remain unchanged, while Statement \ref{det:downstream_blocking} is replaced by the following:
    \begin{enumerate}[label={\arabic*$'$.}, ref = {\arabic*$'$}]\addtocounter{enumi}{2}
    \item \label{modified:downstream_blocking}
    $\frac{1}{\mu_0} > \max\left\{\frac{1}{\mu_1}, \frac{C^p}{C^G \mu_2}\right\}$, ensuring any initial downstream service completes before upstream.
\end{enumerate}
\end{assumption}
As before, consider a decision-maker $\text{NP}^*_{(1)}$ (resp. $\text{NP}^*_{(2)}$) who initially takes action $a=0$ (resp. $a=1$), resulting in System 1 (resp. System 2) starting from $(i,j-1,k+1,\ell)$ (resp. $(i,j-1,k,\ell+1)$), with all services about to begin. 
The definitions of the approximations $H_0^0$ and $H_{1,2}^0$ under small $\mu_0$ remain as given in Definition \ref{def:H0_and_H12}, with Assumption \ref{assum:deterministic} replaced by Assumption \ref{assum:deterministic_initial_block}.

Consider System 2 and the calculations for System 1 are analogous. It follows immediately from Statements \ref{det:amount_of_work} and Statement \ref{modified:downstream_blocking} in Assumption \ref{assum:deterministic_initial_block} that the time required to complete a service at Station $n$ equals $\frac{1}{\mu_n}$ for $n = 0,1$. For the initial set of patients at Station 2, the first $C^G$ ones begin collaborative service immediately and finish in $\frac{1}{\mu_2}$ time units, while any remaining pairs must wait in the queue. For the most recent arrival to the queue, the total remaining time in system is $\frac{\ell+1}{C^G \mu_2}$ (recall \eqref{eq:time_and_cost}). The system times for the remaining queued patients are computed analogously. After this initial group, all subsequent collaborative services proceed at a fixed duration of $\frac{1}{\mu_2}$ without delay. 
Recalling that the superscript denotes the case where $\mu_0$ is small, the following proposition provides explicit expressions for $H_0^0$ and $H_{1,2}^0$ in this setting. The proof appears in Appendix \ref{sec:proof_heuristic}.

\begin{proposition} \label{prop:heur_queue_in_collab_small_mu0}
    Consider $(i,j,k,\ell) \in \X_D$ where $\ell \geq C^G$ and Systems 1 and 2 as previously defined. The following hold for $H_0^0(i,j,k,\ell)$ and $H_{1,2}^0(i,j,k,\ell)$, as defined in Definition \ref{def:H0_and_H12}  (with Assumption \ref{assum:deterministic} replaced by Assumption \ref{assum:deterministic_initial_block}).
    {\small
    \begin{align}
        H_0^0(i,j,k,\ell)
        & =
        \begin{cases}
            \left( \left\lceil\frac{i-k}{C^p}\right\rceil \left(\frac{1}{\mu_1} - \frac{1}{\mu_2}\right)  - \sum_{r=C^G}^\ell \left\lceil\frac{i-k-r}{C^p}\right\rceil \frac{1}{C^G \mu_2}\right)\cdot h_0, & \text{ if } \frac{1}{\mu_1} \leq \frac{1}{\mu_2},\\
            \left( \left\lceil\frac{i-k-\ell'}{C^p}\right\rceil \left(\frac{1}{\mu_1} - \frac{\ell'+1}{C^G \mu_2}\right)  - \sum_{r=\ell'+1}^{\ell} \left\lceil\frac{i-k-r}{C^p}\right\rceil \frac{1}{C^G \mu_2}\right)\cdot h_0, & \text{ if } \frac{1}{\mu_2} < \frac{1}{\mu_1} \leq \frac{\ell+1}{C^G \mu_2},\\
            \left\lceil\frac{i-\ell}{C^p}\right\rceil \left(\frac{1}{\mu_1} - \frac{\ell+1}{C^G \mu_2}\right) \cdot h_0, & \text{ if } \frac{1}{\mu_1} > \frac{\ell+1}{C^G \mu_2}.
        \end{cases} \label{eq:small_mu0_H0}\\
        H_{1,2}^0(i,j,k,\ell) &= b_\ell = \frac{h_1}{\mu_1} - \frac{\ell+1}{C^G} \frac{h_2}{\mu_2}, \label{eq:small_mu0_H12}
    \end{align}
    }
    as provided in \eqref{eq:def_b_ell}. Here, $\ell' \in \N$ is defined in \eqref{eq:def_ell_prime}. 
\end{proposition}

We note that any non-idling control policy under which corresponding NP pairs in the two systems take identical actions at each decision epoch equivalently, under which patients in the same position of the initial queues in the two systems receive the same type of service, yields the same expressions for $H_0^0$ in \eqref{eq:small_mu0_H0} and $H_{1,2}^0$ in \eqref{eq:small_mu0_H12}. This restriction is natural when evaluating $H_0^0$ and $H_{1,2}^0$ induced by differing a single action, and the equivalence itself follows from the fact that the triage service operates on a sufficiently slower time scale (Statement \ref{modified:downstream_blocking}). Consequently, without loss of generality, we assume that both systems operate under the \emph{always-collaborative policy $\pi^{co\ell}$}.

With \eqref{eq:small_mu0_H0} and \eqref{eq:small_mu0_H12} established, substituting them into \eqref{eq:H_decompose} yields the expression for $H^0$ in \eqref{eq:small_mu0_H} for $\ell \geq C^G$ under the small-$\mu_0$ assumption.
As before, we can further simplify the formulation by replacing ceiling terms of the form $\left\lceil\frac{i-s}{C^p}\right\rceil$, where $s$ represents $k,\ell$, or $k+r$ depending on the expressions in \eqref{eq:small_mu0_H0}, with a linear counterpart $\frac{i}{C^p}$. This leads to the linear approximation $H^0_{\text{Lin}}$ in \eqref{eq:small_mu0_H0}. 
This linearization entails implications analogous to those discussed previously:
\begin{itemize}
    \item Combined with another linear function $H^\infty$ (introduced later), this yields a final approximation that remains linear and thus monotone, again implying a threshold-type policy via \eqref{eq:def_a_Hl}. 
    \item The upstream and downstream cost differences resulting from a single action remain clearly identifiable in the expression. The intercept $b_\ell$ represents the downstream component, which depends on the downstream parameters $\mu_1$, $h_1$, $\mu_2$, and $h_2$, as well as the state variable $\ell$, capturing the severity of queueing for collaborative service. The remaining term in \eqref{eq:small_mu0_H_linear} corresponds to the upstream component, which is linear in the number of upstream patients $i$. Its slope $c_\ell$ is proportional to $h_0$ and the immediate service time difference $\frac{1}{\mu_1}-  \frac{\ell+1}{C^G\mu_2}$.
    \item Similarly, as $i$ (the threshold $R_1$) grows large, the discrepancy between \eqref{eq:small_mu0_H} and \eqref{eq:small_mu0_H_linear} diminishes, reinforcing the accuracy and tractability of the linear approximation.
\end{itemize}
Next, we compute $H^\infty$ under Case \ref{case:large_mu0}, where triage completes almost instantaneously. As an approximation, we consider the case when $\mu_0 = \infty$, which effectively reduces the system to a single-stage queue studied in depth in \citep{lu2026balancingindependentcollaborativeservice}. Under instantaneous triage, each NP performs a single service per patient and immediately proceeds to the next upon completion of the current, making another decision at that point. Accordingly, the state space collapses to three dimensions:
\begin{align*}
    \tilde{\X}_{red} := \Big\{(i,k,\ell)\in \Z_+^3\Big{|} i=0, k+\ell < C^p \text{ or } i \geq 0, k+\ell = C^p\Big\}.
\end{align*}
\noindent
We use tildes to denote the value functions and related quantities in this three-dimensional context.
Following notations in \citep{lu2026balancingindependentcollaborativeservice}, we define for $(i,k,\ell) \in \tilde{\X}_{red}$ with $k \geq 1$:
\begin{align}
    \tilde{D}(i,k,\ell) := \tilde{v}(i,k,\ell) - \tilde{v}(i,k-1,\ell+1). \label{def:D_no_triage}
\end{align}

For $(i,j,k,\ell) \in \X_D$ with $\ell \geq C^G$, we now seek to relate the sign of $D^\infty(i,j,k,\ell)$ to that of the reduced-model cost difference $\tilde{D}$ in \eqref{def:D_no_triage}, since $\tilde{D}$ admits an effective linear approximation $\tilde{H}$ (\citep{lu2026balancingindependentcollaborativeservice}). We leverage this connection by introducing the auxiliary function below:
{\small
\begin{align}
    G(i,j,k,\ell) &:= \min_{m=0,1, \ldots, j-1}\Big\{\tilde{v}\big(i, k+1+(j-1-m), \ell+m\big)\Big\} \nonumber \\
    &\quad \qquad - \min_{m=0,1, \ldots, j-1}\Big\{\tilde{v}\big(i, k+(j-1-m), \ell+1+m\big)\Big\}. \label{eq:diff_of_min}
\end{align}
}
Then under fast triage, the allocation of all $j$ NP–patient pairs to downstream services is determined virtually instantaneously, resulting in $D^\infty(i,j,k,\ell) \approx G(i,j,k,\ell)$.
This approximation reduces the problem of characterizing the sign of $G$. 
The following proposition describes how the sign of $G(i,j,k,\ell)$ evolves with increasing $i$, and establishes its connection with the sign of $\tilde{D}$ in the three-dimensional (infinite-triage-rate) model. The proof is provided in Appendix \ref{sec:proof_heuristic}.
\begin{proposition} \label{prop:large_mu0_D_sign}
    Consider $(i,j,k,\ell) \in \X_D$ where $\ell \geq C^G$. The following holds for $G(i,j,k,\ell)$.
    \begin{enumerate}
        \item \label{state:large_mu0_no_decision}
        When $i=0$, we have
        {\small
        \begin{align}
            G(0,j,k,\ell) = \left\{
            \begin{array}{lcl}
                \frac{h_1}{\mu_1} - \frac{(\ell+1) h_2}{C^G\mu_2}, & \text{if} & \frac{h_1}{\mu_1} \leq \frac{(\ell+1) h_2}{C^G\mu_2}, \\
                0, & \text{if} & \frac{(\ell+1) h_2}{C^G\mu_2} < \frac{h_1}{\mu_1} \leq \frac{(\ell+j) h_2}{C^G\mu_2}, \\
                \frac{h_1}{\mu_1} - \frac{(\ell+j) h_2}{C^G\mu_2}, & \text{if} & \frac{h_1}{\mu_1} > \frac{(\ell+j) h_2}{C^G\mu_2}.
            \end{array}
            \right. \label{eq:large_mu0_D_no_decision}
        \end{align}
        }
        \item \label{state:large_mu0_monotone}
        For $j+k+\ell  = C^p$, $G(i,j,k,\ell)$ is monotone non-increasing in $i$, i.e., for all $i \geq 0$,
        \begin{align}
            G(i,j,k,\ell) - G(i+1,j,k,\ell)\geq 0. \label{eq:large_mu0_monotone}
        \end{align}
        Moreover, it eventually becomes negative as $i$ increases.
        \item \label{state:large_mu0_neg}
        If either $G(i,j,k,\ell) < 0$ or $\tilde{D}(i,k+j,\ell) < 0$, then $G(i,j,k,\ell) = \tilde{D}(i,k+j,\ell) < 0$.
        \item \label{state:large_mu0_pos}
        If either $G(i,j,k,\ell) > 0$ or $\tilde{D}(i,k+1,\ell+j-1) > 0$, then $G(i,j,k,\ell) = \tilde{D}(i,k+1,\ell+j-1) > 0$.
    \end{enumerate}
\end{proposition}
Note if $i$ satisfies both $\tilde{D}(i,k+j,\ell) \geq 0$ and $\tilde{D}(i,k+1,\ell+j-1) \leq 0$, then Statements \ref{state:large_mu0_neg} and \ref{state:large_mu0_pos} together (via the contrapositive) yield $G(i,j,k,\ell) = 0$. Such a value of $i$ may indeed exist, since $\tilde{D}(i,k+j,\ell) \geq \tilde{D}(i,k+1,\ell+j-1)$ (Proposition 3.11 in \citep{lu2026balancingindependentcollaborativeservice}). 
Consequently, the sign behavior of $G$ (and hence that of $D^\infty$ since $D^\infty \approx G$) is fully characterized through its connection to $\tilde{D}$ in the three-dimensional model by Proposition \ref{prop:large_mu0_D_sign}. 
Building on this, we now introduce our suggestion of $H^\infty$ under Case \ref{case:large_mu0}, which is divided into two subcases; with $i =0$ and $i \geq 1$.

\begin{enumerate} [label= \textbf{Subcase} (\alph*):, leftmargin=3\parindent]
    \item For $i=0$, the value of $G$ is explicitly given by \eqref{eq:large_mu0_D_no_decision}, and we define $H^\infty$ to match:
    {\small
    \begin{align}
        H^\infty(0,j,k,\ell) = \left\{
        \begin{array}{lcl}
            \frac{h_1}{\mu_1} - \frac{(\ell+1) h_2}{C^G\mu_2}, & \text{if} & \frac{h_1}{\mu_1} \leq \frac{(\ell+1) h_2}{C^G\mu_2}, \\
            0, & \text{if} & \frac{(\ell+1) h_2}{C^G\mu_2} < \frac{h_1}{\mu_1} \leq \frac{(\ell+j) h_2}{C^G\mu_2}, \\
            \frac{h_1}{\mu_1} - \frac{(\ell+j) h_2}{C^G\mu_2}, & \text{if} & \frac{h_1}{\mu_1} > \frac{(\ell+j) h_2}{C^G\mu_2}.
        \end{array}
        \right. \label{eq:large_mu0_H_no_decision}
    \end{align}
    }
    \item For $i \geq 1$, Statement \ref{state:large_mu0_monotone} in Proposition \ref{prop:large_mu0_D_sign} confirms that $G(i,j,k,\ell)$ is monotone non-increasing in $i$ and eventually becomes negative, suggesting a threshold-type policy. 
    Moreover, Statement \ref{state:large_mu0_neg} establishes that $G(i,j,k,\ell)$ coincides with $\tilde{D}(i,k+j,\ell)$ whenever either is negative. 
    Since our interest lies in the threshold at or beyond which the sign of $D^\infty \approx G$ stabilizes, we may, without loss of generality, assume that $G(i,j,k,\ell) = \tilde{D}(i,k+j,\ell)$ holds for all $i \geq 0$, rather than only for values where $G$ is negative.  
    
    \citet{lu2026balancingindependentcollaborativeservice} propose a linear approximation $\tilde{H}(i, k+j, \ell)$ of $\tilde{D}(i, k+j, \ell)$ with strong analytical justification and empirical accuracy (e.g., Lemma 4.5): 
    \begin{align*}
        \tilde{H}(i,k+j,\ell) = (i - y_\ell)c' + b', 
    \end{align*}
   where $b'$ and $c'$ are defined in \eqref{def:b'_and_c'}, and $y_\ell$ is given in \eqref{def:y}. Using the equivalence that $D^\infty(i,j,k,\ell) \approx G(i,j,k,\ell) = \tilde{D}(i,k+j,\ell)$, we let $H^{\infty}(i, j, k, \ell) = \tilde{H}(i,k+j,\ell)$, leading to the expression in \eqref{eq:large_mu0_H_general}.
    This construction also admits the decomposition in \eqref{eq:H_decompose}, capturing the upstream and downstream components of approximation.
\end{enumerate}

Building on the previous analyses, we now summarize the definitions of $H(i,j,k,\ell)$, where $(i,j,k,\ell) \in \X_D$ with $\ell \geq C^G$, distinguishing between the cases $i = 0$ and $i \geq 1$:
\begin{itemize}
    \item When $i = 0$ (second case in \eqref{eq:def_H_base}), the expression \eqref{eq:small_mu0_H} gives $H^0(0,j,k,\ell) = b_\ell$. Combining this with the expression of $H^\infty$ in \eqref{eq:large_mu0_H_no_decision}, the weighted form in \eqref{eq:convex_comb} becomes:
    {\small
    \begin{align*}
        &H(0,j,k,\ell) =
        \begin{cases}
            b_\ell  & \text{ if } b_\ell \leq 0,\\
            (1-w_{k,\ell}) \cdot b_\ell & \text{ if } b_\ell > 0 \text{ and } b_{\ell+j-1} \leq 0,\\
            w_{k,\ell} \cdot b_{\ell+j-1} + (1-w_{k,\ell}) \cdot b_{\ell} & \text{ if } b_{\ell+j-1} > 0.\\
        \end{cases} 
    \end{align*}
    }
   However, since only the sign of $H$ is of interest, this simplifies to $H(0,j,k,\ell) = b_\ell$ by
    leveraging the fact that $b_{\ell+j-1} \leq b_\ell$.
    This coincides with Proposition \ref{prop:noncollab}, which implies that when $b_\ell \leq 0$ (and $c_\ell \leq 0$), we should have $H \leq 0$, consistent with $D\leq 0$.
    \item When $i \geq 1$:
    \begin{itemize}
        \item If $c_\ell \leq 0$ and $b_\ell \leq 0$ (the first case in \eqref{def:H_linear}), we set $H = -1 \leq 0$, aligning with $D \leq 0$.
        \item If either $b_\ell > 0$ or $c_\ell > 0$, we apply the piecewise linear approximation $H^0$ from \eqref{eq:small_mu0_H} and combine it with $H^\infty$ in \eqref{eq:large_mu0_H_general} using \eqref{eq:convex_comb}, yielding the final definition in \eqref{eq:def_H}.
    \end{itemize}
\end{itemize}
The linear approximation $H_{\text{Lin}}$, defined in \eqref{eq:def_H_linear} follows a similar construction, with $H^0$ replaced by its linearized counterpart $H^0_{\text{Lin}}$.

\section{Numerical Analysis} \label{sec:numerical}
This section evaluates the accuracy and robustness of the proposed heuristics through a numerical study. The parameter settings are summarized in Table \ref{table:params}, where we test all configurations that satisfy Assumption \ref{assm:preferred_collab}. Since optimal policies are invariant under positive scaling of holding costs (when service rates are fixed), we without loss of generality, fix $h_1 = 1$. Similarly, we normalize $\mu_1 = 4$ per hour, corresponding to an average non-collaborative service time of 15 minutes. The collaborative service rate $\mu_2$ varies from 0.4 to 2.5 times of $\mu_1$ to reflect a range of practical scenarios. Although triage is typically faster than downstream services, we allow $\mu_0$ to vary between $\frac{1}{4}\mu_1$ to $4\mu_1$ to illustrate the effect of the weighting parameter $w_{k,\ell}$ in \eqref{eq:def_w}.
{\small
\begin{table}[htbp]
\caption{Parameter configurations for numerical test with $h_1=1$ and $\mu_1 = 4$.}
\label{table:params}
\centering
\begin{tabular}{cc}
\hline
\multicolumn{1}{c}{\textbf{Parameters}} & \multicolumn{1}{c}{\textbf{Values}} \\ \hline
$h_0$                               & 0.05, 0.1, 0.2, 0.5, 1  \\
$h_2$                               & 0.1, 0.2, 0.5, 1, 1.5           \\
$\mu_0$                              & 1, 2, 4, 8, 16    \\
$\mu_2$                              & 1.6, 2, 3.2, 4, 5, 8, 10    \\ \hline    
\end{tabular}
\end{table}
}
We refer to our proposed heuristic policy induced by $H$ (resp. $H_{\text{Lin}}$) through \eqref{eq:def_a_H} (resp. \eqref{eq:def_a_Hl}) as $\pi'$ (resp. $\pi'_{\text{Lin}}$), to distinguish from baseline policies commonly adopted in practice, defined as follows for all $(i,j,k,\ell) \in \X_D$:
\begin{enumerate}
[leftmargin=2.5cm, label=\textbf{Policy} $\pi_\arabic*$.]
    \item Always select non-collaborative service: $a=0$ for all $i$.
    \item Choose collaborative care if and only if the upstream queue length is below the threshold $10$: $a=\mathbb{I}\{i < 10\}$.
    \item Always select collaborative service: $a=1$ for all $i$.
    \item Seek collaboration only when collaborative service can begin immediately without delay: $a = \mathbb{I}\{\ell < C^G\}$.
\end{enumerate}
For states $(i_0,j_0,k_0,\ell_0) \in \X_D$ with $i_0=20$, we evaluate the performance usingthe relative error with respect to the optimal value:
\begin{align*}
    err = \frac{v^{\pi}(i_0,j_0,k_0,\ell_0) - v(i_0,j_0,k_0,\ell_0)}{v(i_0,j_0,k_0,\ell_0)}, 
\end{align*}where $v^{\pi}$ denotes the value function under policy $\pi$, and $v$ is the optimal value function. 

The numerical results, summarized in Tables \ref{table:cmp_state20_larger_mu1}-\ref{table:cmp_state20_larger_mu2}, demonstrate that the proposed heuristic policies $\pi'$ and $\pi'_{\text{Lin}}$ achieve near-optimal performance. Over a wide range of parameters, they incur costs within 5\% of the optimal, with an average error below 0.2\%, highlighting strong accuracy and robustness. In contrast, the benchmark heuristics $\pi_m$ for $m = 1,2, \ldots 4$, exhibit considerable sensitivity to parameter variations and can exceed the optimal cost by more than 100\%. Notably, the simplified heuristic $\pi'_{\text{Lin}}$ performs comparably to $\pi'$ and is slightly better in most cases, underscoring the practical effectiveness of the linear approximation despite its simplicity.

{\small
\begin{table}
    \caption{Comparison of relative errors $(\%)$ for states $(20,j_0,k_0,\ell_0)$ when $\mu_1\geq\mu_2$.}
    \centering
    \renewcommand{\arraystretch}{1.0}  
    \label{table:cmp_state20_larger_mu1}
    \begin{tabular}{|c|r|r|r|r|r|r|r|}
        \hline
        \multicolumn{2}{|c|}{} & \multicolumn{1}{c|}{$C_1 = 2$} & \multicolumn{2}{c|}{$C_1 = 3$} & \multicolumn{3}{c|}{$C_1 = 4$} \\ \cline{3-8}
        \multicolumn{2}{|c|}{} & $C_2 = 1$ & $C_2 = 1$ & $C_2 = 2$ & $C_2 = 1$ & $C_2 = 2$ & $C_2 = 3$ \\
        \hline

        \multirow{3}{*}{\textbf{Policy $\pi'$}} & Max error & 0.25 & 0.35 & 0.43 & 0.77 & 0.31 & 1.15 \\
                                                 & Avg error & 0.01 & 0.03 & 0.03 & 0.12 & 0.05 & 0.06 \\
                                                 & Std error & 0.03 & 0.05 & 0.06 & 0.12 & 0.06 & 0.13 \\
        \hline

        \multirow{3}{*}{\textbf{Policy $\pi'_{\text{Lin}}$}} & Max error & 0.25 & 0.35 & 0.36 & 0.74 & 0.57 & 0.91 \\
                                                 & Avg error & 0.01 & 0.03 & 0.03 & 0.11 & 0.04 & 0.05 \\
                                                 & Std error & 0.03 & 0.04 & 0.05 & 0.11 & 0.06 & 0.10 \\
        \hline
        
        \multirow{3}{*}{\textbf{Policy $\pi_1$}}  & Max error & 94.95 & 81.77 & 213.14 & 76.24 & 190.58 & 300.50 \\
                                                 & Avg error & 12.81 & 12.26 & 26.17 & 11.34 & 25.86 & 36.69 \\
                                                 & Std error & 17.25 & 14.85 & 34.40 & 12.86 & 31.87 & 45.38 \\
        \hline

        \multirow{3}{*}{\textbf{Policy $\pi_2$}}  & Max error & 43.81 & 57.54 & 111.51 & 64.09 & 96.36 & 164.34 \\
                                                 & Avg error & 12.10 & 16.41 & 12.55 & 20.41 & 12.49 & 16.04 \\
                                                 & Std error & 8.50  & 9.94  & 15.88 & 12.44 & 13.25 & 21.94 \\
        \hline

        \multirow{3}{*}{\textbf{Policy $\pi_3$}}  & Max error & 246.79 & 376.89 & 145.43 & 487.64 & 200.36 & 107.05 \\
                                                 & Avg error & 41.17  & 68.07  & 14.88  & 92.05  & 23.32  & 7.83 \\
                                                 & Std error & 45.08  & 64.65  & 23.33  & 78.93  & 31.26  & 15.28 \\
        \hline

        \multirow{3}{*}{\textbf{Policy $\pi_4$}}  & Max error & 29.83  & 25.28  & 50.01  & 22.73  & 64.01  & 55.17 \\
                                                 & Avg error & 5.32   & 3.52   & 6.77   & 2.82   & 6.17   & 6.56 \\
                                                 & Std error & 6.65   & 4.47   & 8.83   & 3.85   & 8.59   & 8.82 \\
        \hline
    \end{tabular}
\end{table}

\begin{table}[htbp]
    \caption{Comparison of relative errors $(\%)$ for states $(20,j_0,k_0,\ell_0)$ when $\mu_1<\mu_2$.}
    \centering
    \renewcommand{\arraystretch}{1.0}
    \label{table:cmp_state20_larger_mu2}
    \begin{tabular}{|c|r|r|r|r|r|r|r|}
        \hline
        \multicolumn{2}{|c|}{} & \multicolumn{1}{c|}{$C_1 = 2$} & \multicolumn{2}{c|}{$C_1 = 3$} & \multicolumn{3}{c|}{$C_1 = 4$} \\ \cline{3-8}
        \multicolumn{2}{|c|}{} & $C_2 = 1$ & $C_2 = 1$ & $C_2 = 2$ & $C_2 = 1$ & $C_2 = 2$ & $C_2 = 3$ \\
        \hline

        \multirow{3}{*}{\textbf{Policy $\pi'$}} & Max error & 2.01 & 1.82 & 4.85 & 2.30 & 2.70 & 2.66 \\
                                                 & Avg error & 0.08 & 0.10 & 0.07 & 0.17 & 0.09 & 0.03 \\
                                                 & Std error & 0.26 & 0.21 & 0.36 & 0.31 & 0.29 & 0.20 \\
        \hline

        \multirow{3}{*}{\textbf{Policy $\pi'_{\text{Lin}}$}} & Max error & 2.01 & 1.77 & 4.85 & 2.28 & 2.70 & 2.66 \\
                                                 & Avg error & 0.08 & 0.09 & 0.06 & 0.16 & 0.09 & 0.03 \\
                                                 & Std error & 0.26 & 0.21 & 0.35 & 0.30 & 0.28 & 0.19 \\
        \hline

        \multirow{3}{*}{\textbf{Policy $\pi_1$}}  & Max error & 348.39 & 307.42 & 594.06 & 275.73 & 587.44 & 760.19 \\
                                                 & Avg error & 49.32  & 43.57  & 73.53  & 38.68  & 71.19  & 85.45 \\
                                                 & Std error & 48.46  & 44.50  & 76.21  & 40.42  & 74.95  & 90.43 \\
        \hline

        \multirow{3}{*}{\textbf{Policy $\pi_2$}}  & Max error & 193.53 & 165.53 & 337.53 & 144.66 & 332.50 & 437.46 \\
                                                 & Avg error & 29.86  & 26.59  & 41.28  & 25.07  & 37.68  & 46.19 \\
                                                 & Std error & 25.12  & 20.63  & 40.45  & 17.81  & 37.80  & 47.60 \\
        \hline

        \multirow{3}{*}{\textbf{Policy $\pi_3$}}  & Max error & 53.24  & 120.31 & 18.32  & 189.08 & 47.19  & 8.08 \\
                                                 & Avg error & 5.64   & 15.90  & 0.84   & 28.01  & 3.14   & 0.20 \\
                                                 & Std error & 9.87   & 22.88  & 2.54   & 36.36  & 7.13   & 0.89 \\
        \hline

        \multirow{3}{*}{\textbf{Policy $\pi_4$}}  & Max error & 88.41  & 116.64 & 92.42  & 120.02 & 151.36 & 78.80 \\
                                                 & Avg error & 5.33   & 7.19   & 5.29   & 7.79   & 8.55   & 3.79 \\
                                                 & Std error & 11.46  & 15.17  & 10.72  & 15.99  & 17.09  & 8.23 \\
        \hline
    \end{tabular}
\end{table}
}
For ease of terminology, we say that a policy is of \textbf{collaborative (non-collaborative) threshold type with threshold} $T$ if it selects collaborative (non-collaborative) care when the initial upstream queue length is below $T$, and chooses the opposite action otherwise. Clearly, each $\pi_m$ for $m=1,2, \ldots, 4$, falls into this category with a collaborative threshold $T_j$, where $T_1 = 0$, $T_2 = 10$, $T_3 = \infty$ are constants independent of system parameters or state variables. In contrast, $T_4$ depends on the state variable $\ell$: $T_4 = \infty$ if $\ell < C^G$, and $T_4 = 0$ otherwise. On the other hand, neither $\pi'$ nor $\pi'_{\text{Lin}}$ can be fully captured by this threshold-type classification. For example, when $c_\ell = \frac{1}{\mu_1} - \frac{\ell+1}{C^G \mu_2} > 0$, the policy $\pi'$ induced by $H$ via \eqref{eq:def_a_H} may not be monotone in $i$, as shown in the second case of \eqref{eq:def_H}. While $H_{\text{Lin}}$ is linear and thus induces a monotone policy, whether it corresponds to a collaborative or non-collaborative threshold type depends on the sign of its slope $\tilde{c}_{k,\ell}$. In other regions of the parameter and state spaces, both $\pi'$ and $\pi'_{\text{Lin}}$ do conform to the collaborative threshold type, with thresholds that may be finite, zero, or infinite.

An immediate observation is that performance tends to improve when $\mu_1 \geq \mu_2$ compared to the case where $\mu_1 < \mu_2$. This is largely because, under $\mu_1 \geq \mu_2$, the structure of the optimal policy becomes more predictable: there exists a threshold (possibly infinite) beyond which non-collaborative service is always preferred. In other words, the sign of the cost difference $D$ defined in \eqref{def:D} are known for sufficiently large $i$. 
In particular, when $\mu_1 > \mu_2$, Statement \ref{state:fast_noncollab} in Theorem \ref{thm:lowcost_noncollab} confirms that such threshold is finite. Although this does not imply that the optimal control is globally monotone in $i$ (recall $D$ may not be monotone non-increasing), our numerical results consistently show that the optimal policy is in fact of collaborative threshold type. 
When $\mu_1 = \mu_2$, similar structural insights hold for $\ell \geq C^G$ (\ref{state:equal_rate_queue_in_collab} in Theorem \ref{thm:lowcost_collab}), while for $\ell < C^G$, Statement \ref{state:equal_rate_no_queue_in_collab} 
shows that the optimal policy is of collaborative threshold type with an infinite threshold. 

Although the optimal control under $\mu_1 \geq \mu_2$ often exhibits a collaborative threshold-type structure, the threshold is highly sensitive to whether collaborative service can begin immediately ($\ell < C^G$) or is delayed by downstream blocking ($\ell\geq C^G$). 
In the latter case, the downstream blocking prolongs the time required for an NP to return to triage compared to the non-collaborative route (Station 1). This additional delay amplifies the upstream holding cost incurred by upstream waiting patients, making non-collaboration more favorable even at smaller upstream queue lengths. Consequently, the threshold beyond which collaboration is no longer optimal is expected to be lower when $\ell \geq C^G$ than when $\ell < C^G$. 
Accounting for this distinction, Policy $\pi_4$ yields the smallest error among the four alternative heuristics. However, it imposes extreme rules by avoiding both idling GPs and delayed service by assigning threshold $T = \infty$ when $\ell < C^G$ and $T = 0$ when $\ell \geq C^G$.

When $\mu_1 < \mu_2$ and $\ell < C^G$, Statement \ref{state:fast_collab_no_queue_in_collab} in Theorem \ref{thm:lowcost_collab} implies that an optimal decision-maker always seeks collaborative care. This aligns with the decisions made by Policies $\pi_3$ and $\pi_4$, explaining their superior performance compared to the other two benchmarks. The same reasoning accounts for the observed performance improvement from $\pi_1$ to $\pi_3$, as these policies increasingly align with the optimal action $a=1$ when $\ell < C^G$. 
In contrast, when $\ell \geq C^G$, the structure of the optimal policy becomes more complex, as discussed in Section \ref{sec:main_results}, due to the nuanced interplay between upstream and downstream service rates. Nonetheless, the weighted convex combination in \eqref{eq:convex_comb} effectively captures the relative contributions of the two approximations corresponding to Cases \ref{case:small_mu0} and \ref{case:large_mu0} outlined in Section \ref{sec:H_queue_in_collab} and consistently yields accurate heuristic decisions.


\section{Conclusion} \label{sec:conclusions}
To address an operational decision-making problem faced by MSPs in clinical settings, this work studied a two-stage queueing system modeled as an MDP with holding costs, where the objective is to clear all jobs and return the system to a normative state. After an initial screening, NPs must decide whether to serve each patient independently or in collaboration with a GP. Drawing on both analytical insights and empirical observations, we developed simple yet effective heuristics to guide practical decision making. A comprehensive numerical study demonstrates that the proposed heuristics are both accurate and robust, consistently outperforming existing benchmark approaches.

\bibliographystyle{unsrtnat}
\bibliography{policy_ref}

@misc{lu2026balancingindependentcollaborativeservice,
      title={Balancing Independent and Collaborative Service}, 
      author={Shuwen Lu and Mark E. Lewis and Jamol Pender},
      year={2026},
      eprint={2601.13586},
      archivePrefix={arXiv},
      primaryClass={math.OC},
      url={https://arxiv.org/abs/2601.13586}, 
}

@misc{lu2026controlpoliciestwostagequeueing,
      title={Control policies for a two-stage queueing system with parallel and single server options}, 
      author={Shuwen Lu and Jamol Pender and Mark E. Lewis},
      year={2026},
      eprint={2601.13576},
      archivePrefix={arXiv},
      primaryClass={math.OC},
      url={https://arxiv.org/abs/2601.13576}, 
}

@article{loeb2020departmental,
  title={Departmental experience and lessons learned with accelerated introduction of telemedicine during the COVID-19 crisis},
  author={Loeb, Alexander E and Rao, Sandesh S and Ficke, James R and Morris, Carol D and Riley III, Lee H and Levin, Adam S},
  journal={The Journal of the American Academy of Orthopaedic Surgeons},
  year={2020},
  publisher={Wolters Kluwer Health}
}

@article{calton2019top,
  title={Top ten tips palliative care clinicians should know about telepalliative care},
  author={Calton, Brook Anne and Rabow, Michael W and Branagan, Linda and Dionne-Odom, James Nicholas and Parker Oliver, Debra and Bakitas, Marie A and Fratkin, Michael D and Lustbader, Dana and Jones, Christopher A and Ritchie, Christine S},
  journal={Journal of palliative medicine},
  volume={22},
  number={8},
  pages={981--985},
  year={2019},
  publisher={Mary Ann Liebert, Inc., publishers 140 Huguenot Street, 3rd Floor New~…}
}

@article{contreras2020telemedicine,
  title={Telemedicine: patient-provider clinical engagement during the COVID-19 pandemic and beyond},
  author={Contreras, Carlo M and Metzger, Gregory A and Beane, Joal D and Dedhia, Priya H and Ejaz, Aslam and Pawlik, Timothy M},
  journal={Journal of Gastrointestinal Surgery},
  volume={24},
  number={7},
  pages={1692--1697},
  year={2020},
  publisher={Springer}
}

@article{calton2020telemedicine,
  title={Telemedicine in the time of coronavirus},
  author={Calton, Brook and Abedini, Nauzley and Fratkin, Michael},
  journal={Journal of Pain and Symptom Management},
  volume={60},
  number={1},
  pages={e12--e14},
  year={2020},
  publisher={Elsevier}
}

@inproceedings{lu2024physician,
  title={Physician Staffing in Telemedicine: A Simulation-Based Approach for a Network of CVS Minute Clinics},
  author={Lu, Shuwen and Lewis, Mark E and Pender, Jamol},
  booktitle={2024 Winter Simulation Conference (WSC)},
  pages={870--881},
  year={2024},
  organization={IEEE}
}

@book{puterman2014markov,
  title={Markov decision processes: discrete stochastic dynamic programming},
  author={Puterman, Martin L},
  year={2014},
  publisher={John Wiley \& Sons}
}

@article{greenglass2003reactions,
  title={Reactions to increased workload: Effects on professional efficacy of nurses},
  author={Greenglass, Esther R and Burke, Ronald J and Moore, Kathleen A},
  journal={Applied psychology},
  volume={52},
  number={4},
  pages={580--597},
  year={2003},
  publisher={Wiley Online Library}
}

@article{lang2004nurse,
  title={Nurse--patient ratios: a systematic review on the effects of nurse staffing on patient, nurse employee, and hospital outcomes},
  author={Lang, Thomas A and Hodge, Margaret and Olson, Valerie and Romano, Patrick S and Kravitz, Richard L},
  journal={JONA: The Journal of Nursing Administration},
  volume={34},
  number={7},
  pages={326--337},
  year={2004},
  publisher={LWW}
}

@article{massey1984operator,
  title={An operator-analytic approach to the {J}ackson network},
  author={Massey, William A},
  journal={Journal of Applied Probability},
  volume={21},
  number={2},
  pages={379--393},
  year={1984},
  publisher={Cambridge University Press}
}

@article{massey1986family,
  title={A family of bounds for the transient behavior of a {J}ackson network},
  author={Massey, William A},
  journal={Journal of applied probability},
  volume={23},
  number={2},
  pages={543--549},
  year={1986},
  publisher={Cambridge University Press}
}

@article{massey1987calculating,
  title={Calculating exit times for series {J}ackson networks},
  author={Massey, William A},
  journal={Journal of applied probability},
  volume={24},
  number={1},
  pages={226--234},
  year={1987},
  publisher={Cambridge University Press}
}

@article{badian2021optimal,
  title={Optimal control policies for an M/M/1 queue with a removable server and dynamic service rates},
  author={Badian-Pessot, Pamela and Lewis, Mark E and Down, Douglas G},
  journal={Probability in the Engineering and Informational Sciences},
  volume={35},
  number={2},
  pages={189--209},
  year={2021},
  publisher={Cambridge University Press}
}

@article{liu1992optimal,
  title={On optimal polling policies},
  author={Liu, Zhen and Nain, Philippe and Towsley, Don},
  journal={Queueing Systems},
  volume={11},
  number={1},
  pages={59--83},
  year={1992},
  publisher={Springer}
}

@article{nain1994optimal,
  title={Optimal scheduling in a machine with stochastic varying processing rate},
  author={Nain, Philippe and Towsley, Don},
  journal={IEEE Transactions on Automatic Control},
  volume={39},
  number={9},
  pages={1853--1855},
  year={1994},
  publisher={IEEE}
}

@article{liu1995sample,
  title={Sample path methods in the control of queues},
  author={Liu, Zhen and Nain, Philippe and Towsley, Don},
  journal={Queueing Systems},
  volume={21},
  number={3},
  pages={293--335},
  year={1995},
  publisher={Springer}
}

@article{paschalidis2000congestion,
  title={Congestion-dependent pricing of network services},
  author={Paschalidis, I Ch and Tsitsiklis, John N},
  journal={IEEE/ACM transactions on networking},
  volume={8},
  number={2},
  pages={171--184},
  year={2000},
  publisher={IEEE}
}

@article{koole1998structural,
  title={Structural results for the control of queueing systems using event-based dynamic programming},
  author={Koole, Ger},
  journal={Queueing systems},
  volume={30},
  number={3},
  pages={323--339},
  year={1998},
  publisher={Springer}
}

@article{chong2018two,
  title={Two-class routing with admission control and strict priorities},
  author={Chong, Kenneth C and Henderson, Shane G and Lewis, Mark E},
  journal={Probability in the Engineering and Informational Sciences},
  volume={32},
  number={2},
  pages={163--178},
  year={2018},
  publisher={Cambridge University Press}
}

@article{yoon2004optimal,
  title={Optimal pricing and admission control in a queueing system with periodically varying parameters},
  author={Yoon, Seunghwan and Lewis, Mark E},
  journal={Queueing Systems},
  volume={47},
  number={3},
  pages={177--199},
  year={2004},
  publisher={Springer}
}

@article{kumar2013dynamic,
  title={Dynamic service rate control for a single-server queue with Markov-modulated arrivals},
  author={Kumar, Ravi and Lewis, Mark E and Topaloglu, Huseyin},
  journal={Naval Research Logistics (NRL)},
  volume={60},
  number={8},
  pages={661--677},
  year={2013},
  publisher={Wiley Online Library}
}

@article{feinberg2007optimality,
  title={Optimality inequalities for average cost Markov decision processes and the stochastic cash balance problem},
  author={Feinberg, Eugene A and Lewis, Mark E},
  journal={Mathematics of Operations Research},
  volume={32},
  number={4},
  pages={769--783},
  year={2007},
  publisher={INFORMS}
}

@article{argon2008scheduling,
  title={Scheduling impatient jobs in a clearing system with insights on patient triage in mass casualty incidents},
  author={Argon, Nilay Tanik and Ziya, Serhan and Righter, Rhonda},
  journal={Probability in the Engineering and Informational Sciences},
  volume={22},
  number={3},
  pages={301--332},
  year={2008},
  publisher={Cambridge University Press}
}

@article{wu2008heuristics,
  title={Heuristics for allocation of reconfigurable resources in a serial line with reliability considerations},
  author={Wu, Cheng-Hung and Down, Douglas G and Lewis, Mark E},
  journal={Iie Transactions},
  volume={40},
  number={6},
  pages={595--611},
  year={2008},
  publisher={Taylor \& Francis}
}

@article{zayas2016dynamic,
  title={Dynamic control of a tandem system with abandonments},
  author={Zayas-Cab{\'a}n, Gabriel and Xie, Jingui and Green, Linda V and Lewis, Mark E},
  journal={Queueing Systems},
  volume={84},
  number={3},
  pages={279--293},
  year={2016},
  publisher={Springer}
}

@article{hordijk1992assignment,
  title={On the assignment of customers to parallel queues},
  author={Hordijk, Arie and Koole, Ger},
  journal={Probability in the Engineering and Informational Sciences},
  volume={6},
  number={4},
  pages={495--511},
  year={1992},
  publisher={Cambridge University Press}
}

@article{down2006dynamic,
  title={Dynamic load balancing in parallel queueing systems: Stability and optimal control},
  author={Down, Douglas G and Lewis, Mark E},
  journal={European Journal of Operational Research},
  volume={168},
  number={2},
  pages={509--519},
  year={2006},
  publisher={Elsevier}
}

@article{nain1989interchange,
  title={Interchange arguments for classical scheduling problems in queues},
  author={Nain, Philippe},
  journal={Systems \& control letters},
  volume={12},
  number={2},
  pages={177--184},
  year={1989},
  publisher={Elsevier}
}

@article{huang2022dynamically,
  title={Dynamically scheduling and maintaining a flexible server},
  author={Huang, Jefferson and Down, Douglas G and Lewis, Mark E and Wu, Cheng-Hung},
  journal={Naval Research Logistics (NRL)},
  volume={69},
  number={2},
  pages={223--240},
  year={2022},
  publisher={Wiley Online Library}
}

@article{harrison1998heavy,
  title={Heavy traffic analysis of a system with parallel servers: asymptotic optimality of discrete-review policies},
  author={Harrison, J Michael},
  journal={The Annals of Applied Probability},
  volume={8},
  number={3},
  pages={822--848},
  year={1998},
  publisher={Institute of Mathematical Statistics}
}

@article{bell2001dynamic,
  title={Dynamic scheduling of a system with two parallel servers in heavy traffic with resource pooling: Asymptotic optimality of a threshold policy},
  author={Bell, Steven L and Williams, Ruth J},
  journal={The Annals of Applied Probability},
  volume={11},
  number={3},
  pages={608--649},
  year={2001},
  publisher={Institute of Mathematical Statistics}
}

@article{ahn2004optimal,
  title={Optimal control of a flexible server},
  author={Ahn, Hyun-Soo and Duenyas, Izak and Zhang, Rachel Q},
  journal={Advances in Applied Probability},
  volume={36},
  number={1},
  pages={139--170},
  year={2004},
  publisher={Cambridge University Press}
}

@article{down2010n,
  title={The N-network model with upgrades},
  author={Down, Douglas G and Lewis, Mark E},
  journal={Probability in the Engineering and Informational Sciences},
  volume={24},
  number={2},
  pages={171--200},
  year={2010},
  publisher={Cambridge University Press}
}

@article{niyirora2016optimal,
  title={Optimal staffing in nonstationary service centers with constraints},
  author={Niyirora, Jerome and Pender, Jamol},
  journal={Naval Research Logistics (NRL)},
  volume={63},
  number={8},
  pages={615--630},
  year={2016},
  publisher={Wiley Online Library}
}

@article{pender2017approximating,
  title={Approximating and stabilizing dynamic rate Jackson networks with abandonment},
  author={Pender, Jamol and Massey, William A},
  journal={Probability in the Engineering and Informational Sciences},
  volume={31},
  number={1},
  pages={1--42},
  year={2017},
  publisher={Cambridge University Press}
}

@article{hordijk1992shortest,
  title={On the shortest queue policy for the tandem parallel queue},
  author={Hordijk, Arie and Koole, Ger},
  journal={Probability in the Engineering and Informational Sciences},
  volume={6},
  number={1},
  pages={63--79},
  year={1992},
  publisher={Cambridge University Press}
}

@article{ahn1999optimal,
  title={Optimal stochastic scheduling of a two-stage tandem queue with parallel servers},
  author={Ahn, Hyun-Soo and Duenyas, Izak and Zhang, Rachel Q},
  journal={Advances in Applied Probability},
  volume={31},
  number={4},
  pages={1095--1117},
  year={1999},
  publisher={Cambridge University Press}
}

@article{wu2006dynamic,
  title={Dynamic allocation of reconfigurable resources ina two-stage tandem queueing system with reliability considerations},
  author={Wu, Cheng-Hung and Lewis, Mark E and Veatch, Michael},
  journal={IEEE Transactions on Automatic Control},
  volume={51},
  number={2},
  pages={309--314},
  year={2006},
  publisher={IEEE}
}

@article{pandelis1994optimal,
  title={Optimal multiserver stochastic scheduling of two interconnected priority queues},
  author={Pandelis, Dimitrios G and Teneketzis, Demosthenis},
  journal={Advances in Applied Probability},
  volume={26},
  number={1},
  pages={258--279},
  year={1994},
  publisher={Cambridge University Press}
}

@article{pandelis2008optimal,
  title={Optimal control of flexible servers in two tandem queues with operating costs},
  author={Pandelis, Dimitrios G},
  journal={Probability in the Engineering and Informational Sciences},
  volume={22},
  number={1},
  pages={107--131},
  year={2008},
  publisher={Cambridge University Press}
}

@article{dobson2012queueing,
  title={A queueing model to evaluate the impact of patient “batching” on throughput and flow time in a medical teaching facility},
  author={Dobson, Gregory and Lee, Hsiao-Hui and Sainathan, Arvind and Tilson, Vera},
  journal={Manufacturing \& Service Operations Management},
  volume={14},
  number={4},
  pages={584--599},
  year={2012},
  publisher={INFORMS}
}

@article{andradottir2021optimizing,
  title={Optimizing the interaction between residents and attending physicians},
  author={Andrad{\'o}ttir, Sigr{\'u}n and Ayhan, Hayriye},
  journal={European Journal of Operational Research},
  volume={290},
  number={1},
  pages={210--218},
  year={2021},
  publisher={Elsevier}
}

@article{yu2024optimal,
  title={Optimal control of supervisors balancing individual and joint responsibilities},
  author={Yu, Zhuoting and Andrad{\'o}ttir, Sigr{\'u}n and Ayhan, Hayriye},
  journal={Probability in the Engineering and Informational Sciences},
  volume={38},
  number={1},
  pages={130--149},
  year={2024},
  publisher={Cambridge University Press}
}

\newpage
\appendix
\newgeometry{
  left=0.8in,
  right=0.8in,
  top=0.8in,
  bottom=0.8in
}
\ls{1.2}
\begingroup
\small
\section{Online Appendix} \label{sec:appendix}
\setlength{\abovedisplayskip}{0pt plus 2pt minus 2pt}
\setlength{\belowdisplayskip}{0pt plus 2pt minus 2pt}
\setlength{\abovedisplayshortskip}{0pt plus 2pt minus 2pt}
\setlength{\belowdisplayshortskip}{0pt plus 2pt minus 2pt}
\subsection{Preliminaries for supporting results} \label{sec:proof_prelim}
We present several preliminary results to support the proof of our main results. The following lemma, which holds independent of Assumption \ref{assm:preferred_collab}, bounds the difference between the cost of starting in various states, helping determine the sign of the difference $D$ defined in \eqref{def:D}. 

\begin{lemma} \label{lemma:diff}
The following inequalities hold:
    \begin{enumerate}
        \item \label{state:lemma_diff_one_more_base}
        Consider $(0,j,k, \ell) \in \X$, where $j+k+\ell < C^p$.  We have
            \begin{equation} \label{eq:one_more_noncollab}
                v(0,j,k+1,\ell) - v(0,j,k,\ell) \geq \frac{h_1}{\mu_1},
            \end{equation}
            \begin{equation} \label{eq:one_more_collab}
                v(0,j,k,\ell+1) - v(0,j,k,\ell) \geq \frac{h_2}{\mu_2}.
            \end{equation}
            \begin{equation} \label{eq:one_more_triage}
                v(0,j+1,k,\ell) - v(0,j,k,\ell) \geq \frac{h_0}{\mu_0} + \min\left\{\frac{h_1}{\mu_1},\frac{h_2}{\mu_2}\right\}.
            \end{equation}
        \item \label{state:lemma_diff_one_more}
        Consider $(i,j,k,\ell) \in \X$ where $i \geq 0$ and $j+k+\ell = C^p$,
            \begin{align}  
                \lefteqn{v(i+1,j,k,\ell) - v(i,j,k,\ell)}& \nonumber \\
                & \quad \geq \min\left\{\frac{h_1}{\mu_1},\frac{h_2}{\mu_2}\right\} + \frac{(i+j+1)h_0}{C^p} \frac{1}{\mu_0} +\frac{(i+1)h_0}{C^p} \min\left\{\frac{1}{\mu_1}, \frac{1}{\mu_2}\right\} \label{eq:one_more_queue_full}\\
                & \quad \geq \min\left\{\frac{h_1}{\mu_1},\frac{h_2}{\mu_2}\right\}. \label{eq:one_more_queue}
            \end{align}
        \item \label{state:lemma_diff_bound_move_one}
            Consider $(i,j,k,\ell) \in \X_D$, where $i \geq 0$ and $j+k+\ell = C^p$,
            \begin{align}
                v(i+1,j-1,k+1,\ell) - v(i,j,k,\ell) & \geq  \min\left\{\frac{h_1}{\mu_1},\frac{h_2}{\mu_2}\right\}. \label{eq:bound_triage_noncollab} \\
                v(i+1,j-1,k,\ell+1) - v(i,j,k,\ell) & \geq  \min\left\{\frac{h_1}{\mu_1},\frac{h_2}{\mu_2}\right\}. \label{eq:bound_triage_collab}
            \end{align}
            Moreover, 
            \begin{align}
                & v(i+1,j-1,k+1,\ell) - v(i,j,k,\ell) \nonumber \\
                &\quad \geq \min\left\{\frac{h_1}{\mu_1},\frac{h_2}{\mu_2}\right\} + \frac{\mu_0}{d(j,k+1,\ell)} \left(\frac{(i+j+1)h_0}{C^p \mu_0} +\frac{(i+1)h_0}{C^p} \min\left\{\frac{1}{\mu_1}, \frac{1}{\mu_2}\right\} \right). \label{eq:bound_triage_noncollab_full} \\
                & v(i+1,j-1,k,\ell+1) - v(i,j,k,\ell) \nonumber \\
                &\quad \geq \min\left\{\frac{h_1}{\mu_1},\frac{h_2}{\mu_2}\right\} + \frac{\mu_0}{d(j,k,\ell+1)} \left(\frac{(i+j+1)h_0}{C^p \mu_0} +\frac{(i+1)h_0}{C^p} \min\left\{\frac{1}{\mu_1}, \frac{1}{\mu_2}\right\} \right). \label{eq:bound_triage_collab_full}
            \end{align}
        \item \label{state:lemma_diff_larger_mu1_ub}
        Suppose 
        $\mu_1 \geq \mu_2$ and consider $(i,j,k,\ell) \in \X_D$.
        \begin{equation}
            v(i,j-1,k+1,\ell) - v(i,j-1,k,\ell+1) \leq \frac{h_1}{\mu_1} - \frac{h_2}{\mu_2}. \label{eq:larger_mu1_ub}
        \end{equation}
    \end{enumerate}
\end{lemma}

\begin{proof} [Proof of Statement \ref{state:lemma_diff_one_more_base} in Lemma \ref{lemma:diff}.] The proof proceeds analogously, albeit more simply than other statements, using induction on $M$, where $M= 2(i+j)+k+\ell$ for state $(i=0,j,k,\ell)$, and is omitted for brevity.
\end{proof}

\begin{proof} [Proof of Statement \ref{state:lemma_diff_one_more} in Lemma \ref{lemma:diff}.]
    The proof proceeds by induction on $M$, where $M= 2(i+j)+k+\ell$ for state $(i,j,k,\ell)$.
    Consider the base case at $M = C^p$, i.e., $i=j=0$ and $k+\ell=C^p$.
    \begin{align}
        \lefteqn{v(1,0,k,\ell) - v(0,0,k,\ell)}& \nonumber \\
        & \quad = \frac{h_0}{d(0,k,\ell)} + \frac{k\mu_1}{d(0,k,\ell)}\Big[v(0,1,k-1,\ell) - v(0,0,k-1,\ell)\Big] \nonumber \\
        & \quad \qquad + \frac{\min\{\ell,C^G\}\mu_2}{d(0,k,\ell)}\Big[v(0,1,k,\ell-1) - v(0,0,k,\ell-1)\Big]. \label{eq:one_more_base0}
    \end{align}
    Applying \eqref{eq:one_more_triage} for the differences with coefficients $\frac{k\mu_1}{d(0,k,\ell)}$ and $\frac{\min\{\ell,C^G\}\mu_2}{d(0,k,\ell)}$ in \eqref{eq:one_more_base0} yields
    \begin{align*}
        &v(1,0,k,\ell) - v(0,0,k,\ell) -\left(\min\left\{\frac{h_1}{\mu_1},\frac{h_2}{\mu_2}\right\} + \frac{h_0}{C^p} \frac{1}{\mu_0} +\frac{h_0}{C^p} \min\left\{\frac{1}{\mu_1}, \frac{1}{\mu_2}\right\}\right)\\
        & \quad \geq  \frac{h_0}{d(0,k,\ell)} +\Big( \frac{k\mu_1 + \min\{\ell,C^G\}\mu_2}{d(0,k,\ell)}\Big) \Big(\frac{h_0}{\mu_0} + \min \left\{\frac{h_1}{\mu_1},\frac{h_2}{\mu_2}\right\}\Big)\\ 
        &\quad \qquad -\left(\min\left\{\frac{h_1}{\mu_1},\frac{h_2}{\mu_2}\right\} + \frac{h_0}{C^p} \frac{1}{\mu_0} +\frac{h_0}{C^p} \min\left\{\frac{1}{\mu_1}, \frac{1}{\mu_2}\right\}\right)\\
        & \quad = \left(\frac{h_0}{\mu_0} - \frac{h_0}{C^p} \frac{1}{\mu_0} \right) + \left(\frac{h_0}{d(0,k,\ell)}- \frac{h_0}{C^p} \min\left\{\frac{1}{\mu_1}, \frac{1}{\mu_2}\right\}\right).
    \end{align*}The result follows by noticing $\frac{h_0}{\mu_0} - \frac{h_0}{C^p} \frac{1}{\mu_0} \geq 0$ and
    \begin{align*}
        \lefteqn{\frac{h_0}{d(0,k,\ell)}- \frac{h_0}{C^p} \min\left\{\frac{1}{\mu_1}, \frac{1}{\mu_2}\right\}}\\
        &\quad = \frac{h_0}{d(0,k,\ell)}- \frac{k\mu_1}{d(0,k,\ell)}\frac{h_0}{C^p} \min\left\{\frac{1}{\mu_1}, \frac{1}{\mu_2}\right\} - \frac{\min\{\ell,C^G\}\mu_2}{d(0,k,\ell)}\frac{h_0}{C^p} \min\left\{\frac{1}{\mu_1}, \frac{1}{\mu_2}\right\}\\
        &\quad \geq \frac{h_0}{d(j,k,\ell)} - \frac{k}{d(j,k,\ell)}\frac{h_0}{C^p} - \frac{\min\{\ell,C^G\}}{d(j,k,\ell)}\frac{h_0}{C^p}\\
        &\quad \geq 0,
    \end{align*}where the first equality uses the fact that $d(0,k,\ell) = k\mu_1+\min\{\ell,C^G\}\mu_2$, the first inequality bounds the minimum in the second and third terms by $\frac{1}{\mu_1}$ and $\frac{1}{\mu_2}$, respectively, and the last step holds since $k+\min\{\ell,C^G\} \leq C^p$.

    Now consider a general $M$ in state $(i,j,k,\ell)$ by assuming the result holds at $M-1$. The proof proceeds by considering whether $i=0$ or $i \geq 1$, which govern the state transitions (and hence Bellman equations).
    \begin{enumerate}[label= \textbf{Case} \arabic*:, leftmargin=3\parindent]
        \item When $i = 0$ with $j \geq 1$, we have
        \begin{align}
            \lefteqn{v(1,j,k,\ell) - v(0,j,k,\ell)}& \nonumber \\
            & \quad = \frac{h_0}{d(j,k,\ell)} + \frac{j\mu_0}{d(j,k,\ell)}\Big[\min\{v(1,j-1,k+1,\ell),v(1,j-1,k,\ell+1)\} \nonumber \\
            & \qquad \qquad - \min\{v(0,j-1,k+1,\ell),v(0,j-1,k,\ell+1)\}\Big] \nonumber \\
            &\qquad \quad + \frac{k\mu_1}{d(j,k,\ell)}\Big[v(0,j+1,k-1,\ell) - v(0,j,k-1,\ell)\Big] \nonumber \\
            & \quad \qquad + \frac{\min\{\ell,C^G\}\mu_2}{d(j,k,\ell)}\Big[v(0,j+1,k,\ell-1) - v(0,j,k,\ell-1)\Big]. \label{eq:one_more_base}
        \end{align}
        Consider the difference in the second term in \eqref{eq:one_more_base} with coefficient $\frac{j\mu_0}{d(j,k,\ell)}$. Suppose $v(1,j-1,k+1,\ell) \leq v(1,j-1,k,\ell+1)$, choosing an upper bound $v(0,j-1,k+1,\ell)$ in the second minimum yields
            \begin{align}
                \lefteqn{\min\{v(1,j-1,k+1,\ell),v(1,j-1,k,\ell+1)\}}& \nonumber \\
                &\quad \qquad - \min\{v(0,j-1,k+1,\ell),v(0,j-1,k,\ell+1)\} \nonumber \\
                &\quad \geq v(1,j-1,k+1,\ell) - v(0,j-1,k+1,\ell) \nonumber \\
                &\quad \geq \min\left\{\frac{h_1}{\mu_1},\frac{h_2}{\mu_2}\right\} + \frac{j h_0}{C^p} \frac{1}{\mu_0} +\frac{h_0}{C^p} \min\left\{\frac{1}{\mu_1}, \frac{1}{\mu_2}\right\}, \label{eq:eq:one_more_base_mu0_terms}
            \end{align}where the last step applies the inductive hypothesis in state $(j-1,k+1,\ell)$ with $M-1$ services to complete.  If $v(1,j-1,k+1,\ell) > v(1,j-1,k,\ell+1)$, choosing an upper bound $v(0,j-1,k,\ell+1)$ in the second minimum and subsequently invoking the inductive hypothesis yields the same lower bound.
        By applying \eqref{eq:eq:one_more_base_mu0_terms} for the difference in the second term and \eqref{eq:one_more_triage} for the differences in both third and fourth terms in \eqref{eq:one_more_base}, we obtain
        \begin{align}
            \lefteqn{v(1,j,k,\ell) - v(0,j,k,\ell) - \left(\min\left\{\frac{h_1}{\mu_1},\frac{h_2}{\mu_2}\right\} + \frac{(j+1)h_0}{C^p} \frac{1}{\mu_0} +\frac{h_0}{C^p} \min\left\{\frac{1}{\mu_1}, \frac{1}{\mu_2}\right\}\right)}& \nonumber \\
            &\quad \geq \frac{h_0}{d(j,k,\ell)} + \frac{j\mu_0}{d(j,k,\ell)}\left(\min\left\{\frac{h_1}{\mu_1},\frac{h_2}{\mu_2}\right\} + \frac{j h_0}{C^p} \frac{1}{\mu_0} +\frac{h_0}{C^p} \min\left\{\frac{1}{\mu_1}, \frac{1}{\mu_2}\right\}\right) \nonumber \\
            &\qquad \quad + \frac{k\mu_1}{d(j,k,\ell)}\left(\min\left\{\frac{h_1}{\mu_1},\frac{h_2}{\mu_2}\right\} + \frac{h_0}{\mu_0}\right) + \frac{\min\{\ell,C^G\}\mu_2}{d(j,k,\ell)}\left(\min\left\{\frac{h_1}{\mu_1},\frac{h_2}{\mu_2}\right\} + \frac{h_0}{\mu_0} \right) \nonumber \\
            &\qquad \quad- \left(\min\left\{\frac{h_1}{\mu_1},\frac{h_2}{\mu_2}\right\} + \frac{(j+1)h_0}{C^p} \frac{1}{\mu_0} +\frac{h_0}{C^p} \min\left\{\frac{1}{\mu_1}, \frac{1}{\mu_2}\right\}\right) \nonumber \\
            &\quad = \frac{h_0}{d(j,k,\ell)} - \frac{j\mu_0}{d(j,k,\ell)}\frac{h_0}{C^p\mu_0} - \frac{k\mu_1+\min\{\ell,C^G\}\mu_2}{d(j,k,\ell)} \frac{h_0}{\mu_0}\left(\frac{j+1}{C^p}-1\right) \nonumber \\
            &\quad \qquad - \frac{k\mu_1}{d(j,k,\ell)}\frac{h_0}{C^p} \min\left\{\frac{1}{\mu_1}, \frac{1}{\mu_2}\right\} - \frac{\min\{\ell,C^G\}\mu_2}{d(j,k,\ell)}\frac{h_0}{C^p} \min\left\{\frac{1}{\mu_1}, \frac{1}{\mu_2}\right\}. \label{eq:one_more_base_inte}
        \end{align}
        Notice the third term in \eqref{eq:one_more_base_inte} is non-negative:
        \begin{enumerate}[label= \textbf{Subcase} (\alph*):, leftmargin=2.2\parindent]
            \item If $j+1 \leq C^p$, it follows immediately that 
            \begin{align}
                - \frac{k\mu_1+\min\{\ell,C^G\}\mu_2}{d(j,k,\ell)} \frac{h_0}{\mu_0}\left(\frac{j+1}{C^p}-1\right) \geq 0. \label{eq:one_more_base_inte_non_neg}
            \end{align}
            \item If $j+1 > C^p$, then necessarily $j = C^p$ (since $j \leq C^p$), which in turn implies $k = \ell = 0$ due to the constraint $j+k+\ell = C^p$. Consequently, the third term in \eqref{eq:one_more_base_inte} vanishes.
        \end{enumerate}
        Applying \eqref{eq:one_more_base_inte_non_neg} in \eqref{eq:one_more_base_inte} yields
        \begin{align}
            \lefteqn{v(1,j,k,\ell) - v(0,j,k,\ell) - \left(\min\left\{\frac{h_1}{\mu_1},\frac{h_2}{\mu_2}\right\} + \frac{(j+1)h_0}{C^p} \frac{1}{\mu_0} +\frac{h_0}{C^p} \min\left\{\frac{1}{\mu_1}, \frac{1}{\mu_2}\right\}\right)}& \nonumber \\
            &\quad \geq \frac{h_0}{d(j,k,\ell)} - \frac{j}{d(j,k,\ell)}\frac{h_0}{C^p} - \frac{k\mu_1}{d(j,k,\ell)}\frac{h_0}{C^p} \min\left\{\frac{1}{\mu_1}, \frac{1}{\mu_2}\right\} - \frac{\min\{\ell,C^G\}\mu_2}{d(j,k,\ell)}\frac{h_0}{C^p} \min\left\{\frac{1}{\mu_1}, \frac{1}{\mu_2}\right\} \nonumber \\
            &\quad \geq \frac{h_0}{d(j,k,\ell)} - \frac{j}{d(j,k,\ell)}\frac{h_0}{C^p} - \frac{k}{d(j,k,\ell)}\frac{h_0}{C^p} - \frac{\min\{\ell,C^G\}}{d(j,k,\ell)}\frac{h_0}{C^p} \nonumber \\
            &\quad \geq 0, \label{eq:one_more_base_non_neg}
        \end{align}where the second-to-last inequality follows by upper bounding the minimum in the third and fourth terms by $\frac{1}{\mu_1}$ and $\frac{1}{\mu_2}$, respectively, and the last step holds since $j+k+\min\{\ell,C^G\} \leq C^p$.
        \item When $i \geq 1$, the difference in \eqref{eq:one_more_queue} now becomes
        \begin{align}
            \lefteqn{v(i+1,j,k,\ell) - v(i,j,k,\ell)}& \nonumber \\
            = & \ \frac{h_0}{d(j,k,\ell)} +\frac{j\mu_0}{d(j,k,\ell)}\Big[\min\{v(i+1,j-1,k+1,\ell),v(i+1,j-1,k,\ell+1)\} \nonumber \\
            & \qquad-\min\{v(i,j-1,k+1,\ell),v(i,j-1,k,\ell+1)\}\Big] \nonumber \\
            & \quad +\frac{k\mu_1}{d(j,k,\ell)}\Big[v(i,j+1,k-1,\ell) - v(i-1,j+1,k-1,\ell)\Big] \nonumber\\
            & \quad +\frac{\min\{\ell,C^G\}\mu_2}{d(j,k,\ell)}\Big[v(i,j+1,k,\ell-1) - v(i-1,j+1,k,\ell-1)\Big] \label{eq:one_more}
        \end{align}
        Using a similar approach as proving \eqref{eq:eq:one_more_base_mu0_terms}, it follows that the difference in the second term in \eqref{eq:one_more} with coefficient $\frac{j\mu_0}{d(j,k,\ell)}$ is lower bounded by the following
        \begin{align}
        \lefteqn{}
            \min\left\{\frac{h_1}{\mu_1},\frac{h_2}{\mu_2}\right\} + \frac{(i+j)h_0}{C^p} \frac{1}{\mu_0} +\frac{(i+1)h_0}{C^p} \min\left\{\frac{1}{\mu_1}, \frac{1}{\mu_2}\right\}. \label{eq:one_more_mu0_terms}
        \end{align}
        Applying \eqref{eq:one_more_mu0_terms} and the inductive hypothesis for the difference in the third and the last terms results in
        \begin{align*}
            &v(i+1,j,k,\ell) - v(i,j,k,\ell) - \left(\min\left\{\frac{h_1}{\mu_1},\frac{h_2}{\mu_2}\right\} + \frac{(i+j+1)h_0}{C^p} \frac{1}{\mu_0} +\frac{(i+1)h_0}{C^p} \min\left\{\frac{1}{\mu_1}, \frac{1}{\mu_2}\right\}\right)\\
            & \quad \geq \frac{h_0}{d(j,k,\ell)} + \frac{j\mu_0}{d(j,k,\ell)}\left(\min\left\{\frac{h_1}{\mu_1},\frac{h_2}{\mu_2}\right\} + \frac{(i+j)h_0}{C^p} \frac{1}{\mu_0} +\frac{(i+1)h_0}{C^p} \min\left\{\frac{1}{\mu_1}, \frac{1}{\mu_2}\right\}\right) \\
            & \quad \qquad + \frac{k\mu_1}{d(j,k,\ell)} \left(\min\left\{\frac{h_1}{\mu_1},\frac{h_2}{\mu_2}\right\} + \frac{(i+j+1)h_0}{C^p} \frac{1}{\mu_0} +\frac{i h_0}{C^p} \min\left\{\frac{1}{\mu_1}, \frac{1}{\mu_2}\right\}\right)\\
            & \quad \qquad + \frac{\min\{\ell,C^G\}\mu_2}{d(j,k,\ell)}\left(\min\left\{\frac{h_1}{\mu_1},\frac{h_2}{\mu_2}\right\} + \frac{(i+j+1)h_0}{C^p} \frac{1}{\mu_0} +\frac{i h_0}{C^p} \min\left\{\frac{1}{\mu_1}, \frac{1}{\mu_2}\right\}\right)\\
            &\quad \qquad - \left(\min\left\{\frac{h_1}{\mu_1},\frac{h_2}{\mu_2}\right\} + \frac{(i+j+1)h_0}{C^p} \frac{1}{\mu_0} +\frac{(i+1)h_0}{C^p} \min\left\{\frac{1}{\mu_1}, \frac{1}{\mu_2}\right\}\right)\\
            &\quad = \frac{h_0}{d(j,k,\ell)} - \frac{j\mu_0}{d(j,k,\ell)}\frac{h_0}{C^p\mu_0}-\frac{k\mu_1+\min\{\ell,C^G\}\mu_2}{d(j,k,\ell)}\frac{h_0}{C^p} \min\left\{\frac{1}{\mu_1}, \frac{1}{\mu_2}\right\} \\
            &\quad \geq 0,
        \end{align*}
    where the final step follows from a similar reasoning as \eqref{eq:one_more_base_non_neg}.
    \end{enumerate}
    Consequently, the result in \eqref{eq:one_more_queue_full} holds, which also leads to \eqref{eq:one_more_queue}.
\end{proof}

\begin{proof} [Proof of Statement \ref{state:lemma_diff_bound_move_one} in Lemma \ref{lemma:diff}.]
    We establish \eqref{eq:bound_triage_collab} and \eqref{eq:bound_triage_collab_full} here. The corresponding results in \eqref{eq:bound_triage_noncollab} and \eqref{eq:bound_triage_noncollab_full} follow a symmetric, albeit slightly simpler, argument.
    The proof proceeds by induction on the total number of remaining services, defined as $M = 2(i+j)+k+\ell$ for the state $(i,j,k,\ell)$.

    We begin with the base case $i=0$, $j=1$, and $k+\ell = C^p-1$, so that  $M = C^p+1$. By thinning the MDPs in \eqref{eq:bound_triage_collab} with a higher rate $d(1,k,\ell+1)$, we have
    \begin{align}
        \lefteqn{v(1,0,k,\ell+1) - v(0,1,k,\ell)}& \nonumber \\
        & \quad = \frac{h_2}{d(1,k,\ell+1)} + \frac{k\mu_1}{d(1,k,\ell+1)}\Big[v(0,1,k-1,\ell+1) - v(0,1,k-1,\ell)\Big] \nonumber \\
        & \quad \qquad + \frac{\min\{\ell,C^G\}\mu_2}{d(1,k,\ell+1)}\Big[v(0,1,k,\ell) - v(0,1,k,\ell-1)\Big]  +\frac{\mu_2\mathbb{I}\{\ell<C^G\}}{d(1,k,\ell+1)}\Big[v(0,1,k,\ell) - v(0,1,k,\ell)\Big] \nonumber \\
        & \quad \qquad +\frac{\mu_0}{d(1,k,\ell+1)}\Big[v(1,0,k,\ell+1) - \min\{v(0,0,k+1,\ell),v(0,0,k,\ell+1)\}\Big] \label{eq:bound_triage_collab_base0}
    \end{align}
    Choosing an upper bound $v(0,0,k,\ell+1)$ in place of the minimum in the last term of \eqref{eq:bound_triage_collab_base0}, which has coefficient $\frac{\mu_0}{d(1,k,\ell+1)}$, yields
    \begin{align}
        \lefteqn{v(1,0,k,\ell+1) - \min\{v(0,0,k+1,\ell),v(0,0,k,\ell+1)\}}&\nonumber\\
        &\quad \geq v(1,0,k,\ell+1) - v(0,0,k,\ell+1) \nonumber \\
        &\quad \geq \min \left\{\frac{h_1}{\mu_1},\frac{h_2}{\mu_2}\right\}, \label{eq:bound_triage_collab_base0_mu0}
    \end{align}where the last step holds by \eqref{eq:one_more_queue}.
    Applying \eqref{eq:one_more_collab} twice for the differences in the second and the third terms, with coefficients $\frac{k\mu_1}{d(1,k,\ell+1)}$ and $\frac{\mu_2\mathbb{I}\{\ell<C^G\}}{d(1,k,\ell+1)}$, respectively, and \eqref{eq:bound_triage_collab_base0_mu0} in \eqref{eq:bound_triage_collab_base0}, we obtain,
    \begin{align}
        \lefteqn{v(1,0,k,\ell+1) - v(0,1,k,\ell)}& \nonumber \\
        & \quad\geq \frac{h_2}{d(1,k,\ell+1)} +\frac{k\mu_1 + \min\{\ell,C^G\}\mu_2}{d(1,k,\ell+1)} \frac{h_2}{\mu_2} + \frac{\mu_0}{d(1,k,\ell+1)} \min \left\{\frac{h_1}{\mu_1},\frac{h_2}{\mu_2}\right\} \nonumber \\
        &\quad \geq \frac{\mu_0+k\mu_1 + (\min\{\ell,C^G\}+1)\mu_2}{d(1,k,\ell+1)} \min \left\{\frac{h_1}{\mu_1},\frac{h_2}{\mu_2}\right\} \nonumber \\
        & \quad \geq  \min \left\{\frac{h_1}{\mu_1},\frac{h_2}{\mu_2}\right\}, \label{eq:bound_triage_collab_base0_result}
    \end{align}where the second inequality uses the fact that $\frac{h_2}{\mu_2} \geq \min \left\{\frac{h_1}{\mu_1},\frac{h_2}{\mu_2}\right\}$, and the final step holds since $\mu_0+k\mu_1 + (\min\{\ell,C^G\}+1)\mu_2 \geq \mu_0+k\mu_1 + \min\{\ell+1,C^G\}\mu_2 = d(1,k,\ell+1)$.

    Now suppose \eqref{eq:bound_triage_collab} holds at $M-1$, and consider it at $M$, where $M \geq C^p+1$, based on whether $i=0$.
    \begin{enumerate}[label= \textbf{Case} \arabic*:, leftmargin=3\parindent]
        \item When $i=0$, thinning the MDPs in \eqref{eq:bound_triage_collab} with a higher rate $d(j,k,\ell+1)$ yields
        \begin{align}
            \lefteqn{v(1,j-1,k,\ell+1) - v(0,j,k,\ell)}& \nonumber \\
            & \quad = \frac{h_2}{d(j,k,\ell+1)} + \frac{k\mu_1}{d(j,k,\ell+1)}\Big[v(0,j,k-1,\ell+1) - v(0,j,k-1,\ell)\Big] \nonumber \\
            & \quad \qquad + \frac{\min\{\ell,C^G\}\mu_2}{d(j,k,\ell+1)}\Big[v(0,j,k,\ell) - v(0,j,k,\ell-1)\Big] \nonumber \\
            &\quad \qquad + \frac{(j-1)\mu_0}{d(j,k,\ell+1)}\Big[\min\{v(1,j-2,k+1,\ell+1),v(1,j-2,k,\ell+2)\} \nonumber \\
            & \qquad \qquad- \min\{v(0,j-1,k+1,\ell),v(0,j-1,k,\ell+1)\}\Big] \nonumber \\
            & \quad \qquad +\frac{\mu_2\mathbb{I}\{\ell<C^G\}}{d(j,k,\ell+1)}\Big[v(0,j,k,\ell) - v(0,j,k,\ell)\Big] \label{eq:bound_triage_collab_base} \\
            & \quad \qquad+ \frac{\mu_0}{d(j,k,\ell+1)}\Big[v(1,j-1,k,\ell+1) - \min\{v(0,j-1,k+1,\ell),v(0,j-1,k,\ell+1)\}\Big]. \nonumber 
        \end{align}
        We proceed to show that the difference in the fourth term of \eqref{eq:bound_triage_collab_base}, which is multiplied by the coefficient $\frac{(j-1)\mu_0}{d(j,k,\ell+1)}$, is lower bounded by $\min \left\{\frac{h_1}{\mu_1},\frac{h_2}{\mu_2}\right\}$. 
        Suppose first $v(1,j-2,k+1,\ell+1) \leq v(1,j-2,k,\ell+2)$, replacing the second minimum with an upper bound $v(0,j-1,k+1,\ell)$ yields
        \begin{align}
            &\min\{v(1,j-2,k+1,\ell+1),v(1,j-2,k,\ell+2)\} - \min\{v(0,j-1,k+1,\ell),v(0,j-1,k,\ell+1)\} \nonumber \\
            & \quad \geq v(1,j-2,k+1,\ell+1) - v(0,j-1,k+1,\ell) \nonumber \\
            & \quad \geq \min \left\{\frac{h_1}{\mu_1},\frac{h_2}{\mu_2}\right\}, \label{eq:bound_triage_collab_base_mu0_terms}
        \end{align}where the last step applies the inductive hypothesis. 
        If $v(1,j-2,k+1,\ell+1) > v(1,j-2,k,\ell+2)$, replacing the second minimum with an upper bound $v(0,j-1,k,\ell+1)$ yields the same lower bound as in \eqref{eq:bound_triage_collab_base_mu0_terms} using the inductive hypothesis. 
        The terms associated with the coefficients $\frac{k\mu_1}{d(j,k,\ell+1)}$, $\frac{\min\{\ell,C^G\}\mu_2}{d(j,k,\ell+1)}$, and $\frac{\mu_0}{d(j,k,\ell+1)}$ satisfy the same lower bound by reasoning analogous to that used in the base case.
        
        Applying the lower bounds $\min \left\{\frac{h_1}{\mu_1},\frac{h_2}{\mu_2}\right\}$ to each of the differences in \eqref{eq:bound_triage_collab_base}, with the exception of the term involving the coefficient $\frac{\mu_2\mathbb{I}\{\ell<C^G\}}{d(j,k,\ell+1)}$, we obtain
        \begin{align*}
            \lefteqn{v(1,j-1,k,\ell+1) - v(0,j,k,\ell)}&\\
            & \quad \geq \frac{h_2}{d(j,k,\ell+1)} + \frac{j\mu_0 + k\mu_1 + \min\{\ell,C^G\}\mu_2}{d(j,k,\ell+1)}  \min \left\{\frac{h_1}{\mu_1},\frac{h_2}{\mu_2}\right\}\\
            & \quad \geq \frac{j\mu_0 + k\mu_1 + (\min\{\ell,C^G\}+1)\mu_2}{d(j,k,\ell+1)}  \min \left\{\frac{h_1}{\mu_1},\frac{h_2}{\mu_2}\right\}\\
            & \quad \geq \min \left\{\frac{h_1}{\mu_1},\frac{h_2}{\mu_2}\right\},
        \end{align*}where the final two steps follow by analogous reasoning to that used in deriving \eqref{eq:bound_triage_collab_base0_result}. 
        
        We have thus far established \eqref{eq:bound_triage_collab} in all cases where $i=0$. Building on this result, we note that the difference in the fourth term of \eqref{eq:bound_triage_collab_base}, with coefficient $\frac{(j-1)\mu_0}{d(j,k,\ell+1)}$, is lower bounded by $\min \left\{\frac{h_1}{\mu_1},\frac{h_2}{\mu_2}\right\}$, following analogous reasoning as in earlier steps, except that we now invoke the established result in place of the inductive hypothesis (see \eqref{eq:bound_triage_collab_base_mu0_terms}). Applying the same lower bound $\min\left\{\frac{h_1}{\mu_1},\frac{h_2}{\mu_2}\right\}$ to the differences in \eqref{eq:bound_triage_collab_base} associated with the coefficients $\frac{k\mu_1}{d(j,k,\ell+1)}$ and $\frac{\min\{\ell,C^G\}\mu_2}{d(j,k,\ell+1)}$, using \eqref{eq:one_more_collab}, and further employing the stronger result in \eqref{eq:one_more_queue_full} (in place of \eqref{eq:one_more_queue}) for the difference with coefficient $\frac{\mu_0}{d(j,k,\ell+1)}$, we conclude that \eqref{eq:bound_triage_collab_full} holds for all states with $i=0$.
        \item When $i \geq 1$, thinning the MDPs with a higher rate $d(j,k,\ell+1)$ yields
        \begin{align}
            \lefteqn{v(i+1,j-1,k,\ell+1) - v(i,j,k,\ell)}& \nonumber \\
            & \quad = \frac{h_2}{d(j,k,\ell+1)} + \frac{k\mu_1}{d(j,k,\ell+1)}\Big[v(i,j,k-1,\ell+1) - v(i-1,j+1,k-1,\ell)\Big] \nonumber \\
            & \quad \qquad + \frac{\min\{\ell,C^G\}\mu_2}{d(j,k,\ell+1)}\Big[v(i,j,k,\ell) - v(i-1,j+1,k,\ell-1)\Big] \nonumber \\
            & \quad \qquad+ \frac{(j-1)\mu_0}{d(j,k,\ell+1)}\Big[\min\{v(i+1,j-2,k+1,\ell+1),v(i+1,j-2,k,\ell+2)\} \nonumber \\
            & \qquad \qquad- \min\{v(i,j-1,k+1,\ell),v(i,j-1,k,\ell+1)\}\Big] \nonumber \\
            & \quad \qquad +\frac{\mu_2\mathbb{I}\{\ell<C^G\}}{d(j,k+1,\ell)}\Big[v(i,j,k,\ell) - v(i,j,k,\ell)\Big] \label{eq:bound_triage_collab_general} \\ 
            & \quad \qquad+ \frac{\mu_0}{d(j,k,\ell+1)}\Big[v(i+1,j-1,k,\ell+1)  - \min\{v(i,j-1,k+1,\ell),v(i,j-1,k,\ell+1)\}\Big]. \nonumber
        \end{align}
        The result in \eqref{eq:bound_triage_collab} follows by analogous reasoning as in \eqref{eq:bound_triage_collab_base}, with the only distinction being the application of the inductive hypothesis (invoked twice) for the terms with coefficients $\frac{k\mu_1}{d(j,k,\ell+1)}$ and $\frac{\min\{\ell,C^G\}\mu_2}{d(j,k,\ell+1)}$, in place of the use of \eqref{eq:one_more_collab}. 

        Building on this, the stronger bound in \eqref{eq:bound_triage_collab_full} follows by a similar argument as in the case $i=0$, with the key difference being the use of the established result in \eqref{eq:bound_triage_collab} (twice) for the terms with coefficients $\frac{k\mu_1}{d(j,k,\ell+1)}$ and $\frac{\min\{\ell,C^G\}\mu_2}{d(j,k,\ell+1)}$, rather than the reliance on \eqref{eq:one_more_collab}. \qedhere
    \end{enumerate}
\end{proof}

\begin{proof} [Proof of Statement \ref{state:lemma_diff_larger_mu1_ub} in Lemma \ref{lemma:diff}.] 
    We prove \eqref{eq:larger_mu1_ub} by induction on $M=2(i+j)+k+\ell$, which indicates the number of services remaining in state $(i,j,k,\ell)$. For the base case $i=k=\ell=0$ and $j=1$, i.e., $M=2$, we have $v(0,0,1,0) - v(0,0,0,1) = \frac{h_1}{\mu_1} - \frac{h_2}{\mu_2}$.
    We now establish the result for $M >2$, assuming it holds for $M-1$. There are two cases that distinguish the discussion. Observe that $\mu_1 \geq \mu_2$ implies $d(j-1,k+1,\ell) \geq d(j-1,k,\ell+1)$.
    \begin{enumerate}[label= \textbf{Case} \arabic*:, leftmargin=3\parindent]
        \item When $i=0$, we apply a thinning argument to the MDPs in \eqref{eq:larger_mu1_ub} with a higher rate $d(j-1,k+1,\ell)$:
        \begin{align}
            \lefteqn{v(0,j-1,k+1,\ell) - v(0,j-1,k,\ell+1)}& \nonumber \\
            & \quad = \frac{h_1-h_2}{d(j-1,k+1,\ell)} + \frac{k\mu_1}{d(j-1,k+1,\ell)}\Big[v(0,j-1,k,\ell) - v(0,j-1,k-1,\ell+1)\Big] \nonumber \\
            & \quad \qquad + \frac{\min\{\ell,C^G\} \mu_2}{d(j-1,k+1,\ell)}\Big[v(0,j-1,k+1,\ell-1) - v(0,j-1,k,\ell)\Big] \nonumber \\
            &\quad \qquad + \frac{(j-1)\mu_0}{d(j-1,k+1,\ell)} \Big[ \min\{v(0,j-2,k+2,\ell),v(0,j-2,k+1,\ell+1)\} \nonumber \\
            & \qquad \qquad - \min\{v(0,j-2,k+1,\ell+1),v(0,j-2,k,\ell+2)\} \Big] \nonumber \\
            & \quad \qquad+ \frac{\mu_2\mathbb{I}\{\ell <C^G\}}{d(j-1,k+1,\ell)}\Big[v(0,j-1,k,\ell) - v(0,j-1,k,\ell)\Big] \nonumber \\
            & \quad \qquad+ \frac{\mu_1-\mu_2\mathbb{I}\{\ell <C^G\}}{d(j-1,k+1,\ell)}\Big[v(0,j-1,k,\ell) - v(0,j-1,k,\ell+1)\Big]. \label{eq:larger_mu1_ub_base}
        \end{align}
        Suppose $v(0,j-2,k+1,\ell+1) \leq v(0,j-2,k,\ell+2)$. Replacing the first minimum in the second term with the upper bound $v(0,j-2,k+2,\ell)$ yields
        \begin{align}
            &\min\{v(0,j-2,k+2,\ell),v(0,j-2,k+1,\ell+1)\} - \min\{v(0,j-2,k+1,\ell+1),v(0,j-2,k,\ell+2)\} \nonumber \\
            &\quad \leq v(0,j-2,k+2,\ell)- v(0,j-2,k+1,\ell+1) \nonumber \\
            &\quad \leq \frac{h_1}{\mu_1} - \frac{h_2}{\mu_2}, \label{eq:larger_mu1_ub_base_mu0_terms}
        \end{align}
        where the last inequality follows from the inductive hypothesis. If $v(0,j-2,k+1,\ell+1) > v(0,j-2,k,\ell+2)$, the same upper bound in \eqref{eq:larger_mu1_ub_base_mu0_terms} holds by choosing an upper bound $v(0,j-2,k+1,\ell+1)$ in the first minimum and using the inductive hypothesis.
        
        Applying \eqref{eq:larger_mu1_ub_base_mu0_terms} and \eqref{eq:one_more_collab} to the differences associated with the coefficients $\frac{(j-1)\mu_0}{d(j-1,k+1,\ell)}$ and $ \frac{\mu_1-\mu_2\mathbb{I}\{\ell <C^G\}}{d(j-1,k+1,\ell)}$, respectively, and invoking the inductive hypothesis (twice) for the differences with coefficients $\frac{k\mu_1}{d(j-1,k+1,\ell)}$ and $\frac{\min\{\ell,C^G\} \mu_2}{d(j-1,k+1,\ell)}$ in \eqref{eq:larger_mu1_ub_base}, we obtain
        \begin{align}
            &\qquad v(0,j-1,k+1,\ell) - v(0,j-1,k,\ell+1) \nonumber \\
            & \leq  \frac{h_1 - h_2}{d(j-1,k+1,\ell)} + \frac{(j-1)\mu_0 + k\mu_1 + \min\{\ell,C^G\} \mu_2}{d(j-1,k+1,\ell)}\Big( \frac{h_1}{\mu_1} - \frac{h_2}{\mu_2}\Big) + \frac{\mu_1-\mu_2\mathbb{I}\{\ell <C^G\}}{d(j-1,k+1,\ell)}\Big(-\frac{h_2}{\mu_2}\Big) \nonumber \\
            & \leq  \frac{h_1 - h_2}{d(j-1,k+1,\ell)} + \frac{(j-1)\mu_0 + k\mu_1 + \min\{\ell,C^G\} \mu_2}{d(j-1,k+1,\ell)}\Big( \frac{h_1}{\mu_1} - \frac{h_2}{\mu_2}\Big)  + \frac{\mu_1-\mu_2}{d(j-1,k+1,\ell)}\Big(-\frac{h_2}{\mu_2}\Big) \nonumber \\
            & \quad  = \frac{h_1}{\mu_1} - \frac{h_2}{\mu_2}, \label{eq:larger_mu1_ub_base_result}
        \end{align}where the second inequality follows from the bound $\mu_1 - \mu_2\mathbb{I}\{\ell < C^G\} \geq \mu_1 - \mu_2$, and the final step uses the expression $(j-1)\mu_0 + (k+1)\mu_1 + \min\{\ell,C^G\} \mu_2 = d(j-1,k+1,\ell)$.
        \item When $i \geq 1$, again thinning the MDPs with a higher rate $d(j-1,k+1,\ell)$ yields
        \begin{align}
            \lefteqn{v(i,j-1,k+1,\ell) - v(i,j-1,k,\ell+1) }& \nonumber\\
            & \quad = \frac{h_1-h_2}{d(j-1,k,\ell+1)} + \frac{k\mu_1}{d(j-1,k+1,\ell)}\Big[v(i-1,j,k,\ell) - v(i-1,j,k-1,\ell+1)\Big] \nonumber  \\
            & \quad \qquad + \frac{\min\{\ell,C^G\} \mu_2}{d(j-1,k+1,\ell)}\Big[v(i-1,j,k+1,\ell-1) - v(i-1,j,k,\ell)\Big]  \nonumber \\
            &\quad \qquad +\frac{(j-1)\mu_0}{d(j-1,k+1,\ell)} \Big[ \min\{v(i,j-2,k+2,\ell),v(i,j-2,k+1,\ell+1)\}  \nonumber \\
            & \qquad \qquad - \min\{v(i,j-2,k+1,\ell+1),v(i,j-2,k,\ell+2)\} \Big]  \nonumber \\
            & \quad \qquad+ \frac{\mu_2\mathbb{I}\{\ell <C^G\}}{d(j-1,k+1,\ell)}\Big[v(i-1,j,k,\ell) - v(i-1,j,k,\ell)\Big]  \nonumber \\
            & \quad \qquad+ \frac{\mu_1-\mu_2\mathbb{I}\{\ell <C^G\}}{d(j-1,k+1,\ell)}\Big[v(i-1,j,k,\ell) - v(i,j-1,k,\ell+1)\Big]. \label{eq:larger_mu1_ub_general}
        \end{align}
        Consider the difference in the last term in \eqref{eq:larger_mu1_ub_general} with the coefficient $\frac{\mu_1-\mu_2\mathbb{I}\{\ell <C^G\}}{d(j-1,k+1,\ell)}$.
        \begin{enumerate} [label= \textbf{Subcase} (\alph*):, leftmargin=2.2\parindent]
            \item If $\frac{h_1}{\mu_1} < \frac{h_2}{\mu_2}$, referring to \eqref{eq:bound_triage_collab} gives
            \begin{align}
                \lefteqn{v(i-1,j,k,\ell) - v(i,j-1,k,\ell+1)}& \nonumber \\
                &\quad = \big(v(i-1,j,k,\ell) - v(i-1,j,k-1,\ell+1)\big) + \big(v(i-1,j,k-1,\ell+1) - v(i,j-1,k,\ell+1)\big) \nonumber \\
                &\quad \leq \left(\frac{h_1}{\mu_1} - \frac{h_2}{\mu_2}\right)  + \left(-\frac{h_1}{\mu_1}\right) \nonumber \\
                &\quad = - \frac{h_2}{\mu_2}, \label{eq:larger_mu1_ub_general_extra}
            \end{align}where the inequality applies the inductive hypothesis and \eqref{eq:bound_triage_noncollab}.
            \item If $\frac{h_1}{\mu_1} \geq \frac{h_2}{\mu_2}$, the inequality \eqref{eq:larger_mu1_ub_general_extra} follows by directly referring to \eqref{eq:bound_triage_collab}.
        \end{enumerate}
        The result follows by a comparable reasoning as obtaining \eqref{eq:larger_mu1_ub_base_result}, with the only distinction being the application of \eqref{eq:larger_mu1_ub_general_extra} in place of \eqref{eq:one_more_collab} for the difference with the coefficient $\frac{\mu_1-\mu_2\mathbb{I}\{\ell <C^G\}}{d(j-1,k+1,\ell)}$.  \qedhere
    \end{enumerate}
\end{proof}
    
\subsection{Proofs of supporting results} \label{sec:proof_support}
\begin{proof} [Proof of Proposition \ref{prop:larger_mu1_threshold}.]
    Without loss of generality, we assume Assumption~\ref{assm:preferred_collab} holds; otherwise, the result follows directly from Statement~\ref{state:lemma_diff_larger_mu1_ub} in Lemma~\ref{lemma:diff}, taking $N(j,k,\ell) = 0$. Furthermore, we focus on the case $i \geq 1$, since the result aims to characterize the behavior of $D(i,j,k,\ell)$ for all $i$ sufficiently large. 
    Note that $d(j-1,k+1,\ell) \geq d(j-1,k,\ell+1)$ since $\mu_1 \geq \mu_2$. 
    Thinning the MDPs in \eqref{eq:larger_mu1_threshold} with a higher rate $d(j-1,k+1,\ell)$ then yields 
    \begin{align}
        \lefteqn{v(i, j-1, k+1, \ell) - v(i, j-1, k, \ell+1)} \nonumber \\
        & \quad = \frac{h_1- h_2}{d(j-1,k+1,\ell)} + \frac{k\mu_1}{d(j-1,k+1,\ell)} \Big[v(i-1, j, k, \ell) - v(i-1, j, k-1, \ell+1)\Big] \nonumber \\
        & \quad + \frac{\min\{\ell,C^G\} \mu_2}{d(j-1,k+1,\ell)} \Big[v(i-1, j, k+1, \ell -1) -  v(i-1, j, k, \ell)\Big] \nonumber \\
        & \quad + \frac{(j-1)\mu_0}{d(j-1,k+1,\ell)} \Big[\min\{v(i, j-2, k+2, \ell),v(i, j-2, k+1, \ell+1)\} \nonumber \\
        &\qquad \qquad - \min\{v(i, j-2, k+1, \ell+1),v(i, j-2, k, \ell+2)\}\Big] \nonumber \\
        & \quad + \frac{\mu_2\mathbb{I}\{\ell<C^G\}}{d(j-1,k+1,\ell)}\Big[v(i-1,j,k,\ell) - v(i-1,j,k,\ell)\Big]  \nonumber \\
        & \quad + \frac{\mu_1 - \mu_2\mathbb{I}\{\ell<C^G\}}{d(j-1,k+1,\ell)} \Big[v(i-1, j, k, \ell) - v(i, j-1, k, \ell+1)\Big]. \label{eq:larger_mu1_threshold_general}
    \end{align}
    Consider the difference in the term with coefficient $\frac{(j-1)\mu_0}{d(j-1,k+1,\ell)}$ in \eqref{eq:larger_mu1_threshold_general}.
    Suppose $v(i, j-2, k+1, \ell+1) \leq v(i, j-2, k, \ell+2)$, replacing the first minimum with an upper bound $v(i, j-2, k+2, \ell)$ yields
    \begin{align}
        & \min\{v(i, j-2, k+2, \ell),v(i, j-2, k+1, \ell+1)\} - \min\{v(i, j-2, k+1, \ell+1),v(i, j-2, k, \ell+2)\} \nonumber \\
        &\quad \leq v(i, j-2, k+2, \ell) - v(i, j-2, k+1, \ell+1) \nonumber\\
        & \quad \leq \frac{h_1}{\mu_1} - \frac{h_2}{\mu_2}, \label{eq:larger_mu1_threshold_general_mu0_terms}
    \end{align}where the last step applies \eqref{eq:larger_mu1_ub}. If $v(i, j-2, k+1, \ell+1) > v(i, j-2, k, \ell+2)$, then the same bound in \eqref{eq:larger_mu1_threshold_general_mu0_terms} follows by taking the upper bound $v(i, j-2, k+1, \ell+1)$ in the first minimum and applying \eqref{eq:larger_mu1_ub}.
    
    Applying \eqref{eq:larger_mu1_threshold_general_mu0_terms} and \eqref{eq:bound_triage_collab_full} to the differences with coefficients $\frac{(j-1)\mu_0}{d(j-1,k+1,\ell)}$ and $\frac{\mu_1}{d(j-1,k+1,\ell)}$, respectively, and invoking \eqref{eq:larger_mu1_ub} twice for the terms in \eqref{eq:larger_mu1_threshold_general} with coefficients $\frac{k\mu_1}{d(j-1,k+1,\ell)}$ and $\frac{C^G\mu_2}{d(j-1,k+1,\ell)}$, we obtain
    \begin{align*}
        \lefteqn{v(i, j-1, k+1, \ell) - v(i, j-1, k, \ell+1)} \nonumber \\
        &\quad \leq \frac{h_1- h_2}{d(j-1,k+1,\ell)} + \frac{(j-1)\mu_0 + k\mu_1 + \min\{\ell,C^G\} \mu_2}{d(j-1,k+1,\ell)} \left(\frac{h_1}{\mu_1} - \frac{h_2}{\mu_2}\right)\\
        &\quad \qquad - \frac{\mu_1 - \mu_2\mathbb{I}\{\ell<C^G\}}{d(j-1,k+1,\ell)}\left(\frac{h_2}{\mu_2} + \frac{\mu_0}{d(j,k,\ell+1)} \left(\frac{(i+j)h_0}{C^p \mu_0} + \frac{i h_0}{C^p \mu_1} \right)\right)\\
        &\quad \leq \left(\frac{h_1}{\mu_1} - \frac{h_2}{\mu_2}\right) - \frac{\mu_1 - \mu_2\mathbb{I}\{\ell<C^G\}}{d(j-1,k+1,\ell)} \frac{\mu_0}{d(j,k,\ell+1)} \left(\frac{(i+j)h_0}{C^p \mu_0} + \frac{i h_0}{C^p \mu_1} \right)\\
        &\quad \leq \left(\frac{h_1}{\mu_1} - \frac{h_2}{\mu_2}\right) - \frac{\mu_1 - \mu_2\mathbb{I}\{\ell<C^G\}}{d(j-1,k+1,\ell)} \frac{\mu_0}{d(j,k,\ell+1)} \left(\frac{h_0}{C^p \mu_0} + \frac{h_0}{C^p \mu_1} \right)i,
    \end{align*}where the second step follows by noticing $\mu_1 - \mu_2\mathbb{I}\{\ell < C^G\} \geq \mu_1 - \mu_2$ and $d(j-1,k+1,\ell) = (j-1)\mu_0 + (k+1)\mu_1 + \min\{\ell,C^G\} \mu_2$, and the last inequality applies $\mu_1 - \mu_2\mathbb{I}\{\ell < C^G\} \geq \mu_1 - \mu_2 \geq 0$ by assumption.

    Define
    \begin{align*}
        C_1(j,k,\ell) &:= \frac{\mu_1 - \mu_2\mathbb{I}\{\ell<C^G\}}{d(j-1,k+1,\ell)} \frac{\mu_0}{d(j,k,\ell+1)} \left(\frac{h_0}{C^p \mu_0} + \frac{h_0}{C^p \mu_1} \right),\\
        C_2 &:= \frac{h_1}{\mu_1} - \frac{h_2}{\mu_2}.
    \end{align*}
    Observe that both $C_1(j,k,\ell)$ and $C_2$ are independent of $i$, and $C_2 > 0$ by assumption. 
    If either 1)$\mu_1 \geq \mu_2$ with $\ell \geq C^G$ or $\mu_1 > \mu_2$, we have $\mu_1 - \mu_2\mathbb{I}\{\ell<C^G\} > 0$, which implies that $C_1(j,k,\ell) > 0$. The desired result then follows by choosing
    \begin{align*}
        N(j,k,\ell) := \frac{C_2}{C_1(j,k,\ell)}.
    \end{align*}    
\end{proof}

\begin{proof} [Proof of Proposition \ref{prop:always_collab}.]
    The proof proceeds with induction on $M = 2(i+j)+k+\ell$, which captures the number of services remaining in state $(i,j,k,\ell)$. For the base case $i=k=\ell=0$ and $j=1$, which implies $M=1$, we have
    \begin{align*}
        v(0,0,1,0) - v(0,0,0,1) =  \frac{h_1}{\mu_1} - \frac{h_2}{\mu_2} \geq 0.
    \end{align*}
    Now assume the result holds at $M-1$ and consider it at $M$, where $M > 1$. There are two cases to discuss based on whether $i = 0$ or $i \geq 1$, which leads to different state transitions (and hence Bellman equations).
    \begin{enumerate} [label= \textbf{Case} \arabic*:, leftmargin=3\parindent]
        \item When $i = 0$, we have $d(j-1,k,\ell+1) \geq d(j-1,k+1,\ell)$ for $\ell < C^G$ under the assumption that $\mu_2 \geq \mu_1$. Thinning the MDPs in \eqref{eq:always_collab} with a higher rate $d(j-1,k,\ell+1)$ gives
        \begin{align}
            \lefteqn{v(0,j-1,k+1,\ell) - v(0,j-1,k,\ell+1)}& \nonumber \\
            & \quad = \frac{h_1-h_2}{d(j-1,k+1,\ell)} + \frac{k\mu_1}{d(j-1,k+1,\ell)}\Big[v(0,j-1,k,\ell) - v(0,j-1,k-1,\ell+1)\Big] \nonumber \\
            & \quad \qquad + \frac{\ell \mu_2}{d(j-1,k+1,\ell)}\Big[v(0,j-1,k+1,\ell-1) - v(0,j-1,k,\ell)\Big] \nonumber \\
            & \quad \qquad +\frac{(j-1)\mu_0}{d(j-1,k+1,\ell)} \Big[\min\{v(0,j-2,k+2,\ell),v(0,j-2,k+1,\ell+1)\} \nonumber \\
            & \qquad \qquad - \min\{v(0,j-2,k+1,\ell+1),v(0,j-2,k,\ell+2)\} \Big] \nonumber \\
            & \quad \qquad+ \frac{\mu_1}{d(j-1,k+1,\ell)}\Big[v(0,j-1,k,\ell) - v(0,j-1,k,\ell)\Big] \nonumber \\
            & \quad \qquad+ \frac{\mu_2-\mu_1}{d(j-1,k+1,\ell)}\Big[v(0,j-1,k+1,\ell) - v(0,j-1,k,\ell)\Big].\label{eq:always_collab_base}
        \end{align}
        Consider the term in \eqref{eq:always_collab_base} with coefficient $\frac{(j-1)\mu_0}{d(j-1,k+1,\ell)}$. Replacing the second minimum with the upper bound $v(0,j-2,k+1,\ell+1)$ yields
        \begin{align*}
            &\min\{v(0,j-2,k+2,\ell),v(0,j-2,k+1,\ell+1)\} - \min\{v(0,j-2,k+1,\ell+1),v(0,j-2,k,\ell+2)\} \\
            &\quad \geq \min\{v(0,j-2,k+2,\ell),v(0,j-2,k+1,\ell+1)\} - v(0,j-2,k+1,\ell+1)  \\
            &\quad = \min\{v(0,j-2,k+2,\ell)- v(0,j-2,k+1,\ell+1),0\} \\
            &\quad \geq 0, 
        \end{align*}where the last step uses the inductive hypothesis. 
        Applying the inductive hypothesis twice again shows that the terms with coefficients $\frac{k\mu_1}{d(j-1,k+1,\ell)}$ and $\frac{\ell \mu_2}{d(j-1,k+1,\ell)}$ in \eqref{eq:always_collab_base} are non-negative. Together with \eqref{eq:one_more_noncollab} for the term with coefficient $\frac{\mu_2-\mu_1}{d(j-1,k+1,\ell)}$, we obtain
        \begin{align*}
        v(0,j-1,k+1,\ell) - v(0,j-1,k,\ell+1) \geq \frac{h_1-h_2}{d(j-1,k+1,\ell)} + \frac{\mu_2-\mu_1}{d(j-1,k+1,\ell)}\frac{h_1}{\mu_1} \geq 0.
    \end{align*}
        \item If $i \geq 1$, thinning the MDPs in \eqref{eq:always_collab} with a higher rate $d(j-1,k,\ell+1)$ results in
        \begin{align*}
            & v(i,j-1,k+1,\ell) - v(i,j-1,k,\ell+1)\\
            & \quad = \frac{h_1-h_2}{d(j-1,k,\ell+1)} + \frac{k\mu_1}{d(j-1,k,\ell+1)}\Big[v(i-1,j,k,\ell) - v(i-1,j,k-1,\ell+1)\Big]\\
            & \quad \qquad + \frac{\ell \mu_2}{d(j-1,k,\ell+1)}\Big[v(i-1,j,k+1,\ell-1) - v(i-1,j,k,\ell)\Big]\\
            &\quad \qquad +\frac{(j-1)\mu_0}{d(j-1,k,\ell+1)} \Big[ \min\{v(i,j-2,k+2,\ell),v(i,j-2,k+1,\ell+1)\}\\
            & \qquad \qquad - \min\{v(i,j-2,k+1,\ell+1),v(i,j-2,k,\ell+2)\} \Big] \\
            & \quad \qquad+ \frac{\mu_1}{d(j-1,k,\ell+1)}\Big[v(i-1,j,k,\ell) - v(i-1,j,k,\ell)\Big]\\
            & \quad \qquad+ \frac{\mu_2-\mu_1}{d(j-1,k,\ell+1)}\Big[v(i,j-1,k+1,\ell) - v(i-1,j,k,\ell)\Big].
        \end{align*}Notice that the non-negativity of the terms with coefficients $\frac{(j-1)\mu_0}{d(j-1,k,\ell+1)}$, $\frac{k\mu_1}{d(j-1,k,\ell+1)}$, and $\frac{\ell \mu_2}{d(j-1,k,\ell+1)}$ follows analogously to the case where $i = 0$. In conjunction with the application of \eqref{eq:bound_triage_noncollab} to the difference with coefficient $\frac{\mu_2-\mu_1}{d(j-1,k+1,\ell)}$ in \eqref{eq:always_collab_base} yields
        \begin{align*}
            v(i,j-1,k+1,\ell) - v(i,j-1,k,\ell+1) \geq \frac{h_1-h_2}{d(j-1,k+1,\ell)} + \frac{\mu_2-\mu_1}{d(j-1,k+1,\ell)}\frac{h_2}{\mu_2}\geq 0.
        \end{align*}
    \end{enumerate}
    If $\frac{h_1}{\mu_1} > \frac{h_2}{\mu_2}$ holds strictly, i.e., Assumption \ref{assm:preferred_collab} is satisfied, a similar argument shows that the inequality in \eqref{eq:always_collab} also holds strictly.
\end{proof}

\begin{proof} [Proof of Corollary \ref{cor:enough_GP}.]
    Recall that $(i,j,k,\ell) \in \X_D$ implies $j \geq 1$ and $(i,j,k,\ell) \in \X$. Since $C^G \geq C^p$, we have $\ell \leq C^p - j - k \leq C^p - 1$, and thus $\ell < C^G$ must hold. Consequently, if $\frac{h_1}{\mu_1} > \frac{h_2}{\mu_2}$, the Bellman equations for the case where $C^G \geq C^p$ coincide with those for $C^G < C^p$ and $\ell < C^G$. A similar argument applies when establishing the corresponding statements in Propositions \ref{prop:always_collab} and \ref{prop:larger_mu1_threshold}. 
    
    On the other hand, if $\frac{h_1}{\mu_1} \leq \frac{h_2}{\mu_2}$ under $C^G \geq C^p$, the impacts of collaborative and non-collaborative care on the system are symmetric. This symmetry in system dynamics implies a corresponding symmetry in the preference for collaboration across the parameter space.
\end{proof}

\begin{proof} [Proof of Proposition \ref{prop:noncollab}.]
Consider the case where $\ell \geq C^G$ first. Define the following sequence.
    Let $\{S_{m,a}, m \geq 1, a \in \{0,1,2\}\}$ be the service time of the $m^{\text{th}}$ patient when/if they are served at Station $n$, where $n = 0, 1, 2$ denotes the station triage, non-collaborative service and collaborative service, respectively.
    Notice that any patient just beyond the triage phase  can begin service immediately (by not requesting collaboration), or join the collaboration station where other patients, along with their assigned NPs, may already be in service or waiting. 
    
    Consider the state $(i, j, k, \ell)$, where $j \geq 1$. When a triage service is complete, the decision-maker needs to compare $v(i, j-1, k+1, \ell)$ and $v(i, j-1, k, \ell+1)$. Under any fixed policy, an accounting of the remaining costs associated with each patient that has just completed triage include the costs associated with any time in service (at one of the two remaining stations) and their impact on the waiting time of patients in the triage queue by virtue of the NP being away from the triage station.
    
    Since non-collaborative service begins immediately after the assignment, the cost associated with that service if this is the $m^{\text{th}}$ patient is $h_1 S_{m,1}$.
    Similarly, if the patient is assigned to get collaborative care.
    For $\ell \geq C^G$, the $m^{\text{th}}$ patient needs to wait for the $\ell+1-C^G$ patients in front of them before beginning service, plus their own service time. Renumbering those $\ell+1-C^G$ patients to be $\{1, 2, \ldots, \ell+1-C^G\}$ yields that the total cost contributed by patient $m$ is $h_2 (S_{m,2} + \sum_{b=1}^{\ell+1-C^G} S_{b,2})$. Taking expectations yields
    \begin{align*}
        &\text{cost of non-collaborative service}  = \frac{h_1}{\mu_1} \\
            & \quad \leq \frac{\ell + 1}{C^G}\frac{h_2}{\mu_2} =  \big( \frac{\ell + 1 - C^G}{C^G} + 1 \big) \frac{h_2}{\mu_2} = \, \text{cost of collaborative service},
    \end{align*}
    where the inequality follows by assumption. Notice that the collaborative service rate of the $\ell+1-C^G$ patients is given by $C^G\mu_2$.
    In addition to this cost, the decision to choose to serve collaboratively or without collaboration has an impact on the upstream patients. In particular, the time away from the triage station reduces the rate at which triage can be performed. Recall, the assumption $\frac{1}{\mu_1} \leq \frac{\ell +1}{C^G\mu_2} = \big(\frac{\ell + 1 - C^G}{C^G} + 1 \big) \frac{1}{\mu_2}$ says that the expected time the NP works at the non-collaborative station (before returning to the triage station) is less than it would be if they were assigned to the collaborative station. This means that the expected amount of work they can do if assigned to the non-collaborative station is higher than that if assigned to the collaborative station.
    Since the expected total cost is lower when assigned to Station 1, which is free of queueing effects, the result in Statement \ref{state:noncollab_queue_in_collab} follows.

    The case where $\ell < C^G$ follows by a comparable argument (or by applying Statement \ref{state:lemma_diff_larger_mu1_ub} in Lemma \ref{lemma:diff}).
\end{proof}

\subsection{Proofs of the results in Heuristic Design}
\label{sec:proof_heuristic}
\begin{proof} [Proof of Proposition \ref{prop:heur_no_queue_in_collab}.]
    Since both systems operate under the \textit{deterministic model assumptions} in Assumption \ref{assum:deterministic}, downstream services commence immediately after the decision point, with each requiring exactly $\frac{1}{\mu_n}$ units of time at Station $n$ for $n = 1,2$. This leads to immediate time and cost difference to complete downstream services equal to $\frac{1}{\mu_1}-  \frac{1}{\mu_2}$ and $\frac{h_1}{\mu_1}-  \frac{h_2}{\mu_2} = b$, respectively.
    Moreover, with two systems controlled by the \textit{always-collaborative policy $\pi^{co\ell}$}, the total cost difference incurred downstream equals the corresponding immediate cost difference, which confirms \eqref{eq:H12}.
    It remains to calculate $H_0(i,j,k,\ell)$ based on $h_0$.
    
    The holding cost at Station 0, starting at state $s \in \X$ under $\pi^{co\ell}$ until all work is completed is
    \begin{align*}
        \left.\int_{0}^{\infty}h_0\left((Q^{\pi^{co\ell}}_{0,0}(t) + Q^{\pi^{co\ell}}_{0,1}(t)\right) dt \right\vert_{s} = h_0 \left.\int_{0}^{\infty} Q^{\pi^{co\ell}}_{0}(t)dt\right\vert_{s},
    \end{align*}
    where $Q_0^{\pi^{co\ell}}(t):=Q^{\pi^{co\ell}}_{0,0}(t) + Q^{\pi^{co\ell}}_{0,1}(t), \forall t \geq 0$ denotes the total number of patients in the upstream phase (including both queueing and receiving the triage service). Note that $Q_0^{\pi^{co\ell}}(t)$ now operates following the \textit{deterministic model assumptions}, in contrast to the stochastic evolution described in \eqref{eq:total_cost_sto}.
    Therefore,
    \begin{align}
        H_0(i,j,k,\ell) &= 
        h_0 \cdot D_A(i,j,k,\ell), \label{eq:H0_area}
    \end{align}where
    \begin{align}
        D_A(i,j,k,\ell)&=\left. \int_{0}^{\infty} \Big(Q^{\pi^{co\ell}}_{0}(t) \right\vert_{(i,j-1,k+1,\ell)} - \left. Q^{\pi^{co\ell}}_{0} \right\vert_{(i,j-1,k,\ell+1)}(t) \Big)dt \label{eq:D_area}
    \end{align} 
    is the approximation of $c_0^1-c_0^2$ (given in \eqref{eq:D0_and_D12}) under the \textit{deterministic model assumptions} and the \textit{always-collaborative policy $\pi^{co\ell}$}.
    
    It remains to compute $D_A(i,j,k,\ell)$ by analyzing the evolution of the total number of patients in the upstream phase for both systems.
    Each system starts with $i + j - 1$ patients and proceeds until the upstream is fully cleared.
    Recall that System 2 starts at state $(i, j-1, k, \ell+1)$. The $k$ decision makers initially at Station 1 (the non-collaborative station) become available at time $\frac{1}{\mu_1}$ and immediately return to the upstream phase to interact with the first $k$ patients at the front of the queue (among the $i$ waiting patients). This occurs no later than $\frac{1}{\mu_2}$ (since $\mu_1 \geq \mu_2$), when $\text{NP}^*_{(2)}$ and other $\ell$ NPs complete their current service and leave Station 2 (the collaborative station). 
    As a result, these $\ell+1$ NPs who start at Station 2, proceeds to serve the next $\ell+1$ patients among the remaining $i - k$ in the queue.
    In parallel, in System 1, $\text{NP}^*_{(1)}$ and other $k$ decision makers initially at Station 1 return to the upstream Station 0 at time $\frac{1}{\mu_1}$ and begin serving the first $k+1$ patients among the $i$ in queue, while the $\ell$ NPs who start at Station 2 return to upstream at time $\frac{1}{\mu_2}$.

    Without loss of generality, we align the patient order served by $\text{NP}^*$ and non-$\text{NP}^*$ decision makers across the two systems: assume that the $k$ (resp. $\ell$) decision makers from the non-collaborative (resp. collaborative) station serve the first $k$ (resp. $\ell$) patients among the initial $i$ (resp. $i-k-1$) in the queue. This ensures that the first patient that $\text{NP}^*$ in each system interacts is the first among the remaining $i-k$. 
    Under the \textit{always collaborative policy} ($a=1$ for all NPs), a repeating pattern emerges in which $\text{NP}^*$ (in each system) contributes one service for every $C^p$ triage services rendered. Consequently, $\text{NP}^*_{(1)}$ and $\text{NP}^*_{(2)}$ each provides triage services exactly $\left\lceil \frac{i - k}{C^p} \right\rceil$ times until all upstream patients are completed. 
    
    Under this service order, whenever a decision is made by a decision maker other than $\text{NP}^*_{(1)}$ in System 1, thus reducing the number of upstream patients by one, there is a corresponding decision maker in System 2 (other than $\text{NP}^*_{(2)}$) who takes the same action, specifically $a = 1$, at the same time, yielding an identical reduction. 
    In contrast, actions taken by $\text{NP}^*_{(1)}$ in System 1 occur earlier than those by $\text{NP}^*_{(2)}$ in System 2 exactly by $\left(\frac{1}{\mu_0} + \frac{1}{\mu_2}\right) - \left(\frac{1}{\mu_0} + \frac{1}{\mu_1}\right) = \frac{1}{\mu_2} - \frac{1}{\mu_1}$
    units of time. Each such action decreases the upstream population by one in System 1, and does so $\frac{1}{\mu_2} - \frac{1}{\mu_1}$ units of time ahead of the corresponding reduction in System 2, contributing a difference in the integral (i.e., signed area) between the two systems equal to 
    \begin{align*}
        1 \cdot \left(\frac{1}{\mu_1}-  \frac{1}{\mu_2}\right) = \frac{1}{\mu_1}-  \frac{1}{\mu_2}.
    \end{align*}
    So far, we know that $\text{NP}^*_{(1)}$ and $\text{NP}^*_{(2)}$ each provide triage services $\left\lceil \frac{i - k}{C^p} \right\rceil$ times until all upstream patients are cleared. Moreover, each instance in which the timing of their triage services (or, equivalently, their decision points) differs contributes an incremental difference of $\frac{1}{\mu_1}-  \frac{1}{\mu_2}$ in the integral. 
    Therefore, the total difference between the two integrals in \eqref{eq:H0_area} is
    \begin{align}
        D_A(i,j,k,\ell)&= \left\lceil\frac{i-k}{C^p}\right\rceil \left(\frac{1}{\mu_1}-  \frac{1}{\mu_2}\right). \label{eq:diff_upstream_area}
    \end{align}
    Figure \ref{fig:deterministic} gives an illustrative example for computing $D_A(i,j,k,\ell)$ as given by \eqref{eq:diff_upstream_area}, with parameters $C^G = C^p = 4$, $\mu_0 = 1$, $\mu_1 = 2$, and $\mu_2 = 0.5$. Systems 1 and 2 start in states $(18,1,2,1)$ and $(18,1,1,2)$, respectively, corresponding to the initial decision state $(i, j, k, \ell) = (18, 2, 1, 1)$. In this case, the signed area of each rectangle is $\frac{1}{\mu_1} - \frac{1}{\mu_2} = -1.5$, and there are $\left\lceil\frac{i-k}{C^p}\right\rceil = \left\lceil\frac{18-1}{4}\right\rceil = 5$ such rectangles.
    
     \begin{figure}[htbp]
    \centering
        \includegraphics[scale=0.33]{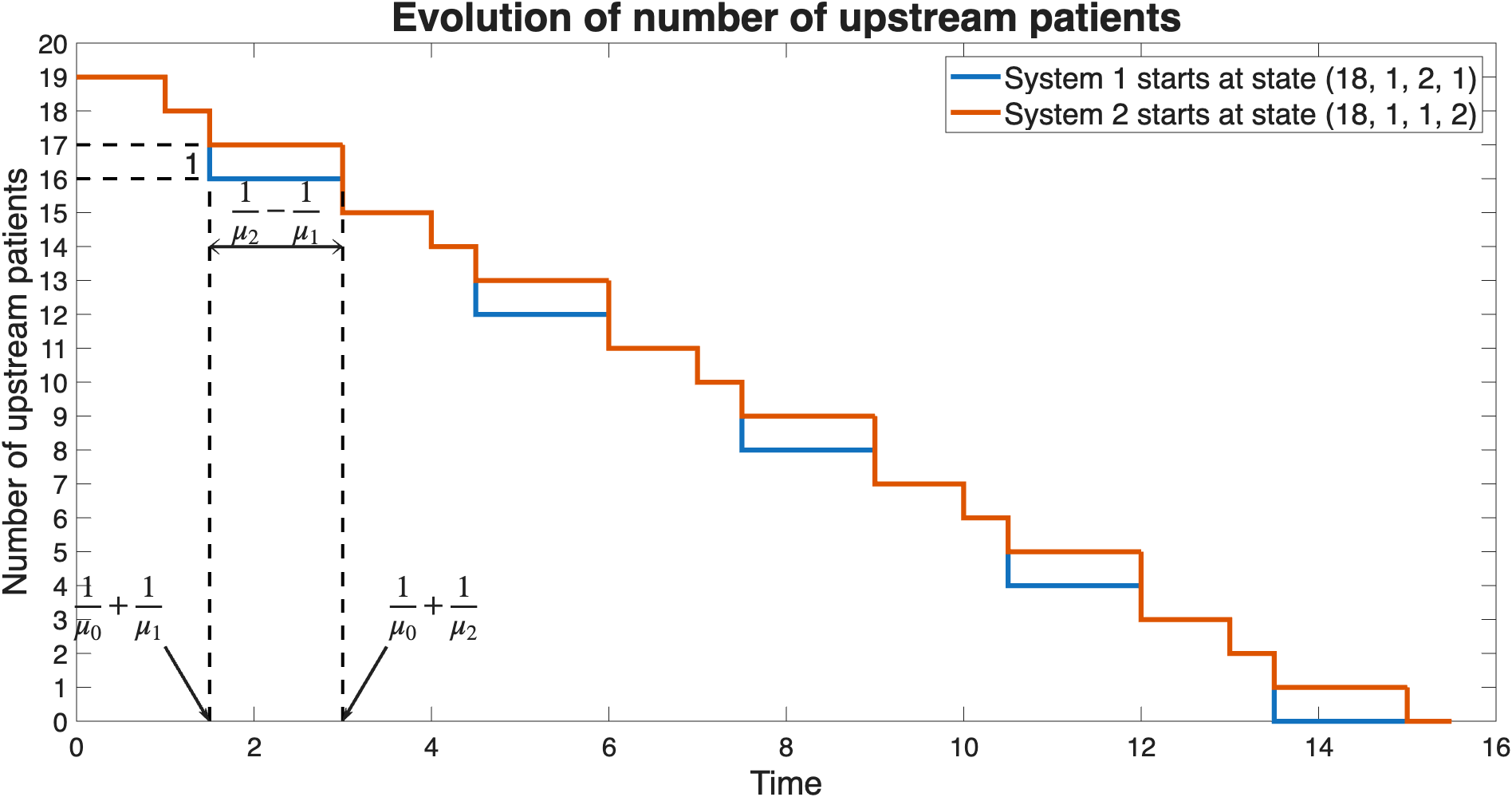}
        \captionsetup{font=small}
    \caption{Evolution of the number of upstream patients under \textit{deterministic model assumptions} and \textit{always collaborative policy $\pi^{co\ell}$}. System 1: $(18,1,2,1)$ and System 2: $(18,1,1,2)$.} \label{fig:deterministic}
    \end{figure}  
    Substituting \eqref{eq:diff_upstream_area} into \eqref{eq:H0_area} then yields the desired result in \eqref{eq:H0}.
\end{proof}

\begin{proof} [Proof of Proposition \ref{prop:heur_queue_in_collab_small_mu0}]
    Notice that if the decision is to seek collaborative care for the current patient ($a=1$) at $(i,j,k,\ell)$, where $\ell \geq C^G$, there are $\ell$ patients ahead (including both in line and receiving service). Under the \textit{modified deterministic model assumptions} given in Assumption \ref{assum:deterministic_initial_block}, this leads to a time difference between Systems 1 and 2 of $\frac{1}{\mu_1} - \frac{\ell+1}{C^G \mu_2}$ and a corresponding cost difference of $\frac{h_1}{\mu_1} - \frac{\ell+1}{C^G} \frac{h_2}{\mu_2} = b_\ell$.
    Moreover, since both systems operate under the \textit{always collaborative policy} $\pi^{co\ell}$ given in Assumption \ref{assm:always-collab}, the total downstream cost difference equals this immediate cost difference, thereby confirming \eqref{eq:small_mu0_H12}.

    The expression in \eqref{eq:H0_area} remains valid for computing $H_0^0(i,j,k,\ell)$ and again reduces to evaluating the integral difference $D_A(i,j,k,\ell)$ in \eqref{eq:D_area} (allowing for Assumption \ref{assum:deterministic} replaced by Assumption \ref{assum:deterministic_initial_block}). The proof builds upon Proposition~\ref{prop:heur_no_queue_in_collab} via a similar line of reasoning. Under the modeling assumptions and the always-collaborative policy, the sequence of the NPs completing downstream service and returning to upstream service repeats until all upstream jobs are cleared.
 
    Note that the number of upstream jobs may not be divisible by $C^p$, resulting in an incomplete final cycle in which some NPs perform exactly one fewer service than others. As in Proposition \ref{prop:heur_no_queue_in_collab}, we group NPs across the two systems into \textbf{pairs}, where each pair serves the same sequence of upstream patients (from the front of the queue to the end) across both systems. The problem then reduces to computing the cumulative difference in integrals incurred by each pair's triage actions and summing over all pairs.
    
    For any pair $m$, $m = 1, 2, \ldots, C^p$, suppose the first patient served by this pair is the first among the $i_m$ upstream patients. Then, due to the repeating $C^p$-length cycle, the total number of services performed by this pair is $\left\lceil\frac{i_m}{C^p}\right\rceil$. Since each triage completion by a pair reduces the upstream population by one in both systems, but possibly occurs at a time offset, say $d_m$. Each action contributes an incremental difference of $1 \cdot d_m = d_m$ to the integral. Therefore, the total contribution of this pair to the integral difference in \eqref{eq:D_area} is 
    \begin{align}
        \left\lceil\frac{i_m}{C^p}\right\rceil \cdot d_m. \label{eq:pair_m}
    \end{align}
    Summing over all $C^p$ pairs yields the final expression for $D_A(i,j,k,\ell)$ in \eqref{eq:D_area}:
    \begin{align}
        D_A(i,j,k,\ell) = \sum_{m=1}^{C^p} \left\lceil\frac{i_m}{C^p}\right\rceil \cdot d_m. \label{eq:D_area_sum}
    \end{align}
    
    To obtain \eqref{eq:small_mu0_H0}, it remains to calculate $i_m$ and $d_m$. Recall that Systems 1 and 2 begin in states $(i, j - 1, k + 1, \ell)$ and $(i, j - 1, k, \ell + 1)$, respectively. For notational clarity, define Pair $m$ as the pair of NPs in Systems 1 and 2 who are the $m^{\text{th}}$ to complete a triage service for the first time in their respective systems. Ties are permitted, and their relative ordering is inconsequential. 
    The subsequent analysis distinguishes between two cases: $\mu_1 \geq \mu_2$ and $\mu_1 < \mu_2$.

    \begin{enumerate} [label= \textbf{Case} \arabic*:,ref={ \arabic*}, leftmargin=2.2\parindent]
        \item Suppose $\frac{1}{\mu_1} \leq \frac{1}{\mu_2}$, or equivalently, $\mu_1 \geq \mu_2$. The behavior of the $C^p$ pairs is characterized as follows:
        \begin{itemize}[leftmargin=0pt, labelsep=0.5em, itemsep=0pt, topsep=0pt]
            \item The first $j - 1$ pairs, initially poised to begin triage in both systems, complete service simultaneously at time $\frac{1}{\mu_0}$, yielding $d_m = 0$ and no contribution to the integral difference in \eqref{eq:D_area}.
            \item Since $\frac{1}{\mu_1} \leq \frac{1}{\mu_2}$, the next $k$ pairs begin at Station 1 and complete both Station 1 and subsequent triage services in sync across the two systems, again with $d_m = 0$. These pairs serve the first $k$ patients among the initial $i$, leaving $i - k$ patients in the upstream.
            \item Pair $m = (j - 1) + k + 1$, consists of $\text{NP}^*_{(1)}$, who starts at Station 1 in System 1, and a decision-maker initiating collaborative service in System 2. 
            This pair serves the first patient among the remaining $i_m = i - k$ and experiences a triage completion time difference $d_m = \frac{1}{\mu_1} - \frac{1}{\mu_2}$ due to the discrepancy in downstream service durations (decision points occur exactly $\frac{1}{\mu_0}$ after downstream service completions).
            \item The following $C^G - 1$ pairs simultaneously begin and complete collaborative service in both systems, resulting in $d_m = 0$, and serve the next $C^G - 1$ (among $i-k-1$) upstream patients.
            \item The remaining $\ell + 1 - C^G$ pairs, specifically Pair $m = (j-1) + k + 1 + r$ for $r = C^G, C^G+1, \ldots, \ell$, commence at Station 2 in both systems. They complete downstream service at times $\frac{r}{C^G \mu_2}$ and $\frac{r + 1}{C^G \mu_2}$ in Systems 1 and 2, respectively, and immediately begin triage for the first of the remaining $i - (k + r)$ upstream patients. Each such pair incurs $d_m = -\frac{1}{C^G \mu_2}$.
        \end{itemize}
        Summing the contributions given by \eqref{eq:pair_m} over all pairs yields the following expression for the integral difference in \eqref{eq:D_area_sum}: 
        \begin{align}
            D_A(i,j,k,\ell) & = \left\lceil\frac{i-k}{C^p}\right\rceil \left(\frac{1}{\mu_1} - \frac{1}{\mu_2}\right)  - \sum_{r=C^G}^\ell \left\lceil\frac{i-k-r}{C^p}\right\rceil \frac{1}{C^G \mu_2}. \label{eq:D_area_large_mu1}
        \end{align}
        \item We now consider the case $\frac{1}{\mu_2} < \frac{1}{\mu_1} \leq \frac{\ell+1}{C^G \mu_2}$. 
        From the definition of $\ell'$ in \eqref{eq:def_ell_prime}, it follows that $\frac{\ell'}{C^G \mu_2} < \frac{1}{\mu_1} \leq \frac{\ell'+1}{C^G \mu_2}$ and $C^G \leq \ell' \leq \ell$. 
        The quantity $\ell'$ represents the number of NPs initially at Station 2 in System 1 who complete downstream service strictly before $\text{NP}^*_{(1)}$. We again compute $d_m$ and $i_m$ for each pair $m$:
        \begin{itemize}[leftmargin=0pt, labelsep=0.5em, itemsep=0pt, topsep=0pt]
            \item The first $j-1$ pairs, initially poised to begin triage in both systems, contribute nothing to the integral difference, analogously to the case $\mu_1 \geq \mu_2$.
            \item Since $\frac{\ell'}{C^G \mu_2} < \frac{1}{\mu_1}$, the next $\ell'$ pairs (Pair $m = j, \ldots, (j-1) + \ell'$) comprise NPs starting at Station 2 in both systems and completing downstream service before $\text{NP}^*_{(1)}$. The first $C^G$ pairs finish at time $\frac{1}{\mu_2}$; each subsequent Pair $m = (j-1)+r$ for $r = C^G+1, C^G+2, \ldots, \ell'-C^G$ completes service at $\frac{r}{C^G \mu_2}$. Each pair returns to the upstream simultaneously, yielding $d_m = 0$. These pairs collectively serve the first $\ell'$ patients, leaving $i - \ell'$ upstream.
            \item The following $k$ pairs, starting at Station 1 in both systems, return at $\frac{1}{\mu_1}$ without delay, leaving $i - \ell' - k$ patients upstream. 
            \item Pair $m = (j-1) + \ell' + k +1$ consists of $\text{NP}^*_{(1)}$, who completes Station 1 at $\frac{1}{\mu_1}$ in System 1, and a decision-maker who leaves Station 2 at $\frac{\ell'+1}{C^G \mu_2}$ in System 2. This pair serves the first of the remaining $i_m = i-\ell'-k$ upstream patients and incurs $d_m = \frac{1}{\mu_1} - \frac{\ell'+1}{C^G \mu_2}$.
            \item The remaining $\ell - \ell'$ pairs (Pair $m = (j-1) + k + r + 1$ for $r = \ell'+1, \ell'+2, \ldots, \ell$) all begin in the Station 2 queue. For each such pair, the NP in System 1 departs at $\frac{r}{C^G \mu_2}$, while their counterpart in System 2 departs at $\frac{r+1}{C^G \mu_2}$, yielding $d_m = -\frac{1}{C^G \mu_2}$. Each pair serves the first patient among the remaining $i_m = i - k - r$ in the upstream.
        \end{itemize}
        Summing all pairwise contributions again gives the expression for $D_A(i,j,k,\ell)$ in \eqref{eq:D_area_sum}: 
        \begin{align}
            D_A(i,j,k,\ell) & = \left\lceil\frac{i-\ell'-k}{C^p}\right\rceil \left(\frac{1}{\mu_1} - \frac{\ell'+1}{C^G \mu_2}\right)  - \sum_{r=\ell'+1}^{\ell} \left\lceil\frac{i-k-r}{C^p}\right\rceil \frac{1}{C^G \mu_2}. \label{eq:D_area_mid_mu2}
        \end{align}
        \item Finally, consider $\frac{1}{\mu_1} > \frac{\ell+1}{C^G \mu_2}$, where all NPs at Station 2 in both systems complete their initial services and return upstream before those starting at Station 1.
        \begin{itemize}[leftmargin=0pt, labelsep=0.5em, itemsep=0pt, topsep=0pt]
            \item The behavior of the first $j-1$ pairs again mirrors previous cases and contributes nothing to \eqref{eq:D_area_sum}.
            \item Since $\frac{\ell}{C^G \mu_2} < \frac{\ell+1}{C^G \mu_2} < \frac{1}{\mu_1}$, the next $\ell$ pairs (Pair $m = j, j+1, \ldots, (j-1) + \ell$) behave analogously to the case where $\frac{1}{\mu_2} \leq \frac{1}{\mu_1} < \frac{\ell+1}{C^G \mu_2}$ with $\ell'$ replaced by $\ell$. They return upstream simultaneously, contributing no difference to the integral, and leave $i - \ell$ patients in the upstream.
            \item Since $\frac{\ell+1}{C^G \mu_2} < \frac{1}{\mu_1}$, Pair $m = (j-1) + \ell + 1$ comprises $\text{NP}^*_{(1)}$ and $\text{NP}^*_{(2)}$, who complete their initial services at times $\frac{1}{\mu_1}$ and $\frac{\ell+1}{C^G\mu_2}$, respectively, incurring $d_m = \frac{1}{\mu_1} - \frac{\ell+1}{C^G\mu_2}$. This pair serves the first of the remaining $i_m = i-\ell$ upstream patients.
            \item The final $k$ pairs begin at Station 1 in both systems and return to the upstream simultaneously at $\frac{1}{\mu_1}$, contributing nothing to the integral difference.
        \end{itemize}
        Since only one pair contributes to the integral difference in \eqref{eq:D_area_sum}, we obtain
        \begin{align}
            D_A(i,j,k,\ell) & = \left\lceil\frac{i-\ell}{C^p}\right\rceil \left(\frac{1}{\mu_1} - \frac{\ell+1}{C^G \mu_2}\right). \label{eq:D_area_large_mu2}
        \end{align}
    \end{enumerate}   
    Substituting \eqref{eq:D_area_large_mu1}-\eqref{eq:D_area_large_mu2} into \eqref{eq:H0_area} completes the proof and results in \eqref{eq:small_mu0_H0}.
\end{proof}

\begin{proof} [Proof of Proposition \ref{prop:large_mu0_D_sign}.]
    Consider Statement \ref{state:large_mu0_no_decision} first. 
    By Statement 1 of Lemma 4.4 in \citep{lu2026balancingindependentcollaborativeservice}, for any $(i,k,\ell) \in \tilde{\X}_{red}$ with $k \geq 1$ and $\ell \geq C^G$, we have
    \begin{align}
        \tilde{D}(0,k,\ell) = \tilde{v}(0,k,\ell) - \tilde{v}(0,k-1,\ell+1) = \frac{h_1}{\mu_1} - \frac{(\ell+1)h_2}{C^G\mu_2}. \label{eq:diff_bdry}
    \end{align}
    
    When $i=0$, substituting \eqref{eq:diff_bdry} into the definition of $G$ in \eqref{eq:diff_of_min} yields
    \begin{align}
        G(0,j,k,\ell)& = \left\{
        \begin{array}{lcl}
            \tilde{v}(0,k+j,\ell) - \tilde{v}(0,k+j-1,\ell+1), & \text{if} & \frac{h_1}{\mu_1} \leq \frac{(\ell+1) h_2}{C^G\mu_2} \\
            \tilde{v}(0,k+j-1,\ell+1) - \tilde{v}(0,k+j-1,\ell+1), & \text{if} & \frac{(\ell+1) h_2}{C^G\mu_2} < \frac{h_1}{\mu_1} \leq \frac{(\ell+2) h_2}{C^G\mu_2} \\
            \qquad \qquad \qquad \qquad\vdots & & \qquad
            \qquad \vdots\\
            \tilde{v}(0,k+1,\ell+j-1) - \tilde{v}(0,k,\ell+j), & \text{if} & \frac{h_1}{\mu_1} > \frac{(\ell+j) h_2}{C^G\mu_2} \\
        \end{array} \nonumber
        \right. \\
        & = \left\{
        \begin{array}{lcl}
            \tilde{D}(0,k+j,\ell), & \text{if} & \frac{h_1}{\mu_1} \leq \frac{(\ell+1) h_2}{C^G\mu_2}, \\
            0, & \text{if} & \frac{(\ell+1) h_2}{C^G\mu_2} < \frac{h_1}{\mu_1} \leq \frac{(\ell+j) h_2}{C^G\mu_2}, \\
            \tilde{D}(0,k+1,\ell+j-1), & \text{if} & \frac{h_1}{\mu_1} > \frac{(\ell+j) h_2}{C^G\mu_2}.
        \end{array}
        \right. \label{eq:large_mu0_D_no_decision_inte}
    \end{align}
    A final substitution of \eqref{eq:diff_bdry} into \eqref{eq:large_mu0_D_no_decision_inte} yields \eqref{eq:large_mu0_D_no_decision}.
    
    We now turn to Statement \ref{state:large_mu0_monotone}.
    \begin{align}
        \lefteqn{G(i,j,k,\ell) - G(i+1,j,k,\ell)}& \nonumber \\
        &\quad = \min_{m=0,1, \ldots, j-1}\Big\{\tilde{v}\big(i, k+1+(j-1-m), \ell+m\big)\Big\} - \min_{m=0,1, \ldots, j-1}\Big\{\tilde{v}\big(i, k+(j-1-m), \ell+1+m\big)\Big\} \label{eq:large_mu0_monotone_inte} \\
        &\quad \qquad - \min_{m=0,1, \ldots, j-1}\Big\{\tilde{v}\big(i+1, k+1+(j-1-m), \ell+m\big)\Big\} + \min_{m=0,1, \ldots, j-1}\Big\{\tilde{v}\big(i+1, k+(j-1-m), \ell+1+m\big)\Big\}. \nonumber
    \end{align}
    For any $i \geq 0$, the discussion is divided into the following cases. 
    \begin{enumerate}[label= \textbf{Case} \arabic*:, leftmargin=2.2\parindent]
        \item Suppose the first term is minimized at $m = 0$, giving the value $\tilde{v}(i,k+j,\ell)$. Also, let the last minimum be attained at $m' \in \{0,1, \ldots, j-1\}$, so its minimizing value is $\tilde{v}\big(i+1, k+(j-1-m'), \ell+1+m'\big)$. Using upper bounds $\tilde{v}(i,k+j-1,\ell+1)$ and $\tilde{v}\big(i+1, k+1+(j-1-m'), \ell+m'\big)$ for the second and third minima in \eqref{eq:large_mu0_monotone_inte}, respectively, we obtain
        \begin{align*}
            \lefteqn{G(i,j,k,\ell) - G(i+1,j,k,\ell)}&\\
            &\quad \geq \tilde{v}(i,k+j,\ell) - \tilde{v}(i,k+j-1,\ell+1) \\
            &\quad \qquad - \tilde{v}\big(i+1, k+1+(j-1-m'), \ell+m'\big) + \tilde{v}\big(i+1, k+(j-1-m'), \ell+1+m'\big) \\
            &\quad = \tilde{D}(i,k+j,\ell) - \tilde{D}\big(i+1,k+1+(j-1-m'),\ell+m'\big)\\
            &\quad \geq \tilde{D}(i,k+j,\ell) - \tilde{D}\big(i,k+1+(j-1-m'),\ell+m'\big)\\
            &\quad  = \sum_{r= 0}^{m'-1} \Big( \tilde{D}\big(i,k+1+(j-1-r),\ell+r\big) - \tilde{D}\big(i,k+(j-1-r),\ell+1+r\big)\Big)\\
            &\quad \geq 0,
        \end{align*}where the second inequality applies Proposition 3.5 when $\mu_1\geq\mu_2$ and Proposition 3.7 when $\mu_2>\mu_1$ in \citep{lu2026balancingindependentcollaborativeservice}, and the last inequality follows from Proposition 3.11 in the same work.
        \item Suppose the last minimum is attained at $m = j-1$, i.e, the minimizing value is $\tilde{v}(i+1, k, \ell+j)$, a similar argument as the previous case implies that \eqref{eq:large_mu0_monotone_inte} is non-negative.
        \item Finally, consider the case where the first and last minima are attained at $m_1 \in \{1, 2, \ldots, j-1\}$ and $m_2 \in \{0, 1, \ldots,j-2\}$, respectively. Then the first term attains the value $\tilde{v}\big(i, k+1+(j-1-m_1), \ell+m_1\big)$, which also serves as an upper bound for the second. Likewise, the last term attains $\tilde{v}\big(i+1, k+(j-1-m_2), \ell+1+m_2\big)$, which upper bounds the third. We thus obtain:
        \begin{align*}
            \lefteqn{G(i,j,k,\ell) - G(i+1,j,k,\ell)}& \nonumber \\
            &\quad \geq \tilde{v}\big(i, k+1+(j-1-m_1), \ell+m_1\big) - \tilde{v}\big(i, k+1+(j-1-m_1), \ell+m_1\big) \nonumber\\
            &\quad \qquad - \tilde{v}\big(i+1, k+(j-1-m_2), \ell+1+m_2\big) + \tilde{v}\big(i+1, k+(j-1-m_2), \ell+1+m_2\big)\\
            &\quad = 0.
    \end{align*}
    \end{enumerate}
    It remains to show that $G(i,j,k,\ell)$ becomes negative as $i$ increases.  
    Suppose the second minimum in \eqref{eq:diff_of_min} is attained at some $m' \in \{0,1, \ldots, j-1\}$, i.e., 
    \begin{align*}
        \min_{m=0,1, \ldots, j-1}\Big\{\tilde{v}\big(i, k+(j-1-m), \ell+1+m\big)\Big\} = \tilde{v}\big(i, k+(j-1-m'), \ell+1+m'\big)
    \end{align*}
    Using an upper bound $\tilde{v}\big(i, k+1+(j-1-m'), \ell+m'\big)$ for the first minimum in \eqref{eq:diff_of_min} yields
    \begin{align}
        G(i,j,k,\ell)&\leq \tilde{v}\big(i, k+1+(j-1-m'), \ell+m'\big) - \tilde{v}\big(i, k+(j-1-m'), \ell+1+m'\big) \nonumber\\
        & = \tilde{D}\big(i, k+1+(j-1-m'), \ell+m'\big). \label{eq:large_mu0_ub}
    \end{align}
    Note that for all $m' \in \{0,1, \ldots, j-1\}$, so the right-hand side of \eqref{eq:large_mu0_ub} eventually becomes negative as $i$ increases (by Proposition 3.5 when $\mu_1 \geq \mu_2$, or Proposition 3.7 when $\mu_2 > \mu_1$ in \citep{lu2026balancingindependentcollaborativeservice}). It follows that $G(i,j,k,\ell) < 0$ for all sufficiently large $i$.

    To prove Statement \ref{state:large_mu0_neg}, suppose $i$ is such that $G(i,j,k,\ell) < 0$. We establish the following identity:
    \begin{align}
        G(i,j,k,\ell) = \tilde{v}(i, k+j, \ell) - \tilde{v}(i, k+j-1, \ell+1) = \tilde{D}(i,k+j,\ell). \label{eq:large_mu0_D_general}
    \end{align}
    To this end, we first show that the first minimum in \eqref{eq:diff_of_min} is achieved at $m = 0$, yielding the value $\tilde{v}(i,k+j,\ell)$. Suppose instead that the minimum occurs at some $m_1 = \{ 1,2, \ldots, j-1\}$, so that $\tilde{v}\big(i, k+1+(j-1-m_1), \ell+m_1\big) \leq \tilde{v}(i,k+j,\ell)$. 
    Notice that the same term $\tilde{v}\big(i, k+1+(j-1-m_1), \ell+m_1\big)$ serves as an upper bound for the second minimum in \eqref{eq:diff_of_min}, leading to
    \begin{align*}
    G(i,j,k,\ell)\geq \tilde{v}\big(i, k+1+(j-1-m_1), \ell+m_1\big) - \tilde{v}\big(i, k+1+(j-1-m_1), \ell+m_1\big) \geq 0,
    \end{align*}contradicting the assumption that $ G(i,j,k,\ell) < 0$. 
    Therefore, the first minimum must be attained at $m = 0$, with value $\tilde{v}(i,k+j,\ell)$. In particular, this implies $\tilde{v}(i,k+j,\ell) \leq \tilde{v}(i,k+j-1,\ell+1)$, or equivalently, $\tilde{D}(i,k+j,\ell) \leq 0$. Then, by Proposition 3.11 in \citep{lu2026balancingindependentcollaborativeservice}, it follows that $\tilde{D}\big(i,k+1+(j-1-m_2),\ell + m_2\big) \leq 0$ for all $m_2 = 1,2,\ldots, j-1$. Hence, the second minimum in \eqref{eq:diff_of_min} is also attained at $m = 0$, i.e., the minimum value is $\tilde{v}(i,k+j-1,\ell+1)$, verifying \eqref{eq:large_mu0_D_general}.

    Suppose now $\tilde{D}(i,k+j,\ell) < 0$, which implies $\tilde{v}(i,k+j,\ell) < \tilde{v}(i,k+j-1,\ell+1)$. Again by Proposition 3.11 in \citep{lu2026balancingindependentcollaborativeservice}, this yields $\tilde{D}\big(i,k+1+(j-1-m_2),\ell + m_2\big) < 0$ for all $m_2 = 1,2,\ldots, j-1$. Consequently, the first and second terms in \eqref{eq:diff_of_min} take the minimum values $\tilde{v}(i,k+j,\ell)$ and $\tilde{v}(i,k+j-1,\ell+1)$, respectively, resulting in
    \begin{align*}
        G(i,j,k,\ell) = \tilde{v}(i,k+j,\ell) - \tilde{v}(i,k+j-1,\ell+1) = \tilde{D}(i,k+j,\ell),
    \end{align*}as desired.

    Finally, the proof of Statement \ref{state:large_mu0_pos} follows by a symmetric argument to that of Statement \ref{state:large_mu0_neg} and is omitted for brevity.
\end{proof} 

\endgroup

\end{document}